\documentclass{article}
\usepackage{amsmath,amsthm,amssymb}
\usepackage[usenames,dvipsnames]{xcolor}
\usepackage{enumerate}
\usepackage{graphicx}
\usepackage{cite}
\usepackage{comment}
\usepackage{oands}
\usepackage{tikz}
\usepackage{changepage}
\usepackage{bbm}
\usepackage{mathtools}
\usepackage[margin=1in]{geometry}
\usepackage[pagewise,mathlines]{lineno}
\usepackage{appendix}
\usepackage{stmaryrd} 
\usepackage{multicol}
\usepackage{microtype}

\usepackage[pdftitle={Convergence of the free Boltzmann quadrangulation with simple boundary to the Brownian disk},
  pdfauthor={Ewain Gwynne and Jason Miller},
colorlinks=true,linkcolor=NavyBlue,urlcolor=RoyalBlue,citecolor=PineGreen,bookmarks=true,bookmarksopen=true,bookmarksopenlevel=2,unicode=true,linktocpage]{hyperref}

\theoremstyle{plain}
\newtheorem{thm}{Theorem}[section]

\newtheorem{lem}[thm]{Lemma}
\newtheorem{prop}[thm]{Proposition}

\def\@rst #1 #2other{#1}
\newcommand\MR[1]{\relax\ifhmode\unskip\spacefactor3000 \space\fi
  \MRhref{\expandafter\@rst #1 other}{#1}}
\newcommand{\MRhref}[2]{\href{http://www.ams.org/mathscinet-getitem?mr=#1}{MR#2}}

\theoremstyle{definition}
\newtheorem{defn}[thm]{Definition}
\newtheorem{remark}[thm]{Remark}

\numberwithin{equation}{section}

\newcommand{\dsb}{\begin{adjustwidth}{2.5em}{0pt}
\begin{footnotesize}}
\newcommand{\dse}{\end{footnotesize}
\end{adjustwidth}}

\newcommand{\ssb}{\begin{adjustwidth}{2.5em}{0pt}}
\newcommand{\sse}{\end{adjustwidth}}

\newcommand{\aryb}{\begin{eqnarray*}}
\newcommand{\arye}{\end{eqnarray*}}
\def\alb#1\ale{\begin{align*}#1\end{align*}}
\def\allb#1\alle{\begin{align}#1\end{align}}
\newcommand{\eqb}{\begin{equation}}
\newcommand{\eqe}{\end{equation}}
\newcommand{\eqbn}{\begin{equation*}}
\newcommand{\eqen}{\end{equation*}}

\newcommand{\BB}{\mathbbm}
\newcommand{\ol}{\overline}

\newcommand{\op}{\operatorname}

\newcommand{\frk}{\mathfrak}
\newcommand{\eqD}{\overset{d}{=}}
\newcommand{\ep}{\epsilon}
\newcommand{\rta}{\rightarrow}

\newcommand{\wt}{\widetilde}
\newcommand{\wh}{\widehat} 
\newcommand{\mcl}{\mathcal}

\newcommand{\bdy}{\partial}

\newcommand{\srta}{\shortrightarrow}
\newcommand{\el}{l}

\newcommand{\glu}{{\operatorname{Glue}}}

\let\originalleft\left
\let\originalright\right
\renewcommand{\left}{\mathopen{}\mathclose\bgroup\originalleft}
\renewcommand{\right}{\aftergroup\egroup\originalright}

\title{Convergence of the free Boltzmann quadrangulation with\\ simple boundary to the Brownian disk}
\date{  }
\author{
\begin{tabular}{c} Ewain Gwynne\\[-5pt]\small MIT \end{tabular}
\begin{tabular}{c} Jason Miller\\[-5pt]\small Cambridge \end{tabular}
}

\begin{document}

\maketitle

\begin{abstract}
We prove that the free Boltzmann quadrangulation with simple boundary and fixed perimeter, equipped with its graph metric, natural area measure, and the path which traces its boundary converges in the scaling limit to the free Boltzmann Brownian disk. The topology of convergence is the so-called Gromov-Hausdorff-Prokhorov-uniform (GHPU) topology, the natural analog of the Gromov-Hausdorff topology for curve-decorated metric measure spaces.  From this we deduce that a random quadrangulation of the sphere decorated by a $2l$-step self-avoiding loop converges in law in the GHPU topology to the random curve-decorated metric measure space obtained by gluing together two independent Brownian disks along their boundaries. 
\end{abstract}
 

\tableofcontents

\section{Introduction}
\label{sec-intro}

\subsection{Overview}
\label{sec-overview}

A \emph{planar map} is a connected graph embedded in the sphere with two such maps declared to be equivalent if they differ by an orientation-preserving homeomorphism of the sphere. Random planar maps are a natural model of discrete random surfaces. In recent years, it has been shown that there exist continuum random surfaces, i.e.\ random metric spaces, called \emph{Brownian surfaces} which arise as the scaling limits of uniform random planar maps of various types in the Gromov-Hausdorff topology. The convergence of uniform random planar maps toward Brownian surfaces is expected to be universal in the sense that different uniform random planar map models (e.g., triangulations, quadrangulations, general maps) with the same topology all converge in the scaling limit to the same Brownian surface. 

The best-known Brownian surface is the \emph{Brownian map}, which has the topology of the sphere and has been shown to be the scaling limit of a number of different uniform random planar map models on the sphere in~\cite{miermont-brownian-map,legall-uniqueness,ab-simple,bjm-uniform,abraham-bipartite,beltran-legall-pendant}.  In this paper, we will primarily be interested in the \emph{Brownian disk}~\cite{bet-mier-disk}, which is the scaling limit of uniform random planar maps with the topology of the disk.  Other Brownian surfaces include the \emph{Brownian plane}, which is the scaling limit of the uniform infinite planar quadrangulation~\cite{curien-legall-plane}; and the \emph{Brownian half-plane}, which is the scaling limit of the uniform infinite planar quadrangulation with general or simple boundary~\cite{gwynne-miller-uihpq,bmr-uihpq}.

A \emph{quadrangulation with boundary} is a random planar map $Q$ whose faces all have degree four except for one special face, called the \emph{external face}, whose degree is allowed to be arbitrary. The \emph{boundary} $\bdy Q$ of $Q$ is the border of the external face and the \emph{perimeter} of $Q$ is the degree of the external face.  One can consider both quadrangulations with general boundary, where the boundary is allowed to have multiple edges and multiple vertices; and quadrangulations with simple boundary, where the boundary is constrained to be simple. 

It is shown in~\cite{bet-mier-disk} that the Brownian disk is the scaling limit of uniform random quadrangulations with general boundary as both the perimeter and number of internal faces properly rescaled converge to given positive values.  However,~\cite{bet-mier-disk} does not treat the case of quadrangulations with simple boundary since their proof is based on a variant of the Schaeffer bijection~\cite{schaeffer-bijection,bdg-bijection} for quadrangulations with boundary which does not behave nicely if one conditions the boundary to be simple.  In~\cite[Section~8.1]{bet-disk-tight}, it is left as an open problem to show that the Brownian disk is also the scaling limit of uniformly sampled quadrangulations with simple boundary. 

The main purpose of this paper is to prove a variant of this statement for quadrangulations with simple boundary having fixed perimeter, but not fixed area.  In particular, we will consider the \emph{free Boltzmann distribution} on quadrangulations with simple boundary with fixed perimeter (defined precisely in Definition~\ref{def-fb} below) and show that a quadrangulation sampled from this distribution converges in law in the scaling limit to the \emph{free Boltzmann Brownian disk}, a variant of the Brownian disk with fixed boundary length, but random area. 

We will prove this scaling limit result in a stronger topology than the Gromov-Hausdorff topology.  Namely, we will show that a free Boltzmann quadrangulation with simple boundary equipped with its natural area measure and the path which traces its boundary converges to a free Boltzmann Brownian disk equipped with its natural area measure and boundary path in the Gromov-Hausdorff-Prokhorov-uniform (GHPU) topology, the natural analog of the Gromov-Hausdorff topology for curve-decorated metric measure spaces~\cite{gwynne-miller-uihpq}.  

Our scaling limit result for free Boltzmann quadrangulations with simple boundary will be deduced from another theorem which says that a uniform random quadrangulation with general boundary and its \emph{simple-boundary core}, which is the largest sub-graph which is itself a quadrangulation with simple boundary, converge jointly in law in the scaling limit to two copies of the same Brownian disk in the GHPU topology.  
  
Free Boltzmann quadrangulations with simple boundary are particularly natural since these quadrangulations arise as the bubbles disconnected from $\infty$ when one performs the peeling procedure on a uniform infinite planar quadrangulation either with no boundary or with simple boundary. More precisely, if we reveal the face incident to the root edge, then the bounded complementary connected components of this face are free Boltzmann quadrangulations with simple boundary.  One also gets free Boltzmann quadrangulations with simple boundary from the peeling procedure on a free Boltzmann quadrangulation with simple boundary, which gives these quadrangulations a natural Markov property. Peeling was first introduced in the physics literature by Watabiki~\cite{watabiki-lqg}, was first studied rigorously in~\cite{angel-peeling}, and was developed further in several works including~\cite{benjamini-curien-uipq-walk,curien-legall-peeling,curien-miermont-uihpq,angel-uihpq-perc,angel-curien-uihpq-perc}. See Section~\ref{sec-peeling} for more on the peeling procedure.

In the course of proving our main results, we obtain several results of independent interest concerning free Boltzmann quadrangulations with simple boundary and the uniform infinite half-plane quadrangulation (UIHPQ$_{\op{S}}$) which is their infinite-boundary length limit.  We prove half-plane analogs of some of the results in~\cite{curien-legall-peeling} for peeling processes on the uniform infinite planar quadrangulation as well as estimates which enable us to compare the local behavior of a free Boltzmann quadrangulation with simple boundary and the UIHPQ$_{\op{S}}$ (see in particular Lemma~\ref{lem-bubble-cond}, Lemma~\ref{lem-full-bubble-cond}, and Proposition~\ref{prop-map-coupling}). 

The results of this paper enable us to prove that various curve-decorated random quadrangulations converge in the scaling limit to $\sqrt{8/3}$-Liouville quantum gravity surfaces decorated by independent SLE$_{8/3}$ or SLE$_6$~\cite{schramm0} curves. For $\gamma \in (0,2)$, a \emph{$\gamma$-Liouville quantum gravity} (LQG) surface is, heuristically speaking, the random Riemann surface parameterized by a domain $D\subset \BB C$ whose Riemannian metric tensor is $e^{\gamma h} \, dx\otimes dy$, where $h$ is some variant of the Gaussian free field (GFF) on $D$ (see~\cite{shef-kpz,shef-zipper,wedges} for more on $\gamma$-LQG surfaces).  This does not make rigorous sense since $h$ is a distribution, not a function, so does not take values at points. However, it is shown in~\cite{shef-kpz} that a $\sqrt{8/3}$-LQG surface admits a natural measure and in~\cite{lqg-tbm1,lqg-tbm2,lqg-tbm3}, building on \cite{qle,tbm-characterization,sphere-constructions}, that in the special case when $\gamma =\sqrt{8/3}$, a $\sqrt{8/3}$-LQG surface admits a natural metric.

Certain special $\sqrt{8/3}$-LQG surfaces are equivalent as metric measure spaces to Brownian surfaces. In particular, the Brownian disk (resp.\ half-plane, map) is equivalent to the quantum disk (resp.\ $\sqrt{8/3}$-quantum wedge, quantum sphere).  Moreover, it is shown in~\cite{lqg-tbm3} that the conformal structure of a $\sqrt{8/3}$-LQG surface (represented by the distribution $h$) is a.s.\ determined by its metric measure space structure.  This gives us a canonical embedding of a Brownian surface into a domain in $\BB C$. 

Quadrangulations with simple boundary can be glued together along their boundaries to obtain uniform random quadrangulations decorated by some form of a self-avoiding walk (SAW); see~\cite{bg-simple-quad,bet-mier-disk} for the case of quadrangulations with finite simple boundary case and~\cite{caraceni-curien-saw} for the case of the UIHPQ$_{\op{S}}$.  It is shown in~\cite{gwynne-miller-saw} that the uniform infinite SAW-decorated quadrangulations obtained by gluing together UIHPQ$_{\op{S}}$'s along their boundaries converge in the scaling limit in the local GHPU topology to the analogous continuum curve-decorated metric measure spaces obtained by gluing together copies of the Brownian half-plane along their boundaries. These limiting spaces can be identified with $\sqrt{8/3}$-Liouville quantum gravity surfaces decorated by SLE$_{8/3}$-type curves using the results of~\cite{gwynne-miller-gluing}.

As a consequence of our scaling limit results for free Boltzmann quadrangulations with simple boundary and the results of~\cite{gwynne-miller-saw}, we obtain finite-boundary analogs of the results of~\cite{gwynne-miller-saw}. In particular, we prove that a random quadrangulation of the sphere decorated by a self-avoiding loop of length $2\el$, which can be obtained by identifying the boundaries of two independent free Boltzmann quadrangulations with simple boundary, converges in the scaling limit in law as $\el\rta\infty$ with respect to the GHPU topology to a pair of independent Brownian disks glued together along their boundaries. Due to local absolute continuity between the Brownian disk and the Brownian half-plane and the results of~\cite{gwynne-miller-gluing}, this latter metric measure space locally looks like a $\sqrt{8/3}$-Liouville quantum gravity surface decorated by an independent SLE$_{8/3}$-type loop.

The results of this paper will also be used in~\cite{gwynne-miller-perc} to prove scaling limit results for a free Boltzmann quadrangulation with simple boundary (resp.\ the UIHPQ$_{\op{S}}$) decorated by a critical ($p=3/4$~\cite{angel-curien-uihpq-perc}) face percolation exploration path toward the Brownian disk (resp.\ Brownian half-plane), equivalently the quantum disk (resp.\ $\sqrt{8/3}$-quantum wedge) decorated by an independent chordal SLE$_6$~\cite{schramm0}.  It is shown in~\cite{gwynne-miller-char} that the law of a chordal SLE$_6$ on a quantum disk with fixed area is uniquely characterized by its equivalence class as a curve-decorated topological measure space plus the fact that the internal metric spaces corresponding to the complementary connected components of the curve stopped at each time $t\geq0$ are independent free Boltzmann Brownian disks.  The scaling limit results proven in the present paper allows us to check these conditions for a subsequential scaling limit of face percolation on a free Boltzmann quadrangulation with simple boundary.

\bigskip

\noindent{\bf Acknowledgements} J.M.\ thanks Institut Henri Poincar\'e for support as a holder of the Poincar\'e chair, during which part of this work was completed. We thank an anonymous referee for helpful comments on an earlier version of this article.

\subsection{Preliminary definitions} 
\label{sec-intro-def}

In this subsection we give precise definitions of the objects involved in the statements of our main results. 

\subsubsection{Quadrangulations with boundary}
\label{sec-intro-def-quad}

Here we state several definitions for quadrangulations; see Figure~\ref{fig-general-quad} for an illustration.

\begin{figure}[ht!]
 \begin{center}
\includegraphics[scale=1]{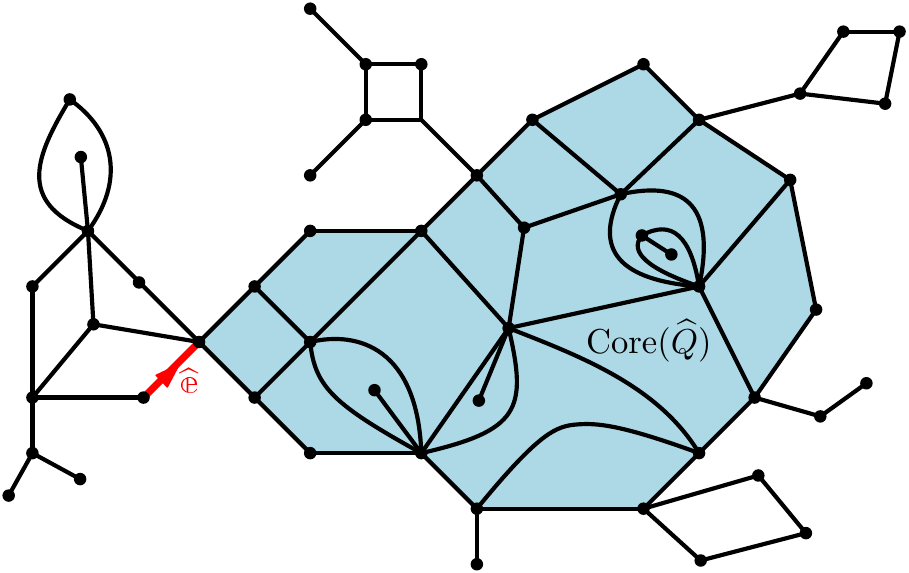} 
\caption{A rooted quadrangulation $(\wh Q , \wh{\BB e})$ with general boundary and its simple-boundary core $\op{Core}(\wh Q)$ (light blue). Note that the root edge $\wh{\BB e}$ can be in $\bdy \op{Core}(\wh Q)$ or $\bdy\wh Q \setminus \bdy\op{Core}(\wh Q)$ (as shown in the figure). }\label{fig-general-quad}
\end{center}
\end{figure} 
 
A \emph{quadrangulation with (general) boundary} is a (finite or infinite) planar map~$\wh Q$ with a distinguished face~$f_\infty$, called the \emph{exterior face}, such that every face of~$\wh Q$ other than~$f_\infty$ has degree~$4$. The \emph{boundary} of $\wh Q$, denoted by $\bdy \wh Q$, is the smallest subgraph of~$\wh Q$ which contains every edge of $\wh Q$ incident to $f_\infty$. The \emph{perimeter} $\op{Perim}(\wh Q) $ of~$\wh Q$ is defined to be the degree of the exterior face, with edges counted according to multiplicity. 
We note that the perimeter of a quadrangulation with boundary is always even.

For $n\in\BB N_0$ and $\el \in \BB N$, we write $\wh{\mcl Q}^\srta(n,\el)$ for the set of pairs $(\wh Q , \wh{\BB e})$ where $\wh Q$ is a quadrangulation with boundary having $n$ interior faces and perimeter $2\el$ and $\wh{\BB e}$ is an oriented edge of $\bdy \wh Q$, called the \emph{root edge}.  

We say that $\wh Q$ has \emph{simple boundary} if $\bdy \wh Q$ is a simple path, i.e.\ it has no vertices or edges of multiplicity bigger than~$1$. 
We will typically denote quadrangulations with general boundary with a hat and quadrangulations with simple boundary without a hat. 

For $n \in \BB N_0$ and $\el \in \BB N$, we write $\mcl Q_{ \op{S} }^\shortrightarrow(n,\el)$ for the set of pairs $(Q , \BB e)$ where $Q$ is a quadrangulation with simple boundary having $2\el$ boundary edges and $n$ interior edges and $\BB e$ is an oriented edge in~$\bdy Q$, called the \emph{root edge}.  By convention, we consider the trivial quadrangulation with one edge and no interior faces to be a quadrangulation with simple boundary of perimeter 2 and define $\mcl Q_{\op{S}}^\srta(0,1)$ to be the set consisting of this single quadrangulation, rooted at its unique edge. We define $\mcl Q_{\op{S}}^\srta(0,\el) = \emptyset$ for $\el \geq 2$.

For a quadrangulation $\wh Q$ with general boundary, we define its \emph{simple-boundary components} to be the maximal sub-quadrangulations of $\wh Q$ having at least one interior face whose boundary is simple and is a sub-graph of $\bdy \wh Q$. Equivalently, the simple-boundary components of $\wh Q$ are the interior faces of the planar map $\bdy \wh Q$. We define the \emph{simple-boundary core} $\op{Core}(\wh Q)$ of $\wh Q$ to be the simple-boundary component of $\wh Q$ with the largest boundary length, with ties broken by some arbitrary deterministic convention. 

A \emph{boundary path} of a quadrangulation $\wh Q$ with simple or general boundary is a path $\wh\beta$ from $[0, \op{Perim}(\wh Q) ]_{\BB Z}$ (if $\bdy \wh Q$ is finite) or $\BB Z$ (if $\bdy \wh Q$ is infinite) to $\mcl E(\bdy \wh Q)$ which traces the edges of $\bdy \wh Q$ (counted with multiplicity) in cyclic order. Choosing a boundary path is equivalent to choosing an oriented root edge on the boundary. This root edge is $\wh \beta(0)$, oriented toward $\wh\beta(1)$. 

We define the \emph{free Boltzmann partition function} by  
\eqb \label{eqn-fb-partition}
\frk Z(2 \el )   :=    \frac{8^\el(3 \el-4)!}{(\el-2)! (2 \el)!}, \qquad \frk Z(2 \el+1) = 0 ,\qquad \forall \el \in \BB N ,
\eqe 
where here we set $(-1)! = 1$. 

\begin{defn} \label{def-fb}
For $\el\in\BB N$, the \emph{free Boltzmann distribution} on quadrangulations with simple boundary and perimeter $2\el$ is the probability measure on $\bigcup_{n=0}^\infty \mcl Q_{\op{S}}^\shortrightarrow (n ,  \el)$ which assigns to each element of $ \mcl Q_{\op{S}}^\shortrightarrow (n, \el)$ a probability equal to $12^{-n} \frk Z(2\el)^{-1}$.  
\end{defn}

It is shown in~\cite{bg-simple-quad} that $\frk Z(2\el) = \sum_{n=0}^\infty 12^{-n} \# \mcl Q_{\op{S}}^\shortrightarrow(n, \el) $.

The \emph{uniform infinite half-plane quadrangulation with simple (resp.\ general) boundary}, abbreviated UIHPQ$_{\op{S}}$ (resp.\ UIHPQ) is the infinite rooted quadrangulation $(Q^\infty ,\BB e^\infty)$ (resp.\ $(\wh Q^\infty, \wh{\BB e}^\infty)$) with infinite simple (resp.\ general) boundary which is the Benjamini-Schramm local limit~\cite{benjamini-schramm-topology} in law based at the root edge of a uniform sample from $\wh{\mcl Q}^\srta(n,\el)$ (resp.\ $\mcl Q^\srta(n,\el)$) as $n$, then $\el$, tends to $\infty$~\cite{curien-miermont-uihpq,caraceni-curien-uihpq}.  

When we refer to a free Boltzmann quadrangulation with perimeter $2\el = \infty$, we mean the UIHPQ$_{\op{S}}$. 
 
\subsubsection{The Brownian disk}
\label{sec-intro-def-disk}

For $\frk a , \frk l >0$, the \emph{Brownian disk with area $\frk a$ and perimeter $\frk l$} is the random curve-decorated metric measure space $(H , d , \mu , \xi)$ with the topology of the disk which arises as the scaling limit of uniform random quadrangulations with boundary (see~\cite{bet-mier-disk} for the case of uniform quadrangulations with general boundary). The Brownian disk can be constructed as a metric space quotient of $[0,\frk a]$ via a continuum analog of the Schaeffer bijection~\cite{bet-mier-disk}, using a Brownian motion $X$ conditioned to first hit $-\frk l$ at time $\frk a$ and a ``label process" $Z$ on the continuum random forest constructed from the excursions of $X$ above its running minimum. We will not need this construction here so we will not review it carefully; see~\cite[Section 2]{bet-mier-disk} for the precise definition. The area measure~$\mu$ is the pushforward of Lebesgue measure on $[0,\frk a]$ under the quotient map. The path $\xi : [0,\frk l] \rta \bdy H$, called the \emph{boundary path}, parameterizes $\bdy H$ according to its natural length measure. Precisely, $\xi(r)$ for $r\in [0,\ell]$ is the image under the quotient map of the first time at which the encoding function $X$ hits $-r$.

Following~\cite[Section 1.5]{bet-mier-disk}, we define the \emph{free Boltzmann Brownian disk with perimeter $\frk l$} to be the random curve-decorated metric measure space $(H , d , \mu , \xi)$ obtained as follows: first sample a random area $\frk a$ from the probability measure $  \frac{\frk l^3}{ \sqrt{2\pi a^5 } } e^{-\frac{\frk l^2}{2 a} } \BB 1_{(a\geq 0)} \, da$, then sample a Brownian disk with boundary length $\frk l$ and area $\frk a$. Note that the law of the area of the free Boltzmann Brownian disk with perimeter $\frk l$ can be obtained by scaling the law of the area of the free Boltzmann Brownian disk with perimeter $1$ by $\frk l^2$. 
Consequently, it follows from~\cite[Remark 3]{bet-mier-disk} that the free Boltzmann Brownian disk with perimeter $\frk l$ can be obtained from the free Boltzmann Brownian disk with perimeter 1 by scaling areas by $\frk l^2$, boundary lengths by $\frk l$, and distances by $\frk l^{1/2}$.

\subsubsection{The Gromov-Hausdorff-Prokhorov-uniform metric}
\label{sec-ghpu}

In this subsection we will review the definition of the Gromov-Hausdorff-Prokhorov-uniform (GHPU) metric from~\cite{gwynne-miller-uihpq}, 
which is the metric with respect to which our scaling limit results hold.  

For a metric space $(X,d)$, we let $C_0(\BB R , X)$ be the space of continuous curves $\eta : \BB R\rta X$ which are ``constant at $\infty$," i.e.\ $\eta$ extends continuously to the extended real line $[-\infty,\infty]$. 
Each curve $\eta : [a,b] \rta X$ can be viewed as an element of $C_0(\BB R ,X)$ by defining $\eta(t) = \eta(a)$ for $t < a$ and $\eta(t) = \eta(b)$ for $t> b$. 
\begin{itemize}
\item Let $\BB d_d^{\op{H}}$ be the $d$-Hausdorff metric on compact subsets of $X$.
\item Let $\BB d_d^{\op{P}}$ be the $d$-Prokhorov metric on finite measures on $X$.
\item Let $\BB d_d^{\op{U}}$ be the $d$-uniform metric on $C_0(\BB R , X)$.
\end{itemize}

Let $\BB M^{\op{GHPU}}$ be the set of $4$-tuples $\frk X  = (X , d , \mu , \eta )$ where $(X,d)$ is a compact metric space, $\mu$ is a finite Borel measure on $X$, and $\eta \in C_0(\BB R,X)$. 

Given elements $\frk X^1 = (X^1 , d^1, \mu^1 , \eta^1) $ and $\frk X^2 =  (X^2, d^2,\mu^2,\eta^2) $ of $ \BB M^{\op{GHPU}}$, a compact metric space $(W, D)$, and isometric embeddings $\iota^1 : X^1\rta W$ and $\iota^2 : X^2\rta W$, we define their \emph{GHPU distortion} by 
\begin{align}
\label{eqn-ghpu-var}
\op{Dis}_{\frk X^1,\frk X^2}^{\op{GHPU}}\left(W,D , \iota^1, \iota^2 \right)   
:=  \BB d^{\op{H}}_D \left(\iota^1(X^1) , \iota^2(X^2) \right) +   
\BB d^{\op{P}}_D \left(( (\iota^1)_*\mu^1 ,(\iota^2)_*\mu^2) \right) + 
 \BB d_D^{\op{U}}\left( \iota^1 \circ \eta^1 , \iota^2 \circ\eta^2 \right) .
\end{align}
We define the \emph{Gromov-Hausdorff-Prokhorov-uniform (GHPU) distance} by
\begin{align} \label{eqn-ghpu-def}
 \BB d^{\op{GHPU}}\left( \frk X^1 , \frk X^2 \right) 
 = \inf_{(W, D) , \iota^1,\iota^2}  \op{Dis}_{\frk X^1,\frk X^2}^{\op{GHPU}}\left(W,D , \iota^1, \iota^2 \right)      ,
\end{align}
where the infimum is over all compact metric spaces $(W,D)$ and isometric embeddings $\iota^1 : X^1 \rta W$ and $\iota^2 : X^2\rta W$.  It is shown in~\cite[Proposition~1.3]{gwynne-miller-uihpq} that $\BB d^{\op{GHPU}}$ is a complete separable metric on~$\BB M^{\op{GHPU}}$ provided we identify two elements of~$\BB M^{\op{GHPU}}$ which differ by a measure- and curve- preserving isometry.  

There is also a local variant of the GHPU metric for locally compact curve-decorated length spaces equipped with a locally finite Borel measure, which is obtained from the GHPU metric by restricting to metric balls centered at $\eta(0)$, then integrating over all the possible radii of these balls; see~\cite{gwynne-miller-uihpq} for more details.

\begin{remark}[Graphs as elements of $\BB M^{\op{GHPU}}$] \label{remark-ghpu-graph}
In this paper we will often be interested in a graph $G$ equipped with its graph distance $d_G$. 
In order to study continuous curves in $G$, we identify each edge of $G$ with a copy of the unit interval $[0,1]$ and extend the graph metric on $G$ by requiring that this identification is an isometry. 
If $\lambda$ is a path from some discrete interval $[a,b]_{\BB Z}$ into $\mcl E(G)$, we extend $\lambda$ from $[a,b]_{\BB Z}$ to $[a-1,b] $ by linear interpolation. 
If $G$ is a finite graph and we are given a finite measure $\mu$ on vertices of $G$ and a curve $\lambda$ in $G$ and we view $G$ as a connected metric space and $\lambda $ as a continuous curve as above, then $(G , d_G , \mu , \lambda )$ is an element of~$\BB M^{\op{GHPU} }$.   
\end{remark}

\subsection{Main results}
\label{sec-results}

\subsubsection{Joint convergence of a quadrangulation with general boundary and its simple-boundary core}
\label{sec-core-ghpu}

Our first main result shows that a uniform random quadrangulation with general boundary and its simple-boundary core converge jointly in the scaling limit in the GHPU topology to the same Brownian disk; note that the boundary paths are scaled differently. 

Let $\frk a , \frk l  >0$ and let $\{(a^n , \el^n) \}_{n\in\BB N}$ be a sequence of pairs of positive integers such that $(2n)^{-1} a^n \rta \frk a$ and $(2n)^{-1/2} \el^n \rta \frk l$ as $n\rta\infty$. 
For $n\in\BB N$, let $(\wh Q^n , \wh{\BB e}^n)$ be sampled uniformly from $\wh{\mcl Q}^\srta(a^n , \el^n)$ and view $\wh Q^n$ as a connected metric space by identifying each edge with an isometric copy of the unit interval as in Remark~\ref{remark-ghpu-graph}. 

For $n\in\BB N$, 
let $\wh d^n$ be the graph metric on $\wh Q^n$, rescaled by $(9/8)^{1/4} n^{-1/4}$.  
Let $\wh\mu^n$ be the measure on $\mcl V(\wh Q^n)$ which assigns a mass to each vertex equal to $(4n)^{-1}$ times its degree. 
Let $\wh\beta^n : [0, 2\el^n]  \rta \bdy \wh Q^n$ be the boundary path of $\wh Q^n$ started from $\wh{\BB e}^n$ and extended by linear interpolation and let $\wh\xi^n(t) := \wh\beta^n\left(2^{3/2} n^{1/2} t \right)$ for $t\in [0, (2n)^{-1/2} \el^n] $. 

Also let $Q^n := \op{Core}(\wh Q^n)$ be the simple-boundary core of $\wh Q^n$, let $d^n$ be the graph metric on $Q^n$ rescaled by $(9/8)^{1/4} n^{-1/4}$, so that $ d^n := \wh d^n|_{Q^n}$. Let $ \mu^n $ be the measure on $Q^n$ which assigns to each vertex a mass equal to $(4n)^{-1}$ times its degree, and note that $\mu^n$ coincides with $\wh\mu^n$ on $Q^n\setminus \bdy Q^n$. 
Let $\beta^n : [0, \#\mcl E(\bdy Q^n)]  \rta \bdy Q^n$ be the boundary path of $Q^n$ started from the first edge of $\bdy Q^n$ hit by $\beta^n$ and extended by linear interpolation; and let $\xi^n(t) := \beta^n\left( \frac{2^{3/2}}{3} n^{1/2} t \right)$ for $t\in [0, (3/2^{3/2}) n^{-1/2} \#\mcl E(\bdy Q^n)]$. (The reason for the extra factor of 3 in the scaling for $\bdy Q^n$ as compared to $\bdy\wh Q^n$ is that, as we will see, typically only about $1/3$ of the edges of $\bdy\wh Q^n$ are part of $\bdy Q^n$.) 
 
For $n\in\BB N$, define the curve-decorated metric measure spaces
\eqb \label{eqn-core-ghpu}
\wh{\frk Q}^n := \left(\wh Q^n , \wh d^n , \wh\mu^n, \wh\xi^n \right)  \quad \op{and} \quad  \frk Q^n := \left(Q^n , d^n , \mu^n, \xi^n \right)   .
\eqe  

Also let $\frk H = (H , d , \mu , \xi)$ be a Brownian disk with area $\frk a$ and boundary length $\frk l$, equipped with its natural metric, area measure, and boundary path.

\begin{thm} \label{thm-core-ghpu}
We have the joint convergence $(\wh{\frk Q}^n , \frk Q^n) \rta (\frk H , \frk H)$ in law with respect to the Gromov-Hausdorff-Prokhorov-uniform topology as $n\rta\infty$.
\end{thm}

Since we already know $\wh{\frk Q}^n \rta \frk H$ in law in the GHPU topology~\cite[Theorem~4.1]{gwynne-miller-uihpq}, 
the main difficulty in the proof of Theorem~\ref{thm-core-ghpu} is showing the uniform convergence of the rescaled boundary path $\xi^n$ of $\frk Q^n$; indeed, the ``Hausdorff" and ``Prokhorov" parts of the Gromov-Hausdorff-Prokhorov-uniform convergence are easy consequences of~\cite[Theorem~1]{bet-mier-disk} and the fact that the boundary of the Brownian disk is simple. The convergence of the boundary path will be deduced from the explicit law of the ``dangling quadrangulations" of the UIHPQ with general boundary (see Section~\ref{sec-pruning}) and a local absolute continuity argument.

\subsubsection{Scaling limit of free Boltzmann quadrangulations with simple boundary}
\label{sec-fb-ghpu}
 
Theorem~\ref{thm-core-ghpu} implies in particular that certain quadrangulations with simple boundary having \emph{random} area and perimeter converge to the Brownian disk in the GHPU topology. 
Our second main result shows that one also has convergence of random quadrangulations with simple boundary having fixed perimeter, but random area, toward the Brownian disk. 
  
For $\el\in\BB N$, let $( Q^\el , \BB e^\el)$ be a free Boltzmann quadrangulation with simple boundary of perimeter $2\el$ (Definition~\ref{def-fb}) and view $ Q^\el$ as a connected metric space by identifying each edge with an isometric copy of the unit interval as in Remark~\ref{remark-ghpu-graph}. 

For $\el \in \BB N$, 
let $d^\el$ be the graph metric on $Q^\el$, rescaled by $(2\el )^{-1/2}$.  
Let $\mu^\el$ be the measure on $\mcl V(Q^\el)$ which assigns mass to each vertex equal to $18^{-1} \el^{-2}$ times its degree. 
Let $\beta^\el : [0, 2\el ]  \rta \bdy Q^\el$ be the boundary path of $Q^\el$ started from $\BB e^\el$ and extended by linear interpolation and let $\xi^\el(t) := \beta^\el\left(2\el t \right)$ for $t\in [0,  1 ] $. 
Define the curve-decorated metric measure space $\frk Q^\el := \left(Q^\el , d^\el , \mu^\el, \xi^\el \right)$. (Note that the scaling factors are different here than in Theorem~\ref{thm-core-ghpu} since we are fixing the perimeter rather than the area.)

Also let $\frk H = (H , d , \mu , \xi)$ be a free Boltzmann Brownian disk with unit perimeter equipped with its natural metric, area measure, and boundary path.

\begin{thm} \label{thm-fb-ghpu}
We have $\frk Q^\el \rta \frk H$ in law with respect to the Gromov-Hausdorff-Prokhorov-uniform topology as $\el\rta\infty$. 
\end{thm}

Theorem~\ref{thm-fb-ghpu} will be deduced from Theorem~\ref{thm-core-ghpu} by using the peeling procedure to compare a free Boltzmann quadrangulation with simple boundary and fixed perimeter to the core of a free Boltzmann quadrangulation with general boundary. 

\subsubsection{Quadrangulations of the sphere decorated by a self-avoiding loop}
\label{sec-saw-ghpu}

For $\el\in\BB N$, let $(Q_-^\el , \BB e_-^\el)$ and $(Q_+^\el , \BB e_+^\el)$ be independent free Boltzmann quadrangulations with simple boundary of perimeter $2\el$. Let $\beta_\pm^\el$ be their respective boundary paths, started from the root edges. Let $Q_\glu^\el$ be the quadrangulation of the sphere obtained by identifying the edges $\beta_-^\el(k)$ and $\beta_+^\el(k)$ for $k\in [0,2\el]_{\BB Z}$ and let $\beta_\glu^\el : [0,2\el]_{\BB Z}\rta \mcl E(Q_\glu^\el)$ be the path corresponding to $\beta_\pm^\el$ under this identification. 

It is easy to see from Definition~\ref{def-fb} that $(Q_\glu^\el , \beta_\glu^\el(0) ,  \beta_\glu^\el)$ is distributed according to the free Boltzmann measure on triples $(\frk Q , \frk e, \frk b)$ consisting of an edge-rooted quadrangulation of the sphere and a self-avoiding loop of length $2\el$ based at the root edge, i.e.\ the measure which assigns to each such triple a probability proportional to $12^{-\#\mcl F(\frk Q)}$, where $\mcl F(\frk Q)$ is the set of faces of $\frk Q$. In particular, the conditional law of $\beta_\glu^\el$ given $Q_\glu^\el$ and $\beta_\glu^\el(0)$ is uniform over the set of all self-avoiding loops of length $2\el$ on $Q_\glu^\el$ based at $\beta_\glu^\el(0)$.

We now state a scaling limit result for this self-avoiding loop-decorated quadrangulation in the GHPU topology, which is an exact finite-volume analog of~\cite[Theorem~1.2]{gwynne-miller-saw}.
For $\el\in\BB N$, let $d_\glu^\el$ be the graph metric on $Q_\glu^\el$ rescaled by $(2\el )^{-1/2}$, let $\mu_\glu^\el$ be the measure on $Q_\glu^\el$ which assigns to each vertex a mass equal to $18^{-1} \el^{-2}$ times its degree, and let $\xi_\glu^\el(s) := \beta_\glu^\el(2\el s)$ for $s\in [0,1]_{\BB Z}$. Define the curve-decorated metric measure spaces $\frk Q_\glu^\el := (Q_\glu^\el , d_\glu^\el , \mu_\glu^\el , \xi_\glu^\el)$.
 
Let $(H_\pm  , d_\pm   , \mu_\pm , \xi_\pm )$ be a pair of independent free Boltzmann Brownian disks with unit perimeter equipped with their natural metrics, area measures, and boundary paths (each started from the root edge). Let $(H_\glu , d_\glu  )$ be the metric space quotient of $(H_-  ,d_- )$ and $(H_+   , d_+ )$ under the equivalence relation which identifies $\xi_- (t)$ with $\xi_+ (t)$ for each $t\in [0,1]$ (we recall the definition of the quotient metric in Section~\ref{sec-metric-prelim}). Let $\mu_\glu $ be the measure on $H_\glu $ inherited from $\mu_\pm $ and let $\xi_\glu $ be the two-sided path on $\mu_\glu $ corresponding to the image of $\xi_\pm $ under the identification map. Define $\frk H_\glu := (H_\glu,d_\glu,\mu_\glu,\xi_\glu)$.

\begin{thm} \label{thm-saw-conv}
In the notation just above, $\frk Q_\glu^\el \rta \frk H_\glu$ in law with respect to the Gromov-Hausdorff-Prokhorov-uniform topology as $\el\rta\infty$. 
\end{thm}

Theorem~\ref{thm-saw-conv} will be a consequence of Theorem~\ref{thm-fb-ghpu}, the scaling limit results for infinite-volume random SAW-decorated quadrangulations in~\cite{gwynne-miller-saw}, and a local absolute continuity argument. 
Using essentially the same argument used to prove Theorem~\ref{thm-saw-conv}, one can also obtain analogous scaling limit results when one instead glues $Q_\pm^\el$ along a connected boundary arc rather than along their full boundaries; or when one identifies two $\el$-length boundary arcs of a single free Boltzmann quadrangulation with simple boundary of perimeter $2\el$. These statements are finite-volume analogs of~\cite[Theorems 1.1 and 1.3]{gwynne-miller-saw}. For the sake of brevity we do not include precise statements here.

In the infinite-volume case treated in~\cite{gwynne-miller-saw}, the scaling limit of infinite SAW-decorated quadrangulations obtained by gluing together UIHPQ$_{\op{S}}$'s along their boundaries are identified with certain explicit $\sqrt{8/3}$-LQG surfaces decorated by various forms of SLE$_{8/3}$.  Due to the local absolute continuity between the Brownian disk and Brownian half-plane, we see that $\frk H_\glu$ locally looks like a $\sqrt{8/3}$-LQG surface decorated by an independent SLE$_{8/3}$ curve.

\subsection{Notational conventions}
\label{sec-notation-prelim}

In this subsection, we will review some basic notation and definitions which will be used throughout the paper.
 
\subsubsection{Basic notation}
\label{sec-basic-notation}

\noindent
We write $\BB N$ for the set of positive integers and $\BB N_0 = \BB N\cup \{0\}$. 
\vspace{6pt}

\noindent
For $a,b \in \BB R$ with $a<b$ and $r > 0$, we define the discrete intervals $[a,b]_{r \BB Z} := [a, b]\cap (r \BB Z)$ and $(a,b)_{r \BB Z} := (a,b)\cap (r\BB Z)$.
\vspace{6pt} 

\noindent
If $a$ and $b$ are two quantities, we write $a\preceq b$ (resp.\ $a \succeq b$) if there is a constant $C>0$ (independent of the parameters of interest) such that $a \leq C b$ (resp.\ $a \geq C b$). We write $a \asymp b$ if $a\preceq b$ and $a \succeq b$.
\vspace{6pt}

\noindent
If $a$ and $b$ are two quantities depending on a variable $x$, we write $a = O_x(b)$ (resp.\ $a = o_x(b)$) if $a/b$ remains bounded (resp.\ tends to 0) as $x\rta 0$ or as $x\rta\infty$ (the regime we are considering will be clear from the context).  
\vspace{6pt}

\subsubsection{Graphs and maps}
\label{sec-graph-notation}
 
\noindent
For a planar map $G$, we write $\mcl V(G)$, $\mcl E(G)$, and $\mcl F(G)$, respectively, for the set of vertices, edges, and faces of $G$.
\vspace{6pt}

\noindent
By a \emph{path} in $G$, we mean a function $ \lambda : I \rta \mcl E(G)$ for some (possibly infinite) discrete interval $I\subset \BB Z$, with the property that the edges of $\lambda$ can be oriented in such a way that the terminal endpoint of $\lambda(i)$ coincides with the initial endpoint of $\lambda(i+1)$ for each $i\in I$ other than the right endpoint.  
\vspace{6pt}

\noindent
For sets $A_1,A_2$ consisting of vertices and/or edges of~$G$, we write $\op{dist}\left(A_1 , A_2 ; G\right)$ for the graph distance from~$A_1$ to~$A_2$ in~$G$, i.e.\ the minimum of the lengths of paths in $G$ whose initial edge either has an endpoint which is a vertex in $A_1$ or shares an endpoint with an edge in $A_1$; and whose final edge satisfies the same condition with $A_2$ in place of $A_1$.  If $A_1$ and/or $A_2$ is a singleton, we do not include the set brackets. Note that the graph distance from an edge $e$ to a set $A$ is the minimum distance between the endpoints of $e$ and the set $A$.
We write $\op{diam}(G)$ for the maximal graph distance between vertices of $G$. 
\vspace{6pt}

\noindent
For $r>0$, we define the graph metric ball $B_r\left( A_1 ; G\right)$ to be the subgraph of $G$ consisting of all vertices of $G$ whose graph distance from $A_1$ is at most $r$ and all edges of $G$ whose endpoints both lie at graph distance at most $r$ from $A_1$.  If $A_1 = \{x\}$ is a single vertex or edge, we write $B_r\left( \{x\} ; G\right) =  B_r\left( x ; G\right)$.
\vspace{6pt}

\subsubsection{Metric spaces}
\label{sec-metric-prelim}
 
Here we introduce some notation for metric spaces and recall some basic constructions.
Throughout, let $(X,d_X)$ be a metric space. 
\vspace{6pt}

\noindent
For $A\subset X$ we write $\op{diam} (A ; d_X )$ for the supremum of the $d_X$-distance between points in $A$.
\vspace{6pt}

\noindent
For $r>0$, we write $B_r(A;d_X)$ for the set of $x\in X$ with $d_X (x,A) \leq r$. We emphasize that $B_r(A;d_X)$ is closed (this will be convenient when we work with the local GHPU topology). 
If $A = \{y\}$ is a singleton, we write $B_r(\{y\};d_X) = B_r(y;d_X)$.  
\vspace{6pt}
 
\noindent
Let $\sim$ be an equivalence relation on $X$, and let $\ol X = X/\sim$ be the corresponding topological quotient space. For equivalence classes $\ol x , \ol y\in \ol X$, let $\mcl Q(\ol x , \ol y)$ be the set of finite sequences $(x_1 , y_1 ,    \dots , x_n , y_n)$ of elements of~$X$ such that $x_1 \in \ol x$, $y_n \in \ol y$, and $y_i \sim x_{i+1}$ for each $i \in [1,n-1]_{\BB Z}$. Let
\eqb \label{eqn-quotient-def}
\ol d_X (\ol x , \ol y) := \inf_{(x_1 , y_1 ,    \dots , x_n , y_n) \in \mcl Q(\ol x , \ol y)} \sum_{i=1}^n d_X (x_i ,y_i ) .
\eqe  
Then $\ol d_X$ is a pseudometric on $\ol X$ (i.e., it is symmetric and satisfies the triangle inequality), which we call the \emph{quotient pseudometric}.
The quotient pseudometric possesses the following universal property. Suppose $f : (X,d_X) \rta (Y , d_Y)$ is a $1$-Lipschitz map such that $f(x) = f(y)$ whenever $x,y\in X$ with $x\sim y$. Then $f$ factors through the metric quotient to give a 1-Lipschitz map $\ol f : \ol X \rta Y$ such that $\ol f \circ p = f$, where $p : X\rta \ol X$ is the quotient map. To see this, we define $\ol f(\ol x) := f(x)$, where $x$ is any element of the equivalence class $\ol x$ (this is well-defined by our assumption on $f$). To check that $\ol f$ is one-Lipschitz, observe that for any $\ol x , \ol y \in \ol X$ and any $\ep > 0$, we can find $(x_1 , y_1 ,    \dots , x_n , y_n) \in \mcl Q(\ol x , \ol y)$ such that the sum on the right side of~\eqref{eqn-quotient-def} differs from $\ol d_X(\ol x , \ol y)$ by at most $\ep$. Since $f$ is 1-Lipschitz and by the triangle inequality,
\eqbn
\ol d_X(\ol x , \ol y) +\ep \geq \sum_{i=1}^n d_X (x_i ,y_i )   \geq \sum_{i=1}^n d_Y (f(x_i) ,f(y_i) ) \geq   d_Y (\ol f(\ol x) , \ol f(\ol y)) .
\eqen
Since $\ep$ is arbitrary, we conclude.
 
\vspace{6pt}

\noindent
For a curve $\gamma : [a,b] \rta X$, the \emph{$d_X$-length} of $\gamma$ is defined by 
\eqbn
\op{len}\left( \gamma ; d_X  \right) := \sup_P \sum_{i=1}^{\# P} d_X (\gamma(t_{i-1}) , \gamma(t_i)) 
\eqen
where the supremum is over all partitions $P : a= t_0 < \dots < t_{\# P} = b$ of $[a,b]$. Note that the $d_X$-length of a curve may be infinite.
\vspace{6pt}

\noindent
For $Y\subset X$, the \emph{internal metric $d_Y$ of $d_X$ on $Y$} is defined by
\eqb \label{eqn-internal-def}
d_Y (x,y)  := \inf_{\gamma \subset Y} \op{len}\left(\gamma ; d_X \right) ,\quad \forall x,y\in Y 
\eqe 
where the infimum is over all curves in $Y$ from $x$ to $y$. 
The function $d_Y$ satisfies all of the properties of a metric on $Y$ except that it may take infinite values. 
\vspace{6pt}

\subsection{Outline}
\label{sec-outline}

The remainder of this paper is structured as follows. 

In Section~\ref{sec-schaeffer-stuff}, we prove of Theorem~\ref{thm-core-ghpu} in the following manner. We first recall the Schaeffer-type bijections for quadrangulations with general boundary and the UIHPQ and the ``pruning" procedure which allows one to recover the UIHPQ$_{\op{S}}$ as the simple-boundary core of the UIHPQ.  We use these bijections to establish local absolute continuity estimates for the boundary paths of uniform quadrangulations with general boundary and the UIHPQ.  These estimates together with the explicit description of the laws of the dangling quadrangulations of the UIHPQ enable us to show that the rescaled paths $\wh \xi^n$ and $\xi^n$ in Theorem~\ref{thm-core-ghpu} are typically close in the uniform topology when $n$ is large.  We will then deduce Theorem~\ref{thm-core-ghpu} from this statement and the scaling limit result~\cite[Theorem~1]{bet-mier-disk} for the Brownian disk. In Section~\ref{sec-core-ghpu-fb} we explain why Theorem~\ref{thm-core-ghpu} implies Proposition~\ref{prop-core-ghpu-fb}, which is a variant of Theorem~\ref{thm-fb-ghpu} where the boundary length of the free Boltzmann quadrangulation with simple boundary is a random variable $L^\el$ which is typically of order $(1+o_\el(1)) \el $ when $\el$ is large; in particular, $L^\el$ has the law of $1/2$ times the perimeter of the core of a free Boltzmann quadrangulation with general boundary of perimeter $6\el$.

Most of the remainder of the paper is focused on deducing Theorem~\ref{thm-fb-ghpu} from Proposition~\ref{prop-core-ghpu-fb}. The basic idea of the proof is to use the peeling procedure to remove a small cluster from $Q^\el$ whose complement has the law of a free Boltzmann quadrangulation with perimeter $L_\delta^\el$, where $\delta>0$ is small and $L_\delta^\el$ is a random variable with the law in Proposition~\ref{prop-core-ghpu-fb} with $\lfloor (1+\delta) \el \rfloor$ in place of $\el$, independent from $Q^\el$. 

In Section~\ref{sec-peeling}, we recall the definition of the peeling procedure for quadrangulations with simple boundary, introduce some notation to describe it, and review some relevant formulas. We then prove several estimates for general peeling processes. We obtain in Proposition~\ref{prop-length-area-conv} a scaling limit result for the joint law of the area and boundary length processes of an arbitrary peeling process on the UIHPQ$_{\op{S}}$ analogous to the result~\cite[Theorem~1]{curien-legall-peeling} for peeling processes on the UIPQ; and in Section~\ref{sec-rn-deriv} we prove Radon-Nikodym derivative estimates which allow us to compare peeling processes on free Boltzmann quadrangulations with simple boundary and the UIHPQ$_{\op{S}}$.

In Section~\ref{sec-peeling-by-layers}, we introduce the peeling-by-layers process for quadrangulations with simple boundary, which approximates the growing family of filled metric balls centered at an edge on the boundary. This process is a variant of the peeling-by-layers process for the UIPQ introduced in~\cite{curien-legall-peeling} (and the analogous process for the UIPT introduced in~\cite{angel-peeling}). We then prove several estimates for this peeling process in the case of the UIHPQ$_{\op{S}}$ which will be transferred to estimates in the case of the free Boltzmann quadrangulation with simple boundary using the estimates of Section~\ref{sec-rn-deriv}.

In Section~\ref{sec-conclusion}, we conclude the proof of Theorem~\ref{thm-fb-ghpu} and use it to deduce Theorem~\ref{thm-saw-conv}.

\section{Proof of Theorem~\ref{thm-core-ghpu} via the Schaeffer bijection}
\label{sec-schaeffer-stuff}

In Sections~\ref{sec-quad-bdy} and~\ref{sec-uihpq} we will review the Schaeffer-type constructions of quadrangulations with general boundary and of the UIHPQ.  We will also review the so-called pruning procedure by which one obtains an instance of the UIHPQ$_{\op{S}}$ from an instance of the UIHPQ. In Sections~\ref{sec-bdy-rn} and~\ref{sec-core-ghpu-proof}, we use these constructions together with the results for the UIHPQ from~\cite{gwynne-miller-uihpq} to prove Theorem~\ref{thm-core-ghpu}. In Section~\ref{sec-core-ghpu-fb}, we explain why Theorem~\ref{thm-core-ghpu} implies a scaling limit result for free Boltzmann quadrangulations with simple boundary and certain random perimeter.

We emphasize that this is the only subsection of the paper in which the Schaeffer-type constructions discussed just below are used.

\subsection{Schaeffer bijection for quadrangulations with boundary} 
\label{sec-quad-bdy}

For $n , \el \in \BB N$, let 
$\wh{\mcl Q}^{\bullet}(n,l)$ be the set of triples $(\wh Q , \wh{\BB e} , \wh{\BB v} )$ where $\wh Q$ is a quadrangulation with general boundary having $n$ interior faces and $2\el$ boundary edges (counted with multiplicity), $\wh{\BB e} \in \mcl E(\bdy\wh Q)$ is an oriented root edge, and $\wh{\BB v} \in \mcl V(\wh Q)$ is a marked vertex.  By Euler's formula, the number of vertices of an element of $\wh{\mcl Q}^\bullet(n,\el)$ is determined by $n$ and $\el$ (in particular, it is given by $n+\el+1$), so a uniform sample from $\wh{\mcl Q}^\srta(n,\el)$ can be recovered from a uniform sample from $\wh{\mcl Q}^\bullet(n,\el)$ by forgetting the marked vertex $\wh{\BB v}$ (c.f.~\cite[Lemma~10]{bet-mier-disk}).

In this subsection we review a variant of the Schaeffer bijection for elements of $\wh{\mcl Q}^{ \bullet}(n,l)$ which is really a special case of the Bouttier-Di Francesco-Guitter bijection~\cite{bdg-bijection}. Our presentation is similar to that in~\cite[Section~3.3]{curien-miermont-uihpq},~\cite[Section~3.3]{bet-mier-disk}, and~\cite[Section~3.1]{gwynne-miller-uihpq}.  
 
For $l\in\BB N$, a \emph{bridge} of length $2l$ is a function $b^0 : [0,2l]_{\BB Z} \rta \BB Z$ such that $b^0(j+1) - b^0(j ) \in \{-1,1\}$ for each $j\in [0,2l-1]_{\BB Z}$ and $b^0(0) = b^0(2l) = 0$.  A bridge $b^0$ can be equivalently represented by the function $b : [0, l]_{\BB Z} \rta \BB Z$ which skips all of the upward steps. More precisely, let $j_0 = 0$, for $k \in [1,l]_{\BB Z}$ let $j_k$ be the smallest $j\in [j_{k-1}+1,2l]_{\BB Z}$ for which $b^0(j+1) - b^0(j) = - 1$, and let $b(k) := b^0(j_k)$. 

For $n,l\in\BB N$, a \emph{treed bridge} of area $n$ and boundary length $2l$ is an $(l+1)$-tuple $(b^0 ; (\frk t_0 , \frk v_0, L_0) , \ldots , (\frk t_{l-1} , \frk v_{l-1} , L_{l-1}) )$ such that $b^0$ is a bridge of length $2 l$; and $(\frk t_k , \frk v_k, L_k)$ for $k\in [0,l-1]_{\BB Z}$ is a rooted plane tree with a label function $L_k : \mcl V(\frk t_k) \rta \BB Z$ satisfying $ L_k(v) - L_k(v')  \in \{-1,0,1\}$ whenever~$v$ and~$v'$ are joined by an edge and $L_k(\frk v_k) = b(k)$, where~$b$ is constructed from~$b^0$ as above; and the total number of edges in the trees $\frk t_k$ for $k\in [0,l-1]_{\BB Z}$ is~$n$. Let $\mcl T^{ \bullet}(n,l)$ be the set of pairs consisting of a treed bridge of area~$n$ and boundary length~$2l$ together with a sign $\theta \in \{-,+\}$ (which will be used to determine the orientation of the root edge). 

We now explain how to construct an element of $\wh{\mcl Q}^{\bullet}(n,l)$ from an element of $\mcl T^\bullet(n,l)$.  We first construct a rooted, labeled planar map $(F ,e_0 , L)$ with two faces as follows. For each $k\in [0,l-2]_{\BB Z}$, draw an edge connecting the root vertices $\frk v_k$ and $\frk v_{k+1}$. Also draw an edge connecting $\frk v_{l-1}$ and $\frk v_0$. Embed the cycle consisting of the vertices $\frk v_k$ together with these edges into $\BB C$ in such a way that the vertices $\frk v_k$ all lie on the unit circle. We can extend this embedding to the trees $\frk t_k$ in such a way that each is mapped into the unit disk and no two trees intersect. This gives us a planar map $F$ with an inner face of degree $2n +  l$ (containing all of the trees $\frk t_k$) and an outer face of degree $ l$. Let $ e_0$ be the oriented edge of $F$ from $\frk v_{l-1}$ to $\frk v_0$ and let $L$ be the label function on $\mcl V(F) = \bigcup_{k=0}^{l-1} \mcl V(\frk t_k)$ given by restricting each of the label functions $L_k$ for $k\in[0,l-1]_{\BB Z}$. 

To construct a rooted, pointed quadrangulation with boundary, let $\BB p : [0,2n+l]_{\BB Z} \rta \mcl V(F)$ be the contour exploration of the inner face of~$F$ started from $\frk v_1$, i.e.\ the concatenation of the contour explorations of the trees $\frk t_0, \ldots , \frk t_{l-1}$. We abbreviate $L(i) = L(\BB p(i))$.  Each $i\in [0,2n+l]_{\BB Z}$ is associated with a unique corner of the inner face of $F$ (i.e.\ a connected component of $B_\ep(\BB p(i) ) \setminus F$ for small $\ep > 0$).  Let $\wh{\BB v}$ be an extra vertex not connected to any vertex of $F$, lying in the interior face of $F$.  For $i\in [0,2n+l]_{\BB Z}$, define the successor $s(i)$ of $i$ to be the smallest $i' \geq i$ (with elements of $[0,2n+l]_{\BB Z}$ viewed modulo $2n+l$) such that $L(i') = L(i)-1$, or let $s(i) = \infty$ if no such $i'$ exists. For $i\in [0,2n+l]_{\BB Z}$, draw an edge from the corner associated with $i$ to the corner associated with $ s(i) $, or an edge from $\BB p(i)$ to $\wh{\BB v}$ if $s(i) = \infty$. Then, delete all of the edges of $F$ to obtain a map $\wh Q$. The root edge of $\wh Q$ is the oriented edge $\wh{\BB e} \in\mcl E(\bdy Q)$ from $\frk v_0$ to $\BB p(s (0))$ (if $\theta = -$) or from $\BB p(s (0))$ to $\frk v_0$ (if $\theta = +$), viewed as a half-edge on the boundary of the external face.

As explained in, e.g.,~\cite[Section~3.2]{curien-miermont-uihpq} and~\cite[Section~3.3]{bet-mier-disk}, the above construction defines a bijection from $\mcl T^{\bullet}(n,l)$ to $\wh{\mcl Q}^{ \bullet}(n,l)$. 

We now explain how an element of $\mcl T^\bullet(n,l)$, and thereby an element of $\wh{\mcl Q}^\bullet(n,\el)$, can be encoded by a pair of integer-valued functions.
For $i\in [0,2n+l]_{\BB Z}$, let $k_i\in [0,l-1]_{\BB Z}$ be chosen so that the vertex $\BB p(i)$ belongs to the tree $\frk t_{k_i}$ and let
\eqb \label{eqn-quad-bdy-path}
C(i) := \op{dist}\left( \BB p(i) , \frk v_{k_i} ; \frk t_{k_i} \right) - k_i  , \quad \forall i \in [0,2n+l-1]_{\BB Z}  \quad\text{and}\quad C(2n+l)  = - l ,
\eqe  
so that $C$ is the concatenation of the contour functions of the trees $\frk t_{k}$, but with an extra downward step whenever we move between two trees.  Let
\eqb \label{eqn-quad-bdy-inf}
I(k) := \min\left\{ i \in [0,2n+l]_{\BB Z} : C(i)= -k \right\}  ,\quad \forall  k\in [0 ,l ]_{\BB Z}.
\eqe  
be the first time $i$ for which $\BB p(i) \in \frk t_{k}$ (so that $\BB p(I(k)) = \frk v_k$ for $k\in [0,l-1]_{\BB Z}$ and $\BB p(l) = \frk v_0$). 
Also let $L^0(i) := L(i) - b(k_i)$.
To describe the law of the pair $(C  , L^0)$ we need the following definition.
  
\begin{defn} \label{def-discrete-snake}
Let $[a,b]_{\BB Z}$ be a (possibly infinite) discrete interval and let $S : [a,b]_{\BB Z} \rta \BB Z$ be a (deterministic or random) path with $S(i) - S(i-1) \in \{-1,0,1\}$ for each $i\in [a+1,b]_{\BB Z}$. The \emph{head of the discrete snake} driven by $S$ is the (random) function $H : [a,b]_{\BB Z} \rta \BB Z$ whose conditional law given $S$ is described as follows. We set $H(a) = 0$. Inductively, suppose $ i \in [a+1,b]_{\BB Z}$ and $H(i)$ has been defined for $j\in [a,i-1]_{\BB Z}$. If $S(i) - S(i-1) \in \{-1,0\}$, let $i'$ be the largest $j \in [a,i-1]_{\BB Z}$ for which $S(i) = S(i')$; or $i' = -\infty$. If $i' \not=-\infty$, we set $H(i) = H(i')$. Otherwise, we sample $H(i)- H(i-1)$ uniformly from $\{-1,0,1\}$.  
\end{defn}

The following lemma, which also appears in~\cite{gwynne-miller-uihpq}, is immediate from the definitions and the fact that the above construction is a bijection.

\begin{lem} \label{lem-quad-bdy-law}
If we sample $(\wh Q , \wh{\BB e} , \wh{\BB v} )$ uniformly from $\wh{\mcl Q}^{ \bullet}(n,l)$, then the law of $C$ is that of a simple random walk started from 0 and conditioned to reach $-l$ for the first time at time $2n+l$. The process $L^0$ is the head of the discrete snake driven by $i \mapsto C(i)- \min_{j\in [1,i]_{\BB Z}} C(j)$. The pair $(C , L^0)$ is independent from $b^0$.  
\end{lem} 
  
\subsection{Schaeffer bijection for the UIHPQ} 
\label{sec-uihpq}
  
In this subsection we describe an infinite-volume analog of the bijection of Section~\ref{sec-quad-bdy} which encodes the UIHPQ which is alluded to but not described explicitly in~\cite[Section~6.1]{curien-miermont-uihpq} and described in detail in~\cite{gwynne-miller-uihpq,bmr-uihpq}. See also~\cite{caraceni-curien-uihpq} for a different encoding.

We first define the infinite-volume analog of the bridge $b^0$. Let $b^{\infty,0} : \BB Z\rta \BB N_0$ be given by the absolute value of a two-sided simple random walk with increments sampled uniformly from $\{-1,1\}$.  Let $\{j_k\}_{k\in\BB Z}$ be the ordered set of times~$j$ for which $b^{\infty,0}(j+1) - b^{\infty,0} (j) = -1$, enumerated in such a way that $j_1$ is the smallest $j\geq 0$ for which $b^{\infty,0} (j+1) -  b^{\infty,0} (j) = -1$. Also let $b^\infty(k) := b^{\infty,0} (j_k)$. 

Conditional on $b^{\infty,0}$, let $\{(\frk t_k^\infty , \frk v_k^\infty, L_k^\infty )\}_{k\in\BB Z}$ be a bi-infinite sequence of independent triples where each $(\frk t_k^\infty , \frk v_k^\infty)$ is a rooted Galton-Watson tree whose offspring distribution is geometric with parameter $1/2$; and, conditional on each tree $\frk t_k^\infty$, the function $L_k^\infty $ is uniformly distributed over the set of all functions $   \mcl V(\frk t_k^\infty) \rta \BB Z$ satisfying $L_k^\infty(\frk v_{\infty,k}) = b^\infty(k)$ and $ L_k^\infty(u)-L_k^\infty(v) \in \{-1,0,1\}$ whenever $u, v\in\mcl V(\frk t_k^\infty)$ are connected by an edge.

To construct an instance of the UIHPQ from the above objects, we first construct a planar graph $F^\infty$ with two faces. Equip $\BB Z$ with the standard nearest-neighbor graph structure and embed it into the real line in $\BB C$. For $k \in \BB Z$, embed the tree~$\frk t_k^\infty$ into the upper half-plane in such a way that the vertex~$\frk v_k^\infty$ is identified with $k \in \BB Z$ and none of the trees~$\frk t_k^\infty$ intersect each other or intersect~$\BB R$ except at their root vertices. The graph~$F^\infty$ is the union of~$\BB Z$ and the trees~$\frk t_k^\infty$ for $k\in\BB Z$ with this graph structure.  Let $L^\infty$ be the label function on the vertices of $F^\infty$ satisfying $L^\infty|_{\mcl V(\frk t_k^\infty)} = L_k^\infty$ for each $k\in\BB Z$. 

Let $\BB p^\infty : \BB Z\rta \mcl V(F^\infty)$ be the contour exploration of the upper face of $F^\infty$ shifted so that~$\BB p^\infty$ starts exploring the tree~$\frk t_1^\infty$ at time~$0$. Define the successor $s^\infty(i)$ of each time $i\in\BB Z$ exactly as in the Schaeffer bijection (here we do not need to add an extra vertex since a.s.\ $\liminf_{i\rta\infty} L^\infty(i) = -\infty$). Then draw an edge connecting each vertex $\BB p^\infty(i)$ to $\BB p^\infty(s^\infty(i))$ for each $i \in \BB Z$ and delete the edges of $F^\infty$. This gives us an infinite quadrangulation with boundary $\wh Q^\infty$. The root edge $\wh{\BB e}^\infty$ of $\wh Q^\infty$ is the oriented edge $\wh{\BB e}^{\infty}$ which goes from~$\frk v_0^\infty$ to~$\BB p^\infty(s^\infty(0))$. Then $(\wh Q^\infty, \wh{\BB e}^{\infty })$ is an instance of the UIHPQ with general boundary.

As in Section~\ref{sec-quad-bdy}, we re-phrase the above encoding in terms of random paths.  
For $i\in \BB Z$, let $k_i$ be chosen so that the vertex $\BB p^\infty(i)$ belongs to the tree $\frk t_{k_i}^\infty$ and let
\eqb \label{eqn-quad-bdy-path-infty}
C^\infty(i) := \op{dist}\left( \BB p^\infty(i) , \frk v_{k_i}^\infty ; \frk t_{k_i}^\infty \right) - k_i    , \quad \forall i \in \BB Z 
\eqe   
be the contour function of the upper face of $F^\infty$. 
Also let
\eqb \label{eqn-quad-bdy-inf-infty}
I^\infty(k) := \min\left\{ i \in \BB Z : C^\infty(i)= -k \right\}  ,\quad \forall  k\in \BB Z 
\eqe  
so that $\BB p(I^\infty(k)) = \frk v_k$. 
Finally, define $L^\infty(i) := L^\infty(\BB p^\infty(i))  $ and $L^{\infty,0}(i) := L^\infty(i) - b^\infty(k_i)$.

The following is~\cite[Lemma~3.5]{gwynne-miller-uihpq}.

\begin{lem} \label{lem-uihpq-encode-law}
The pair $(C^\infty ,L^{\infty,0} )$ is independent from $b^\infty$ and its law can be described as follows.  The law of $C^\infty|_{\BB N_0}$ is that of a simple random walk started from $0$ and the law of $C^\infty(-\cdot)|_{\BB N_0}  $ is that of a simple random walk started from~$0$ and conditioned to stay positive for all time (see, e.g.,~\cite{bertoin-doney-nonnegative} for a definition of this conditioning for a large class of random walks).  Furthermore, $L^{\infty,0}$ is the head of the discrete snake driven by $i \mapsto C^\infty(i) - \min_{j\in (-\infty,i]_{\BB Z}} C^\infty(j)$ (Definition~\ref{def-discrete-snake}). 
\end{lem}

\subsection{Pruning the UIHPQ to get the UIHPQ$_{\op{S}}$}
\label{sec-pruning}

Recall from Section~\ref{sec-intro-def-quad} that the UIHPQ$_{\op{S}}$ is the Benjamini-Schramm local limit of uniformly random quadrangulations with simple boundary, as viewed from a uniformly random vertex on the boundary, as the area and then the perimeter tends to $\infty$. The simple boundary core (Section~\ref{sec-intro-def-quad}) of the UIHPQ has the law of the UIHPQ$_{\op{S}}$. More precisely, suppose $(\wh Q^\infty , \wh{\BB e}^\infty)$ is a UIHPQ and let $ Q^\infty = \op{Core}(\wh Q^\infty)$ be the quadrangulation obtained from $Q^\infty$ by pruning all of the ``dangling quadrangulations" of $\wh Q^\infty$ which are joined to~$\infty$ by a single vertex. Let $ \BB e^\infty$ be the edge immediately to the left of the vertex which can be removed to disconnect~$\wh{\BB e}^\infty$ from~$\infty$ (if such a vertex exists) or let $\BB e^\infty = \wh{\BB e}^\infty$ if $\wh{\BB e}^\infty$ belongs to $\bdy Q^\infty$. Then $(Q^\infty , \BB e^\infty)$ is an instance of the UIHPQ$_{\op{S}}$. 

One obtains a boundary path $\beta^\infty : \BB Z\rta\mcl E(\bdy Q^\infty)$ with $\beta^\infty(0) = \BB e^\infty$ from the boundary path $\wh\beta^\infty$for $\wh Q^\infty$ by skipping all of the intervals of time during which $\wh\beta^\infty$ is tracing a dangling quadrangulation. 

There is also an explicit sampling procedure which reverses the above construction (c.f.~\cite[Section~6.1.2]{curien-miermont-uihpq} or~\cite[Section~6]{caraceni-curien-uihpq}). Let $(Q^\infty, \BB e^\infty)$ be a UIHPQ$_{\op{S}}$ and conditionally on $Q^\infty$, let $\{(q_v , e_v)\}_{v \in \mcl V(\bdy Q^\infty)}$ be an independent sequence of random finite quadrangulations with general boundary with an oriented boundary root edge, with distributions described as follows.  Let~$v_0$ be the right endpoint of the root edge~$\BB e^\infty$.  Each~$q_v$ for $v\not=v_0$ is distributed according to the so-called \emph{unconstrained free Boltzmann distribution on quadrangulations with general boundary}, which is given by
\eqb \label{eqn-free-boltzman}
\BB P\left[ (q_v ,e_v) = (\frk q, \frk e) \right] = C^{-1} \left( \frac{1}{12} \right)^n \left( \frac18 \right)^l
\eqe  
for any quadrangulation $\frk q$ with $n \in \BB N_0$ interior faces and $2\el$, $\el \in\BB N$, boundary edges (counted with multiplicity) with a distinguished oriented root edge $\frk e\in \bdy \frk q$, where here $C >0$ is a normalizing constant. 
The quadrangulation~$q_{v_0}$ is instead distributed according to
\eqb \label{eqn-free-boltzman-root}
\BB P\left[ (q_{v_0} , e_{v_0}) = (\frk q , \frk e) \right] = \wt C^{-1} (2l+1) \left( \frac{1}{12} \right)^n \left( \frac18 \right)^l  
\eqe  
for a different normalizing constant $\wt C$. We note that by~\cite[Equation~(23)]{caraceni-curien-uihpq}, the expected perimeter of $q_v$ for $v\not= v_0$ is equal to~$2$.

If we identify the terminal endpoint of $e_v$ with $v$ for each $v\in \mcl V(\bdy Q^\infty)$, we obtain an infinite quadrangulation~$\wh Q^\infty$ with general boundary. We choose an oriented root edge~$\wh{\BB e}^\infty$ for~$\wh Q^\infty$ by uniformly sampling one of the oriented edges of $\mcl E(\bdy q_{v_0}) \cup \{\BB e^\infty \}$. Then $(\wh Q^\infty , \wh{\BB e}^\infty)$ is a UIHPQ which can be pruned to recover $(Q^\infty , \BB e^\infty)$.

\subsection{Radon-Nikodym derivative estimates for quadrangulations with general boundary}
\label{sec-bdy-rn}

In the remainder of this section, we assume that we are in the setting of Theorem~\ref{thm-core-ghpu}, so that $(\wh Q^n , \wh{\BB e}^n)$ is a uniform quadrangulation with general boundary with $a^n$ interior faces and perimeter $2\el^n$. To describe $(\wh Q^n , \wh{\BB e}^n)$ via the bijection of Section~\ref{sec-quad-bdy}, we let $\wh{\BB v}^n$ be a marked vertex sampled uniformly from $\mcl V(\wh Q^n)$.

We denote the Schaeffer encoding for $(\wh Q^n , \wh{\BB e}^n)$ from Section~\ref{sec-quad-bdy} with an additional superscript $n$, so that in particular $C^n$ is the contour function, $L^n$ is the label process, $L^{0,n}$ is the shifted label process, $b^{0,n}$ is a random walk bridge independent from $(C^n , L^{0,n})$, and $b^n$ is obtained from $b^{0,n}$ by skipping the upward steps.  Also let $I^n$ be as in~\eqref{eqn-quad-bdy-inf}. 

We also let $(\wh Q^\infty , \wh{\BB e}^\infty)$ be an instance of the UIHPQ with general boundary and define its Schaeffer encoding functions~$C^\infty$, $L^\infty$, $L^{\infty,0}$, $b^{\infty,0}$, $b^\infty$, and~$I^\infty$ as in Section~\ref{sec-uihpq}. 

The aforementioned Schaeffer encoding functions have easy-to-describe laws and determine the corresponding quadrangulations in a local manner. This enables us to obtain Radon-Nikodym derivative estimates for the law of some part of $\wh Q^n$ with respect to the law of the corresponding part of $\wh Q^\infty$. This technique has been used in~\cite{gwynne-miller-uihpq,bmr-uihpq} to couple a uniform quadrangulation with general boundary with the UIHPQ in such a way that they agree with high probability in a small neighborhood of the root edge (see~\cite{curien-legall-plane} for an analogous statement for quadrangulations without boundary). In this subsection, we will prove weaker Radon-Nikodym derivative estimates which hold for a larger part of the quadrangulations in question. We start by proving a Radon-Nikodym estimate for the encoding functions.

\begin{lem}
\label{lem-path-rn}
For each $\ep \in (0,1)$, there exists $A = A(\ep) > 0$ and $n_* = n_*(\ep) \in \BB N$ such that the following is true for each $n\geq n_*$.  On an event of probability at least $1-\ep$ (for the law of $(C^n, L^n)$), the law of $(C^n , L^n)|_{[0 , I^n(\el^n - \ep n^{1/2})]_{\BB Z}}$ is absolutely continuous with respect to the law of $(C^\infty , L^\infty)|_{[0,I^\infty(\el^n - \ep n^{1/2})]_{\BB Z}}$, with Radon-Nikodym derivative bounded above by $A$.   
\end{lem}
\begin{proof}
The proof is similar to that of~\cite[Lemma~4.7]{gwynne-miller-uihpq}. 
Recall from Lemmas~\ref{lem-quad-bdy-law} and~\ref{lem-uihpq-encode-law} that the law of $ C^n$ is that of a simple random walk conditioned to first hit $-\el^n$ at time $2a^n +\el^n$ and the law of $ C^\infty|_{\BB N_0} $ is that of an unconditioned simple random walk. 

By~\cite[Lemma~4.6]{gwynne-miller-uihpq} and Bayes' rule (c.f.\ the proof of~\cite[Lemma~4.7]{gwynne-miller-uihpq}), 
the Radon-Nikodym derivative of the law of $  C^n|_{[0 , I^n(\el^n - \ep n^{1/2})]_{\BB Z} }$ with respect to the law of $  C^\infty|_{[0 , I^\infty(\el^n - \ep n^{1/2})]_{\BB Z}}$ is given by $f_\ep^n\left( I^\infty(\el^n - \ep n^{1/2}) \right)$
where for $k \in  [ 0 , 2a^n + \el^n ]_{\BB Z}$,
\alb
f_\ep^n (k  ) = \frac{  \ep n^{1/2}  (2 a^n + \el^n - k   )^{-3/2} \exp\left(- \frac{ \ep^2 n }{2(2a^n + \el^n - k )}  \right)     + o_n(n^{-1})  }{ l^n  (2 a^n + \el^n    )^{-3/2} \exp\left(- \frac{ (\el^n )^2 }{2(2a^n + \el^n )}  \right)     + o_n(n^{-1})     }    \BB 1_{(k < 2 a^n +\el^n)} .
\ale 
Since $(2n)^{-1/2}  C^n((2n)^{-1} \cdot)$ converges in law in the uniform topology to an appropriate conditioned Brownian motion~\cite[Lemma~14]{bet-pos-genus}, we can find $\zeta  = \zeta(\ep) > 0$ and $n_0 = n_0(\ep) \in\BB N $ such that for $n\geq n_0$, 
\eqbn
\BB P[E_0^n] \geq 1-\ep/2 \quad \op{where} \quad E_0^n := \left\{  I^n(\el^n - \ep n^{1/2}) \leq 2 a^n + \el^n - \zeta n \right\}.
\eqen
Since $(2n)^{-1/2} \el^n \rta \frk l$, we can find $A_0 = A_0(\ep ) > 0$ and $n_1 = n_1(\ep) \geq n_0$ such that for $n\geq n_1$ and $1\leq  k \leq 2 a^n + \el^n -  \zeta n$, we have $f_\ep^n(k) \leq A_0$. 

Hence for $n\geq n_1$, on the event $E_0^n$ the Radon-Nikodym of the law of $C^n|_{[0 , I^n(\el^n - \ep n^{1/2})]_{\BB Z}}$ is absolutely continuous with respect to the law of $C^\infty|_{[0,I^\infty(\el^n - \ep n^{1/2})]_{\BB Z}}$, with Radon-Nikodym derivative bounded above by $A_0$.   
Since the conditional law of the shifted label function $L^{0,n}|_{[0, I^n(\el^n - \ep n^{1/2})]_{\BB Z}}$ given $C^n|_{[0, I^n(\el^n - \ep n^{1/2}) ]_{\BB Z}}$ coincides with the conditional law of $L^{\infty,0}|_{[0, I^\infty(\el^n - \ep n^{1/2})]_{\BB Z}}$ given $C^\infty|_{[0, I^\infty(\el^n - \ep n^{1/2}) ]_{\BB Z}}$, 
we get the same Radon-Nikodym derivative estimate with the pairs $(C^n , L^{0,n})$ and $(C^\infty , L^{\infty,0})$ in place of $C^n$ and $C^\infty$. 

Recall that $L^n$ (resp.\ $L^\infty$) is obtained from $(C^n , L^{0,n})$ and the bridge $b^{0,n}$ (resp.\ $(C^\infty , L^{\infty,0})$ and the walk $b^{\infty,0}$) in the manner described in Section~\ref{sec-quad-bdy} (resp.\ Section~\ref{sec-uihpq}). 
Recall also the processes $b^n$ and $b^\infty$ obtained from $b^{0,n}$ and $b^{\infty,0}$, respectively, by considering only times when the path makes a downward step. 
A similar absolute continuity argument to the one given above shows that there exists $n_* = n_*(\ep) \geq n_1$, $A_1 = A_1(\ep)  >0$, and an event $E_1^n$ with $\BB P[E_1^n] \geq 1-\ep/2$ such that for $n\geq n_*$, the Radon-Nikodym derivative of the law of $b^n|_{[0,\el^n  -\ep n^{1/2}]_{\BB Z}}$ with respect to the law of $b^\infty|_{[0,\el^n  -\ep n^{1/2}]_{\BB Z}}$ on the event $E_1^n$ is bounded above by~$A_1$. 
 
The pair $(C^n , L^{0,n})$ (resp.\ $(C^\infty , L^{\infty,0}$) is independent from $b^{0,n}$ (resp.\ $b^{\infty,0}$), so for $n\geq n_*$ it holds on $E_0^n \cap E_1^n$ that the law of the pair $\left( (C^n , L^{0,n}) |_{[0,I^n(\el^n - \ep n^{1/2})]_{\BB Z}}   ,   b^n|_{[0,\el^n  -\ep n^{1/2}]_{\BB Z}} \right)$ is absolutely continuous with respect to the law of $\left( (C^\infty , L^{\infty,0}) |_{[0,I^\infty(\el^n - \ep n^{1/2})]_{\BB Z}}   ,   b^\infty|_{[0,\el^n  -\ep n^{1/2}]_{\BB Z}} \right)$ with Radon-Nikodym derivative bounded above by $A_0A_1$.  Since these processes determine $ (C^n , L^n) |_{[0,I^n(\el^n - \ep n^{1/2})]_{\BB Z}}$ and $(C^\infty , L^\infty) |_{[0,I^\infty(\el^n - \ep n^{1/2})]_{\BB Z}}$, respectively, via the same deterministic functional and $\BB P[E_0^n\cap E_1^n] \geq 1-\ep$, we obtain the statement of the lemma with $A = A_0A_1$.  
\end{proof}

Let $\wh\beta^n : [0,2\el^n]_{\BB Z} \rta \mcl E(\bdy \wh Q^n)$ and $\wh\beta^\infty : \BB Z\rta \mcl E(\bdy \wh Q^\infty)$ be the boundary paths of our finite and infinite quadrangulations with general boundary, respectively, started from the root edge at time $0$.  For $k \in [0,2\el^n]_{\BB Z}$, we can view $\wh\beta^n([0,k]_{\BB Z})$ as a planar map and $\wh\beta^n|_{[0,k]_{\BB Z}}$ as a path on it. This planar map can have non-trivial structure since $\wh\beta^n$ is not necessarily a simple path.  Hence it makes sense to consider the law of $\wh\beta^n|_{[0,k]_{\BB Z}}$, without reference to the underlying map $\wh Q^n$.  Similar considerations hold for $\wh\beta^\infty$. 

From Lemma~\ref{lem-bdy-rn}, we obtain a Radon-Nikodym estimate for boundary paths, viewed without reference to the underlying map in the manner described just above.

\begin{lem} \label{lem-bdy-rn}
For each $\ep \in (0,1)$, there exists $A = A(\ep) > 0$ and $n_* = n_*(\ep) \in \BB N$ such that the following is true for each $n\geq n_*$. 
On an event of probability at least $1-\ep$ (for the law of $\wh\beta^n$), the law of $\wh\beta^n|_{[0, 2\el^n - \ep n^{1/2}]_{\BB Z}}$ is absolutely continuous with respect to the law of $\wh\beta^\infty|_{[0,2\el^n-\ep n^{1/2}]_{\BB Z}}$, with Radon-Nikodym derivative bounded above by $A$.  
\end{lem}
\begin{proof}  
It is clear from the Schaeffer bijection (c.f.~\cite[Remarks 3.1 and 3.4]{gwynne-miller-uihpq}) and a basic concentration estimate for the empirical distribution of the times when the simple random walk bridge $b^{0,n}$ and the random walk $b^{\infty,0}$ take a downward step that with probability tending to 1 as $n\rta\infty$ (with respect to the laws of each of $\wh\beta^n$ and $\wh\beta^\infty$), $\wh\beta^n|_{[0,2\el^n - \ep n^{1/2}]_{\BB Z}}$ and $\wh\beta^\infty|_{[0,2\el^n-\ep n^{1/2}]_{\BB Z}}$
are given by the same deterministic functional of $(C^n , L^n)|_{[0 , I^n(\el^n - (\ep/4) n^{1/2})]_{\BB Z}}$ and $(C^\infty , L^\infty)|_{[0,I^\infty(\el^n - (\ep/4) n^{1/2})]_{\BB Z}}$, respectively. The statement of the lemma therefore follows from Lemma~\ref{lem-path-rn}.
\end{proof}

\subsection{Proof of Theorem~\ref{thm-core-ghpu}}
\label{sec-core-ghpu-proof}

In this subsection we will prove our scaling limit result for the simple-boundary core $Q^n$ of $\wh Q^n$.  The main difficulty of the proof is the uniform convergence of the rescaled boundary path $\xi^n$ of $Q^n$.  This will be extracted from the following proposition, which in turn will follow from the estimates of Section~\ref{sec-bdy-rn} and the analogous statement for the UIHPQ which is proven in~\cite{gwynne-miller-uihpq} using the pruning procedure of Section~\ref{sec-pruning}.  Here we recall that $\wh\beta^n$ is the boundary path of $\wh Q^n$.

\begin{prop} \label{prop-core-bdy-reg}
For each $\ep \in (0,1)$, there exists $n_* = n_*(\ep) \in \BB N$ such that for $n\geq n_*$, it holds with probability at least $1-\ep$ that the following is true. For each $k_1,k_2 \in [0,2\el^n]_{\BB Z}$ with $k_1 \leq k_2$, the number of edges in $\wh\beta^n([k_1,k_2]_{\BB Z})$ which belong to the simple-boundary core $Q^n$ is between $\frac13( k_2-k_1)  - \ep n^{1/2}$ and $\frac13( k_2-k_1 )  + \ep n^{1/2}$.
\end{prop}

\begin{figure}[ht!]
 \begin{center}
\includegraphics[scale=1]{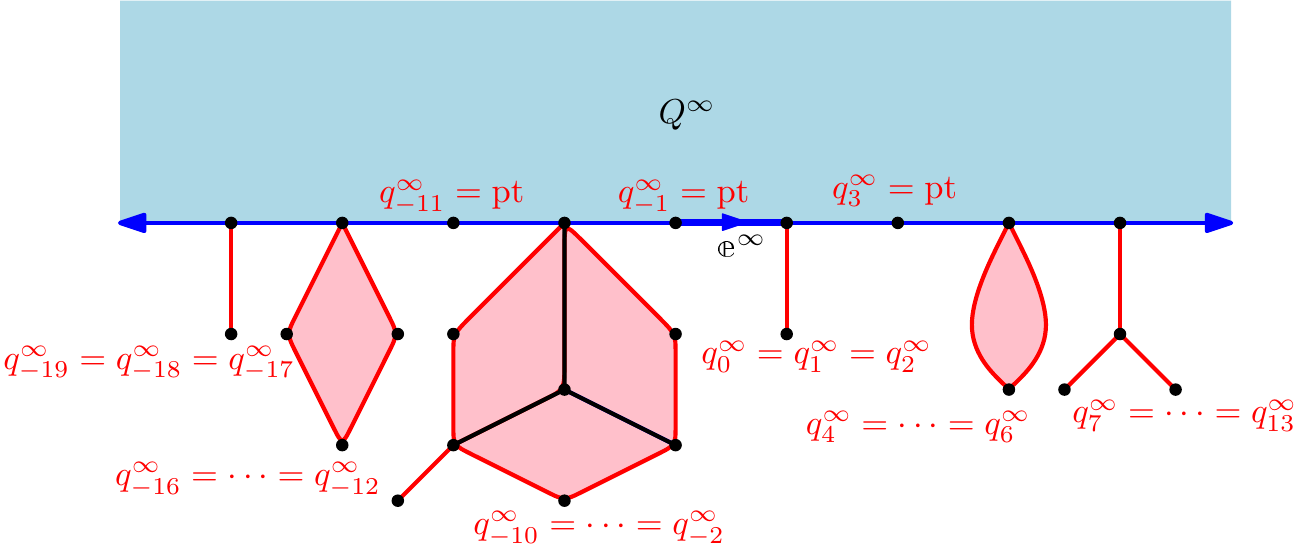} 
\caption{\label{fig-pruning} The UIHPQ $\wh Q^\infty$ (blue and red) and its UIHPQ$_{\op{S}}$ core $Q^\infty$ (blue). The dangling quadrangulations $\wh q_k^\infty$ for $k\in \BB N$, used in Section~\ref{sec-core-ghpu-proof}, are shown in red. The dangling quadrangulation $\wh q_k^\infty$ is the one which contains the right endpoint of the boundary edge $\wh \beta^\infty(k)$ if $\wh\beta^\infty(k) \in \bdy Q^\infty$ (i.e., $\wh\beta^\infty(k)$ is one of the blue edges) or the one which contains $\wh\beta^\infty(k)$ if $\wh\beta^\infty(k)$ is one of the red edges.}
\end{center}
\end{figure}

For the proof of Proposition~\ref{prop-core-bdy-reg}, we will need to consider the pruning procedure described in Section~\ref{sec-pruning}.  Recall the UIHPQ $(\wh Q^\infty  , \wh{\BB e}^\infty)$ and its boundary path $\wh\beta^\infty$. Also let $(Q^\infty,\BB e^\infty)$ be the UIHPQ$_{\op{S}}$ with $ Q^\infty = \op{Core}(\wh Q^\infty)$, as in Section~\ref{sec-pruning}. 

For $k\in \BB Z$, let $q_k^\infty$ be the dangling quadrangulation of $\wh Q^\infty$ (i.e., the quadrangulation which can be disconnected from the core $Q^n$ by removing a single vertex) such that either $\wh\beta^\infty(k) \in \bdy q_k^\infty$ or $\wh\beta^\infty(k) \in \bdy Q^\infty$ and $q_k^\infty$ is attached to the right endpoint of $\wh\beta^\infty(k)$. Similarly, for $k \in [0,2\el^n]_{\BB Z}$ let $q_k^n$ be the dangling quadrangulation of $\wh Q^n$ such that either $\wh\beta^n(k) \in \bdy q_k^n$ or $\wh\beta^n(k) \in \bdy \op{Core}(Q^n)$ and $q_k^n$ is attached to the right endpoint of $\wh\beta^n(k)$. See Figure~\ref{fig-pruning} for an illustration of these definitions. We note that $\{q_k^\infty\}_{k\in\BB Z}$ is the same, as a set, as the set of dangling quadrangulations $\{q_v^\infty\}_{v\in\mcl V(\bdy Q^\infty)}$ described in Section~\ref{sec-pruning}. However, the index $k$ in the present section corresponds to the boundary path of $\wh Q^\infty$, so in particular it is possible that $q_{k_1}^\infty = q_{k_2}^\infty$ for $k_1\not=k_2$. 

The following lemma tells us that the size of a dangling quadrangulation is typically of constant order, independently of $k$ and $n$. 

\begin{lem} \label{lem-dangling-size}
For each $\ep \in (0,1)$, there exists $N = N(\ep) \in \BB N$ such that the following is true. For each $n\in\BB N$ and each $k \in [0,2\el^n]_{\BB Z}$, we have $\BB P[\#\mcl E(q_k^n) \leq N ] \geq 1-\ep$ and for each $k \in \BB Z$ we have $\BB P[\#\mcl E(q_k^\infty) \leq N ] \geq 1-\ep$.
\end{lem}
\begin{proof}
If we condition on $\wh Q^n$, then the root edge $\wh{\BB e}^n = \wh\beta^n(0)$ is sampled uniformly from $\bdy \wh Q^n$. It follows that the law of $(\wh Q^n, \wh{\BB e}^n)$ is invariant under the operation of replacing $\wh{\BB e}^n$ with $\wh\beta^n(k)$ for any $k\in [0,2\el^n]_{\BB Z}$. Passing to the local limit shows that the law of the UIHPQ $(\wh Q^\infty, \wh{\BB e}^\infty)$ is invariant under the operation of replacing $\wh{\BB e}^\infty$ by $\wh\beta^\infty(k)$ for any $k\in\BB Z$.
Therefore, the law of $q_k^n$ (resp.\ $q_k^\infty$) does not depend on $k$.  It is clear that $q_k^\infty$ is a.s.\ finite, so for each $\ep \in (0,1)$ there exists $N  = N (\ep) \in \BB N$ such that for $k\in\BB Z$, we have  $\BB P[\#\mcl E(q_k^\infty) \leq N ] \geq 1-\ep/2$. 

It remains to prove an upper bound for the size of $q_0^n$.  By~\cite[Proposition~4.5]{gwynne-miller-uihpq} there exists $\alpha = \alpha(\ep) > 0$ and $n_* = n_*(\ep) \in\BB N$ such that for $n\geq n_*$, we can couple $\wh Q^n$ and $\wh Q^\infty$ in such a way that it holds with probability at least $1-\ep/4$ that the following is true.  The graph metric balls $B_{\alpha n^{1/4}}(\wh{\BB e}^n ; \wh Q^n)$ and $B_{\alpha n^{1/4}}(\wh{\BB e}^\infty ; \wh Q^\infty)$ equipped with the graph structures they inherit from $\wh Q^n$ and $\wh Q^\infty$, respectively, are isomorphic (as graphs) via an isomorphism which takes $\wh{\BB e}^n$ to $\wh{\BB e}^\infty$ and $ \bdy \wh Q^n  \cap B_{\alpha n^{1/4}}(\wh{\BB e}^n ; \wh Q^n)$ to $ \bdy \wh Q^\infty  \cap B_{\alpha n^{1/4}}(\wh{\BB e}^{\infty} ; \wh Q^\infty)$. 

By~\cite[Lemma~4.9]{gwynne-miller-uihpq}, the maximal rescaled diameter $n^{-1/4} \max_{k\in [0,\el^n]_{\BB Z}} \op{diam} (q_k^n)$ tends to~$0$ in law as $n\rta\infty$. Hence by possibly increasing $n_*$, we can arrange that with probability at least $1-\ep/2$ our coupling is such that $q_0^\infty = q_0^n$.  By combining this with the first paragraph of the proof we obtain the statement of the lemma.
\end{proof}

\begin{figure}[ht!]
 \begin{center}
\includegraphics[scale=1]{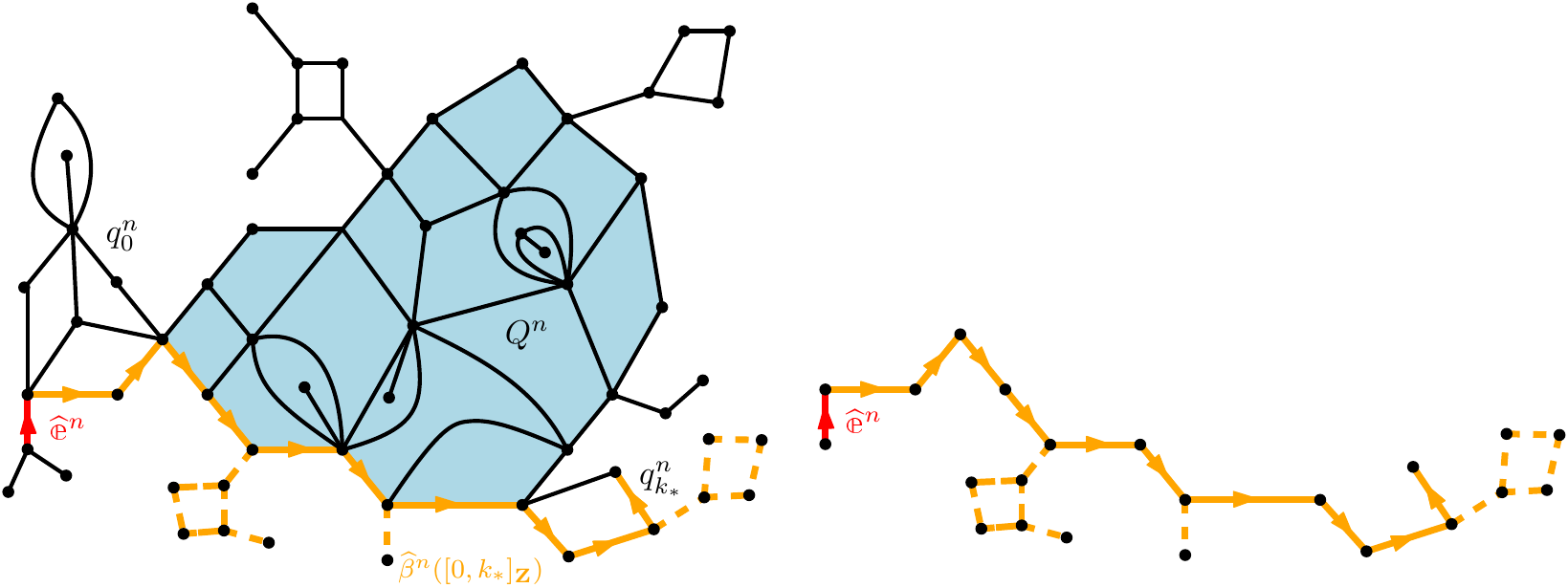} 
\caption{\label{fig-bdy-rn} \textbf{Left:} The quadrangulation $\wh Q^n$ with general boundary together with a segment $\wh\beta^n([0,k_*]_{\BB Z})$ of its boundary path. The associated loop-erased path $\wh\beta_{k_*}^n : [0,k_*]_{\BB Z} \rta \mcl E(\bdy \wh Q^n)$ used in the proof of Proposition~\ref{prop-core-bdy-reg} is obtained by replacing the segments of $\wh\beta^n([0,k_*]_{\BB Z})$ when it is tracing the dotted edges by constant segments. If $k_*$ is smaller than the time $K^n$ when $\wh\beta^n$ finished tracing $\bdy Q^n$, then $\wh\beta_{k_*}^n([0,k_* ]_{\BB Z} )$ contains $\bdy Q^n\cap \wh\beta_{k_*}^n([0,k_* ]_{\BB Z} )$ and is contained in the union of $\bdy Q^n \cap \wh\beta_{k_*}^n([0,k_* ]_{\BB Z})$ and the dangling quadrangulations $q_0^n$ and $ q_{k_*}^n$. \textbf{Right:} The path $\wh\beta^n|_{[0,k_*]_{\BB Z}}$, viewed without reference to the map $\wh Q^n$, is absolutely continuous with respect to the analogous boundary path increment for the UIHPQ$_{\op{S}}$.}
\end{center}
\end{figure}

\begin{proof}[Proof of Proposition~\ref{prop-core-bdy-reg}]
See Figure~\ref{fig-bdy-rn} for an illustration.  For $k_* \in [0,2\el^n]_{\BB Z}$, let $\wh\beta_{k_*}^\infty : [0,k_*]_{\BB Z} \rta \mcl E(\bdy \wh Q^\infty )$ be the path obtained by erasing the loops from $\wh\beta^\infty|_{[0,k_*]_{\BB Z}}$ in the following manner.  Let $[a_1,b_1]_{\BB Z} ,\dots  , [a_N , b_N]_{\BB Z}$ be the maximal discrete intervals $[a,b]_{\BB Z}$ in $[1,k_*]_{\BB Z}$ with the property that $\wh\beta^\infty([a,b]_{\BB Z})$ is a cycle which is not contained in any larger cycle in $\wh\beta^\infty([1,k]_{\BB Z})$, ordered from left to right. For $k \in [0,k_*]_{\BB Z}$ let $\wh\beta_{k_*}^\infty(k)  = \wh\beta^\infty( k')$ where $ k'$ is the largest $s \in [0, k]_{\BB Z}$ which is not contained in $\bigcup_{i=1}^N [a_i , b_i]_{\BB Z}$.  Similarly construct~$\wh\beta_{k_*}^n$ from~$\wh\beta^n|_{[0,k_*]_{\BB Z}}$. 

Since the boundary of the UIHPQ$_{\op{S}}$ $Q^\infty = \op{Core}(\wh Q^\infty)$ does not contain a cycle and the boundary of each $q_k^\infty$ is a cycle traced by $\wh\beta^\infty$, it follows that for each $0\leq k_1 \leq k_2 \leq k_*$,  
\eqb \label{eqn-no-loop-path-infty}
Q^\infty \cap \wh\beta^\infty([k_1,k_2]_{\BB Z}) \subset \wh\beta_{k_*}^\infty([k_1,k_2]_{\BB Z}) \subset \left( Q^\infty \cap \wh\beta^\infty([k_1,k_2]_{\BB Z})  \right) \cup q_0^\infty \cup q_{k_*}^\infty .
\eqe 
Similarly, if we let $K^n$ be the time at which $\beta^n$ finishes tracing $\bdy Q^n$ then for $0 \leq k_1 \leq k_2 \leq k_*  < K^n$,  
\eqb \label{eqn-no-loop-path-n}
Q^n \cap \wh\beta^n([k_1,k_2]_{\BB Z}) \subset  \wh\beta_{k_*}^n([k_1,k_2]_{\BB Z}) \subset \left( Q^n \cap \wh\beta^n([k_1,k_2]_{\BB Z})  \right) \cup q_0^n \cup q_{ k_* }^n .
\eqe 

Now fix $\ep \in (0,1)$ and set 
\eqbn
k_*^n(\ep) := \lfloor 2\el^n - \ep n^{1/2} \rfloor . 
\eqen
Also fix $\delta \in (0,1)$ to be chosen later, depending only on $\ep$. 

Recall from Section~\ref{sec-pruning} that the ordered (from left to right) collection of distinct dangling quadrangulations other than $q_0^n$ is i.i.d., and the expected perimeter of each of these quadrangulations is 2; note here that the law of the quadrilateral dangling from the $j$th edge of $Q^\infty$ does \emph{not} have the same law as $q_k^n$ since $q_k^n$ is more likely to be one of the dangling quadrangulations with longer perimeter.  From this and the law of large numbers (see~\cite[Lemma~4.12]{gwynne-miller-uihpq} for a careful justification) we infer that there exists $n_* = n_*(\ep,\delta) \in \BB N$ such that for $n\geq n_*$, it holds with probability at least $1-\delta/2$ that the following is true. For each $k_1,k_2 \in [0,2\el^n - \ep n^{1/2}]_{\BB Z}$ with $k_1<k_2$, the number of edges in $\wh\beta^\infty([k_1,k_2]_{\BB Z})$ which belong to $Q^\infty$ is between $\frac13( k_2-k_1 ) -\frac12\ep n^{1/2}$ and $\frac13( k_2-k_1 ) +  \frac12\ep n^{1/2}$. By Lemma~\ref{lem-dangling-size}, by possibly increasing $n_*$ we can arrange that it holds with probability at least $1-\delta/2$ that $\#\mcl E( q_0^\infty \cup q_{ k_*^n(\ep)}^\infty) \leq \frac12 \ep n^{1/2}$.  By~\eqref{eqn-no-loop-path-infty}, it holds with probability at least $1-\delta$ that 
\eqb \label{eqn-bdy-path-reg-infty}
\frac13( k_2 - k_1 )  - \ep n^{1/2} \leq  \# \wh\beta_{ k_*^n(\ep) }^\infty ([k_1,k_2]_{\BB Z}) \leq \frac13( k_2 - k_1 )   + \ep n^{1/2} , \quad \forall 0 \leq k_1 \leq k_2 \leq k_*^n(\ep) .
\eqe 

By Lemma~\ref{lem-bdy-rn}, by possibly increasing $n_*$ we can find $A= A(\ep)  > 0$ that for $n\geq n_*$, there is an event $E^n$ with $\BB P[E^n] \geq 1-\ep/4$ such that on $E^n$, the Radon-Nikodym derivative of the law of $\wh\beta^n|_{[0, k_*^n(\ep) ]_{\BB Z}}$ is absolutely continuous with respect to the law of $\wh\beta^\infty|_{[0, k_*^n(\ep) ]_{\BB Z}}$, with Radon-Nikodym derivative bounded above by $A$.  

Set $\delta = \frac14 A^{-1} \ep$ for this choice of $A$. By~\eqref{eqn-bdy-path-reg-infty} for $n\geq n_*$, it holds with probability at least $1-\ep/2$ that
\eqb \label{eqn-bdy-path-reg-n}
\frac13( k_2 - k_1 )  - \ep n^{1/2} \leq  \# \wh\beta_{ k_*^n(\ep) }^n([k_1,k_2]_{\BB Z}) \leq \frac13( k_2 - k_1 )   + \ep n^{1/2} , \quad \forall 0 \leq k_1 \leq k_2 \leq k_*^n(\ep) .
\eqe 
By Lemma~\ref{lem-dangling-size}, by possibly increasing $n_*$ we can arrange that for $n\geq n_*$, it holds with probability at least $1-\ep/2$ that $ \#\mcl E( q_0^n \cup q_{k_*^n(\ep)}^n    ) \leq \frac12 \ep n^{1/2} $. If this is the case and~\eqref{eqn-bdy-path-reg-n} holds, then necessarily $k_*^n(\ep) \leq K^n$ since $\wh\beta^n([k_*^n(\ep) , 2\el^n]_{\BB Z}) \subset q_0^n$. By~\eqref{eqn-no-loop-path-n}, for $n\geq n_*$ it holds with probability at least $1-\ep$ that  
\eqb \label{eqn-core-reg-n}
\frac13( k_2 - k_1 )  - 2\ep n^{1/2} \leq  \# \left(   Q^n \cap \wh\beta^n([k_1,k_2]_{\BB Z}) \right) \leq \frac13( k_2 - k_1 )   + 2\ep n^{1/2} ,\quad \forall 0 \leq k_1 \leq k_2 \leq k_*^n(\ep) . 
\eqe 

By the re-rooting invariance of the law of $(\wh Q^n,\wh{\BB e}^n)$ (which comes from the fact that $\wh{\BB e}^n$ is sampled uniformly from $\mcl E(\bdy \wh Q^n)$), we can apply the same argument with $\wh Q^n$ rooted at $\wh\beta^n(\lfloor \ep n^{1/2} \rfloor)$ instead of $\wh{\BB e}^n = \wh\beta^n(0)$ to find that with probability at least $1-2 \ep$,~\eqref{eqn-core-reg-n} also holds for each $\lfloor \ep n^{1/2} \rfloor \leq k_1 \leq k_2 \leq 2\el^n $.  By splitting a given interval $[k_1,k_2]_{\BB Z} \subset [0,2\el^n]_{\BB Z}$ into an interval contained in $[0,k_*^n(\ep)]_{\BB Z}$ and an interval contained in $[\ep n^{1/2} ,2\el^n]_{\BB Z}$, we obtain the statement of the proposition with $4\ep$ in place of $\ep$. Since $\ep$ can be made arbitrarily small, we conclude.
\end{proof}

\begin{proof}[Proof of Theorem~\ref{thm-core-ghpu}]
By~\cite[Theorem~4.1]{gwynne-miller-uihpq}, $\wh{\frk Q}^n \rta \frk H$ in law in the GHPU topology.  By the Skorokhod representation theorem, we can couple $\{(\wh Q^n, \wh{\BB e}^n) \}_{n\in\BB N}$ with $\frk H$ in such a way that this convergence occurs a.s.  By~\cite[Proposition~1.5]{gwynne-miller-uihpq}, for any such coupling we can a.s.\ find a compact metric space $(W,D)$ and isometric embeddings $(\wh Q^n ,\wh d^n) \rta (W,D)$ for $n\in\BB N$ and $(H,d) \rta (W,D)$ such that if we identify these spaces with their images under the corresponding embeddings then a.s.\ $\wh Q^n \rta H^n$ in the $D$-Hausdorff distance, $\wh\mu^n \rta \mu$ in the $D$-Prokhorov distance, and $\wh\xi^n\rta\xi$ in the $D$-uniform distance. Henceforth fix such a coupling and such a space $(W,D)$. We note that the isometric embedding $(\wh Q^n,\wh d^n) \rta (W,d)$ restricts to an isometric embedding $(Q^n , d^n) \rta (W,D)$, so $Q^n$ is identified with a subset of $W$.

Since $H$ has the topology of a disk, it follows that the maximal $\wh d^n$-diameter of the dangling quadrangulations of $\wh Q^n$ tends to zero in probability as $n\rta \infty$ (see~\cite[Lemma~4.9]{gwynne-miller-uihpq} for a careful justification). By possibly choosing a different coupling we can take this convergence to occur a.s.  From this we infer that a.s.\ $ Q^n \rta H$ in the $D$-Hausdorff distance and (since $\wh \xi^n \rta \xi$ uniformly) that $\wh Q^n\setminus (Q^n \setminus \bdy Q^n) \rta \bdy H = \xi([0,\frk l])$ in the $D$-Hausdorff distance.

Since each of the measures $\wh\mu^n - \mu^n $ is supported on $\wh Q^n\setminus (Q^n \setminus \bdy Q^n)$, we infer that any subsequential limit of the measures $\wh\mu^n - \mu^n$ in the $D$-Prokhorov distances is supported on $\bdy H$ and is dominated by $\mu$. Since $\mu(\bdy H) = 0$ any such subsequential limit must be the zero measure. Therefore, $\mu^n \rta \mu$ in the $D$-Prokhorov distance.

To show the uniform convergence $\xi^n \rta \xi$ of the re-scaled boundary paths, for $n\in\BB N$ and $t\in  [0,   (3/2^{3/2}) n^{-1/2} \#\mcl E(\bdy Q^n) ]$ let $\tau^n(t)$ be the smallest $s  \in [0, (2n)^{-1/2} \el^n]$ for which the number of edges of $\bdy Q^n$ traversed by $\wh\xi^n$ between times 0 and $s$ is at least $ (2^{3/2} /3) n^{1/2} t $. Equivalently, $2^{3/2} n^{1/2} \tau^n(t)$ is the smallest time at which the un-scaled boundary path $\wh\beta^n$ has traversed at least $ (2^{3/2} /3) n^{1/2} t $ edges of $\bdy Q^n$. Then 
\eqbn
\wh\xi^n(\tau^n(t) ) = \beta^n\left( \frac{2^{3/2}}{3} n^{1/2} t + o_n(1) \right) = \xi^n(t + o_n(1)) 
\eqen
 where the $o_n(1)$ comes from rounding error. Note that the scaling factors in the time parameterizations of $\xi^n$ and $\wh\xi^n$ differ by a factor of 3. By Proposition~\ref{prop-core-bdy-reg}, the function $\tau^n$ converges uniformly to the identity function in probability, whence $\xi^n \rta \xi$ uniformly in probability.  From this we infer that $\frk Q^n \rta \frk H$ in probability, as required.
\end{proof}

\subsection{Scaling limit of free Boltzmann quadrangulations with random perimeter}
\label{sec-core-ghpu-fb}

In this brief subsection, we explain why Theorem~\ref{thm-core-ghpu} implies a scaling limit result (Proposition~\ref{prop-core-ghpu-fb}) for a free Boltzmann quadrangulation with \emph{random} boundary length. Most of the remainder of the paper will be devoted to transferring this result to the case when we specify the exact boundary length of the quadrangulation. 

To state our result, we need to recall the definition of a free Boltzmann quadrangulation with general boundary, which appears, e.g., in~\cite[Section~1.4]{bet-mier-disk}. 
The \emph{free Boltzmann distribution on quadrangulations with general boundary of perimeter $2\el$} is the probability measure on $\bigcup_{n=0}^\infty \wh{\mcl Q}^\srta(n,\el)$ (defined as in Section~\ref{sec-intro-def-quad}) which assigns to each $(\wh{\frk Q}  ,\wh{\frk e} ) \in \wh{\mcl Q}^\srta(\el)$ a probability equal to $\wh{\mcl Z}(\el)^{-1} 12^{-\#\mcl F(\wh{\frk Q})}$, where $\wh{\mcl Z}(\el) = \sum_{n=0}^\infty 12^{-n} \#\wh{\mcl Q}^\srta(n,\el)$ is the partition function. 

Free Boltzmann quadrangulations with general and simple boundaries are related by the following lemma.

\begin{lem} \label{lem-fb-gen-law}
Let $\el \in \BB N$ and let $(\wh Q , \wh{\BB e} )$ be a free Boltzmann quadrangulation with general boundary of perimeter $2\el$. If we condition on the rooted planar map $(\bdy \wh Q , \wh{\BB e})$ then the conditional law of the collection of simple-boundary components of $\wh Q$, each rooted at an oriented boundary edge which is chosen in a $\sigma (\bdy \wh Q , \wh{\BB e})$-measurable manner, is that of a collection of independent free Boltzmann quadrangulations with simple boundary with perimeters given by the perimeters of the internal faces of $\bdy \wh Q$. 
\end{lem}
\begin{proof}
Let $N$ by the (random) number of simple-boundary components of $\wh Q$ and let $Q_1 , \dots , Q_N$ be these components, enumerated in the order in which their boundaries are first hit by the boundary path of $\wh Q$ started from $\wh{\BB e}$. Also let $\BB e_k \in \mcl E(\bdy Q_k)$ for $k\in [1,N]_{\BB Z}$ be a root edge for $Q_k$ chosen in a $\sigma (\bdy \wh Q , \wh{\BB e})$-measurable manner. 

Let $(\wh{\frk B}, \wh{\frk e} )$ be a possible realization of $(\bdy \wh Q , \wh{\BB e})$, let $\frk n$ be the corresponding realization of $N$, and let $(\frk B_k , \frk e_k)$ be the corresponding realizations of $(\bdy Q_k , \BB e_k)$ for $k\in [1, \frk n]_{\BB Z}$. Also let $\frk l_k := \frac12 \#\mcl E(\frk B_k)$ be half the perimeter of $Q_k$. 
 
There is a bijection from the set of possible realizations of $(\wh Q , \wh{\BB e})$ with $(\bdy \wh Q , \wh{\BB e}) = (\wh{\frk B} , \wh{\frk e})$ to $\mcl Q_{\op{S}}^\srta(\frk l_1) \times \dots \times \mcl Q_{\op{S}}^\srta(\frk l_{\frk n})$: the forward bijection is obtained by taking the $\frk n$-tuple of simple-boundary components of such a realization, each rooted at the corresponding edge $\frk e_k$; and the inverse bijection is obtained by identifying the boundary of each of the $\frk n$ components of an element of $\mcl Q_{\op{S}}(\frk l_1) \times \dots \times \mcl Q_{\op{S}}(\frk l_{\frk n})$ with the boundary of the corresponding internal face of $\wh{\frk B}$ via an orientation-preserving map which takes the root edge to $\wh{\frk e}$. 

Suppose now that $((\frk Q_1 , \frk e_1') ,\dots , (\frk Q_{\frk n} , \frk e_{\frk n}') ) \in \mcl Q_{\op{S}}^\srta(\frk l_1) \times \dots \times \mcl Q_{\op{S}}^\srta(\frk l_{\frk n})$ and let $\wh{\frk Q}$ be the corresponding realization of $\wh Q$ satisfying $\bdy \wh{\frk Q} = \wh{\frk B}$.   
Each internal face of $\wh{\frk Q}$ is an internal face of precisely one of the $\frk Q_k$'s. Therefore,
\allb \label{eqn-fb-gen-decomp}
&\BB P\left[ (Q_k , \BB e_k) = (\frk Q_k , \frk e_k') ,\: \forall k \in [1,\frk n]_{\BB Z}  \,|\, (\bdy \wh Q , \wh{\BB e}) = (\wh{\frk B} ,\wh{\frk e}) \right] 
  =\BB P\left[ \wh Q = \wh{\frk Q} \,|\, (\bdy \wh Q , \wh{\BB e}) = (\wh{\frk B} ,\wh{\frk e}) \right] \notag \\
&\qquad =  \BB P\left[(\bdy \wh Q , \wh{\BB e}) = (\wh{\frk B} ,\wh{\frk e}) \right]^{-1}   12^{-\#\mcl F(\wh Q)}  
 = \BB P\left[(\bdy \wh Q , \wh{\BB e}) = (\wh{\frk B} ,\wh{\frk e}) \right]^{-1}  \prod_{k=1}^{\frk n} 12^{-\#\mcl F(\frk Q_k)}.
\alle
By Euler's formula, if $Q$ is a quadrangulation with simple boundary then $\#\mcl F(Q) = \frac12 \#\mcl E(\bdy Q) - 1  + \#\mcl V(Q\setminus \bdy Q)$. Hence the right side of~\eqref{eqn-fb-gen-decomp} equals
\eqbn
C(\wh{\frk B} , \wh{\frk e}) \prod_{k=1}^{\frk n} 12^{-\#\mcl V(\frk Q_k \setminus \bdy \frk Q_k)}
\eqen
where $C(\wh{\frk B} , \wh{\frk e})$ is a constant depending only on $(\wh{\frk B} , \wh{\frk e})$. Therefore, the conditional law of $\{(Q_k , \BB e_k)\}_{k \in [1,\frk n]_{\BB Z}}$ is as described in the statement of the lemma. 
\end{proof}

From Lemma~\ref{lem-fb-gen-law} and Theorem~\ref{thm-core-ghpu}, we obtain the following variant of Theorem~\ref{thm-fb-ghpu} when we randomize the perimeter, which will be used in subsequent sections to prove Theorem~\ref{thm-fb-ghpu}.

\begin{prop} \label{prop-core-ghpu-fb}
Fix $\el \in \BB N$ and for $n\in\BB N$ let $L^\el$ be a random variable whose law is that of $\frac12 \#\mcl E(\op{Core}(\wh Q^\el))$, where $\wh Q^\el$ is a free Boltzmann quadrangulation with \emph{general} boundary of perimeter $6\el$.  Condition on $L^\el$, sample a free Boltzmann quadrangulation with simple boundary of perimeter $2L^\el$ and 
define the curve-decorated metric measure spaces $\frk Q^{L^\el}$ for $\el\in\BB N$ as in Theorem~\ref{thm-fb-ghpu} with $L^\el$ in place of $\el$. Then $\frk Q^{L^\el} \rta \frk H$ in law, where $\frk H$ is the limiting curve-decorated metric measure space from Theorem~\ref{thm-fb-ghpu}. 
\end{prop}
\begin{proof}
For $\el\in\BB N$ let $(\wh Q^\el ,\wh{\BB e}^\el)$ be a free Boltzmann quadrangulation with general boundary of perimeter $6\el$ and let $L^\el$ be half the perimeter of its core.
By Lemma~\ref{lem-fb-gen-law}, the conditional law given $L^\el$ of $\op{Core}(\wh Q^\el)$ (equipped with an oriented root edge chosen in a manner which depends only on $(\bdy \wh Q^\el , \wh{\BB e}^\el$) is that of a free Boltzmann quadrangulation with simple boundary of perimeter $2L^\el$. Hence we can couple $Q^{L^\el}$ with $\wh Q^\el$ in such a way that $Q^{L^\el} = \op{Core}(\wh Q^\el)$ a.s. 
 
Let $\wh A^\el$ for $\el \in \BB N$ be the (random) number of faces of $\wh Q^\el$. 
The proof of~\cite[Theorem~8]{bet-mier-disk} shows that $(2/9) \el^{-2} \wh A^\el$ converges in law as $\el\rta\infty$ to the law of the area of a free Boltzmann Brownian disk with unit perimeter (alternatively, this can be extracted from Lemma~\ref{lem-fb-area-asymp} below, which is a re-statement of a result from~\cite{curien-legall-peeling}).
By Theorem~\ref{thm-core-ghpu} applied to the conditional law of $(\wh Q^\el , \wh{\BB e}^\el)$ given $\wh A^\el$ we obtain the statement of the proposition.
\end{proof}

\section{Peeling processes on quadrangulations with simple boundary}
\label{sec-peeling}

In this section we will study general peeling processes on the UIHPQ$_{\op{S}}$ and on free Boltzmann quadrangulations with simple boundary, which will be our main tool in the remainder of the paper (we will not have any further occasion to consider quadrangulations with general boundary or the Schaeffer bijection).  In Section~\ref{sec-peeling-def}, we review the definition of the peeling procedure, introduce some notation to describe it, and recall some standard formulas and estimates for peeling steps in the UIHPQ$_{\op{S}}$. 

In Section~\ref{sec-bdy-process}, we introduce the boundary length processes for a general peeling process on the UIHPQ$_{\op{S}}$; this result is an analog for the UIHPQ$_{\op{S}}$ of the scaling limit result~\cite[Theorem~1]{angel-curien-uihpq-perc} for the peeling process on the UIPQ, and is proven in the same manner.

In Section~\ref{sec-rn-deriv}, we prove Radon-Nikodym estimates which enable us to compare peeling processes on free Boltzmann quadrangulations with simple boundary to peeling processes on the UIHPQ$_{\op{S}}$.  The results of this latter subsection (especially Lemma~\ref{lem-bubble-cond}) will be our main tool for studying peeling processes on free Boltzmann quadrangulations with simple boundary, both in the present paper and in~\cite{gwynne-miller-uihpq}. 

\begin{remark} \label{remark-peeling-tri}
All of the results of the present section have exact analogs for \emph{triangulations} with simple boundary of type I (multiple edges and self-loops are allowed) or type II (multiple edges, but not self-loops, are allowed). The statements in either of the two triangulation cases are identical, modulo different choices of normalizing constants, and the proofs are essentially the same but sometimes slightly easier due to the simpler description of peeling in the triangulation case.  See~\cite{angel-peeling,angel-uihpq-perc,angel-curien-uihpq-perc,richier-perc} for more on peeling of triangulations with simple boundary. 
\end{remark}

\subsection{General definitions and formulas for peeling} 
\label{sec-peeling-def}

\subsubsection{Peeling at an edge}
\label{sec-general-peeling}

Let $Q$ be a finite or infinite quadrangulation with simple boundary. For an edge $e  \in \mcl E(\bdy Q)$, let $\frk f(Q,e)$ be the quadrilateral of $Q$ containing $e$ on its boundary or let $\frk f(Q,e) = \emptyset$ if $Q \in \mcl Q_{\op{S}}^\srta(0,1)$ is the trivial one-edge quadrangulation with no interior faces.  If $\frk f(Q,e) \not=\emptyset$, the quadrilateral $\frk f(Q,e)$ has either two, three, or four vertices in $\bdy Q$, so divides $Q$ into at most three connected components, whose union includes all of the vertices of $Q$ and all of the edges of $Q$ except for $e$. These components have a natural cyclic ordering inherited from the cyclic ordering of their intersections with $\bdy Q$. 

If there are $m \in \{1,2,3\}$ connected components of $Q\setminus \frk f(Q,e)$, we write $\frk P(Q,e) \in (\BB N_0 \cup \{\infty\})^m $ for the vector whose $i$th component for $i \in \{1,\dots,m\}$ is $\#\mcl E(Q_i\cap \bdy Q)$, where $Q_i$ is the $i$th connected component of $Q\setminus \frk f(Q,e)$ in counterclockwise order started from $e$. We define $\frk P(Q,e) = \emptyset$ if $\frk f(Q,e) = \emptyset$. Note that
\eqb \label{eqn-peel-indicator}
\frk P(Q,e) \in \{\emptyset\} \cup (\BB N_0 \cup \{\infty\}) \cup (\BB N_0 \cup \{\infty\})^2 \cup (\BB N_0 \cup \{\infty\})^3
\eqe 
 We refer to $\frk P(Q,e)$ as the \emph{peeling indicator}.
Several examples of quadrilaterals $\frk f(Q,e)$ and their associated peeling indicators are shown in Figure~\ref{fig-peeling-cases}.

We note that $\frk P(Q,e)$ determines the total boundary lengths of each of the connected components of $Q\setminus \frk f(Q,e)$, not just the lengths of their intersections with $\bdy Q$. Indeed, if the $i$th component of $\frk P(Q,e)$ is $k$, then the total boundary length of the $i$th connected component of $Q\setminus \frk f(Q,e)$ in counterclockwise cyclic order is equal to
\begin{itemize}
\item $k+3$ if there is only one such component (Figure~\ref{fig-peeling-cases}, leftmost illustration),
\item $k+1$ if there is more than one component and $k$ is odd (Figure~\ref{fig-peeling-cases}, three rightmost illustrations),
\item $k+2$ if $k$ is even (Figure~\ref{fig-peeling-cases}, second to the left and rightmost illustrations), and
\item $\infty$ if $k$ is $\infty$.
\end{itemize}
The procedure of extracting $\frk f(Q,e)$ and $\frk P(Q,e)$ from $(Q,e)$ will be referred to as \emph{peeling $Q$ at $e$}. 

\begin{figure}[ht!]
 \begin{center}
\includegraphics[scale=1]{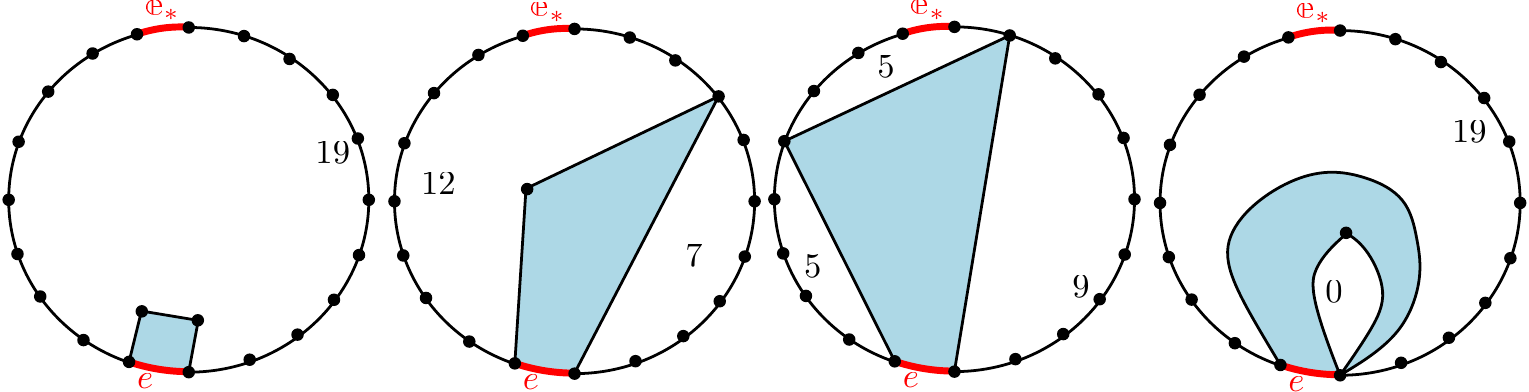} 
\caption{\label{fig-peeling-cases} A finite rooted quadrangulation with simple boundary $(Q,  e) \in \mcl Q_{\op{S}}^\srta(n,20)$ together with four different cases for the peeled quadrilateral $\frk f(Q,  e)$ (shown in light blue). The peeling indicators from left to right are given by $\frk P(Q, e) = 19$, $\frk P(Q, e) = (7,12)$, $\frk P(Q,   e) = (9,5,5)$, and $\frk P(q,e) = (0,19)$. }
\end{center}
\end{figure}

\begin{figure}[ht!]
 \begin{center}
\includegraphics[scale=1]{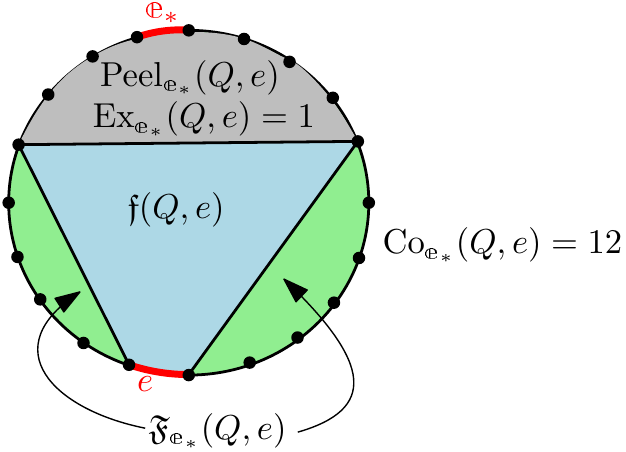} 
\caption{\label{fig-peeling-def} Illustration of various definitions associated with peeling. The peeled quadrilateral $\frk f(Q,e)$ is shown in blue. the connected component $\op{Peel}_{\BB e_*}(Q , e)$ of $Q\setminus \frk f(Q,e)$ containing the target edge is shown in grey, and the union  $\frk F_{\BB e_*}(Q,e)$ of the other connected components is shown in light green. Individual vertices and edges of these components are not shown. }
\end{center}
\end{figure}

Suppose now that $\BB e_* \in \bdy Q\setminus \{e\}$ or that $\bdy Q$ is infinite and $\BB e_* = \infty$. 
\begin{itemize} 
\item Let $\op{Peel}_{\BB e_*}(Q , e)$ be the connected component of $Q\setminus \frk f(Q,e)$ with $\BB e_*$ on its boundary, or $\op{Peel}_{\BB e_*}(Q,e) = \emptyset$ if $\frk f(Q,e) =\emptyset$ (equivalently $(Q,e) \in \mcl Q_{\op{S}}^\srta(0,1)$).
\item Let $\frk F_{\BB e_*}(Q,e)$ be the union of the components of $Q\setminus \frk f(Q,e)$ other than $\op{Peel}_{\BB e_*}(Q,e)$ or $\frk F_{\BB e_*}(Q,e) = \emptyset$ if $\frk f(Q,e) =\emptyset$.  
\item Let $\op{Ex}_{\BB e_*}(Q,e)$ be the number of \emph{exposed edges} of $\frk f(Q,e)$, i.e.\ the number of edges of $ \op{Peel}_{\BB e_*}(Q,e)$ which do not belong to $\bdy Q$ (equivalently, those which are incident to $\frk f(Q,e)$). 
\item Let $\op{Co}_{\BB e_*}(Q,e)$ be the number of \emph{covered edges} of $\bdy Q$, i.e.\ the number of edges of $ \bdy Q $ which do not belong to $\op{Peel}_{\BB e_*}(Q,e)$ (equivalently, one plus the number of such edges which belong to $\frk F_{\BB e_*}(Q,e)$). 
\end{itemize}
See Figure~\ref{fig-peeling-def} for an illustration of the above definitions.

\subsubsection{Markov property and peeling processes}
\label{sec-peeling-process}

If $(Q,\BB e)$ is a free Boltzmann quadrangulation with simple boundary of perimeter $2\el$ for $\el\in \BB N\cup \{\infty\}$ and we condition on $\frk P(Q,\BB e)$, then the connected components of $Q\setminus \frk f(Q,\BB e)$ are conditionally independent.  The conditional law of each of the connected components, rooted at one of the edges of $\frk f(Q,\BB e)$ on its boundary (chosen by some deterministic convention in the case when there is more than one such edge), is the free Boltzmann distribution on quadrangulations with simple boundary and perimeter $2\wt \el$ (Definition~\ref{def-fb}), for a $\sigma(\frk P(Q,\BB e))$-measurable choice of $\wt \el \in \BB N\cup \{\infty\}$. These facts are collectively referred to as the \emph{Markov property of peeling}.
    
Due to the Markov property of peeling, one can iteratively peel a free Boltzmann quadrangulation with boundary to obtain a sequence of quadrangulations with simple boundary with explicitly described laws.  To make this notion precise, let $\el \in \BB N\cup \{\infty\}$ and let $(Q, \BB e_*)$ be a free Boltzmann quadrangulation with simple boundary with perimeter $2\el$; we also allow $\BB e_*=\infty$ in the UIHPQ$_{\op{S}}$ case (when $\el=\infty$). 

Suppose we are given a sequence of (possibly empty) quadrangulations with simple boundary $\ol Q_j \subset Q$ for $j \in \BB N_0$ and edges $\dot e_j \in \bdy \ol Q_{j-1}$ for each $j\in\BB N$ with $\ol Q_j \not=\emptyset$ such that 
\eqbn
\ol Q_0 = Q \quad \op{and} \quad 
\ol Q_j = \begin{cases} 
\op{Peel}_{\BB e_*}(\ol Q_{j-1} , \dot e_j ) ,\quad & \ol Q_{j-1} \not=\emptyset \\
\emptyset ,\quad & \ol Q_{j-1}  =\emptyset .
\end{cases}
\eqen
We refer to the quadrangulations $\ol Q_j$ as the \emph{unexplored quadrangulations}. 
We also define the \emph{peeling clusters} by
\eqb \label{eqn-peel-cluster0}
\dot Q_j := (Q \setminus \ol Q_j) \cup (\bdy \ol Q_j \setminus \bdy Q) ,\quad \forall j \in \BB N_0 ,
\eqe 
equivalently $\dot Q_0 = \emptyset$ and $\dot Q_j = \dot Q_{j-1} \cup \frk f(\ol Q_{j-1} , \dot e_j ) \cup \frk F(\ol Q_{j-1} , \dot e_j )$. 
We also define the \emph{peeling filtration} by
\eqb \label{eqn-peel-filtration0}
\mcl F_j := \sigma\left( \frk P(\ol Q_{i-1} , \dot e_i ) , \dot Q_i : i \in [1,j-1]_{\BB Z} \right) ,\quad \forall j \in \BB N_0 .
\eqe 

We say that $\{\ol Q_j\}_{j\in\BB N_0}$ is a \emph{peeling process} targeted at $\BB e_*$  if each of the peeled edges $\dot e_j$ for $j\in\BB N_0$ is chosen in an $\mcl F_{j-1}$-measurable manner.  It follows from the Markov property of peeling that in this case, it holds for each $j \in \BB N_0$ that the conditional law of $(\ol Q_j , \dot e_{j+1})$ given the $\sigma$-algebra $\mcl F_j$ of~\eqref{eqn-peel-filtration0} is that of a free Boltzmann quadrangulation with perimeter $2\wt\el$ for some $\wt \el\in\BB N_0\cup \{\infty\}$ which is measurable with respect to $\mcl F_j$ (where here a free Boltzmann quadrangulation with perimeter~$0$ is taken to be the empty set). 

We will typically denote objects associated with peeling processes on the UIHPQ$_{\op{S}}$ by a superscript $\infty$.

\subsection{Peeling formulas for the UIHPQ$_{\op{S}}$}
\label{sec-peeling-estimate}
     
As explained in~\cite[Section~2.3.1]{angel-curien-uihpq-perc}, one has explicit formulas for the law of the peeling indicator~\eqref{eqn-peel-indicator} in the case when $(Q^\infty , \BB e^\infty )$ is a UIHPQ$_{\op{S}}$. With $\frk Z$ the free Boltzmann partition function from~\eqref{eqn-fb-partition}, one has 
\begin{align} \label{eqn-uihpq-peel-prob}
\BB P\left[  \frk P(Q^\infty , \BB e^\infty) = \infty  \right]  
&= \frac38 \notag \\
\BB P\left[ \frk P(Q^\infty , \BB e^\infty) = (k , \infty) \right]  
&=  \frac{1}{12} 54^{(1-k)/2}  \frk Z(k+1)  ,\quad \text{$\forall k\in \BB N$ odd} \notag \\
\BB P\left[ \frk P(Q^\infty , \BB e^\infty) = (k , \infty) \right] 
&=   \frac{1}{12}  54^{-k/2}   \frk Z(k+2)  ,\quad \text{$\forall k\in \BB N_0$ even} \notag \\
\BB P\left[ \frk P(Q^\infty , \BB e^\infty) = (k_1, k_2 , \infty)   \right]  
&=  54^{-(k_1+k_2)/2}   \frk Z(k_1+1) \frk Z(k_2+1)  ,   \quad \text{$\forall k_1,k_2\in  \BB N$ odd} . 
\end{align}
We get the same formulas if we replace $(k,\infty)$ with $(\infty,k)$ or $(k_1,k_2,\infty)$ with either $(\infty,k_1,k_2)$ or $(k_1,\infty,k_2)$.

By Stirling's formula, the partition function $\frk Z$ satisfies the asymptotics 
\eqb \label{eqn-stirling-asymp}
\frk Z(\el ) = (c +o_\el(1)) 54^{\el/2}  \el^{-5/2}  \quad \forall  \el\in\BB N \: \op{even} 
\eqe   
where $c>0$ is a universal constant. 
From this, we obtain approximate versions of the probabilities~\eqref{eqn-uihpq-peel-prob}.  
\begin{align} \label{eqn-uihpq-peel-prob-asymp} 
\BB P\left[ \frk P(Q^\infty , \BB e^\infty) = (k , \infty) \right]  
&\asymp k^{-5/2} ,\quad \forall k\in \BB N   \notag \\
\BB P\left[ \frk P(Q^\infty , \BB e^\infty) = (k_1, k_2 , \infty)   \right]  
&\asymp  k_1^{-5/2} k_2^{-5/2}  ,  \quad \text{$\forall k_1,k_2\in  \BB N$ odd} , 
\end{align}
and similarly with the orders of the components of $\frk P$ re-arranged.  One can also write down the exact law of the peeling indicator variable in the case of a free Boltzmann quadrangulation with simple boundary, which is slightly more complicated than the formulas~\eqref{eqn-uihpq-peel-prob}. We will not need this exact law here, however, since all of our estimates for peeling processes on free Boltzmann quadrangulations with simple boundary will be proven by comparison to the UIHPQ$_{\op{S}}$, using the estimates of Section~\ref{sec-rn-deriv}. 

By~\cite[Proposition~3]{angel-curien-uihpq-perc}, the number of covered and exposed edges (notation as in Section~\ref{sec-general-peeling}) satisfy
\eqb \label{eqn-peel-mean}
\BB E\left[\op{Ex}_\infty\left(Q^\infty , \BB e^\infty \right) \right] = 
\BB E\left[ \op{Co}_\infty \left(Q^\infty , \BB e^\infty \right) \right] =  2 ,
\eqe  
in particular the expected net change in the boundary length of $Q^\infty $ under the peeling operation is 0. We always have $\op{Ex}_\infty(Q^\infty  , \BB e^\infty ) \in \{1,2,3\}$, but $\op{Co}_\infty(Q^\infty  , \BB e^\infty )$ can be arbitrarily large. In fact, a straightforward calculation using~\eqref{eqn-uihpq-peel-prob} shows that for $k\in\BB N$,  
\eqb \label{eqn-cover-tail}
\BB P\left[ \op{Co}_\infty (Q^\infty  , \BB e^\infty ) = k \right] = (c_* + o_k(1)) k^{-5/2} \quad\text{with}\quad  c_* = \frac{58\sqrt 2}{81\sqrt{3\pi}}. 
\eqe

\subsection{Boundary length and area processes and their scaling limits}
\label{sec-bdy-process}

\begin{defn}[Boundary length and area processes] \label{def-bdy-process}
Let $(Q,\BB e)$ be a quadrangulation with simple boundary and let $\{\dot Q_j\}_{j\in\BB N_0}$ and $\{\ol Q_j\}_{j\in\BB N_0}$ be the clusters and unexplored quadrangulations of a peeling process of $(Q,\BB e)$. For $j\in\BB N_0$ we define the \emph{exposed, covered, and net boundary length processes}, respectively, by
\eqbn
X_j := \#\mcl E\left( \bdy \dot Q_j \cap \bdy \ol Q_j \right)  ,\quad 
Y_j := \# \mcl E\left( \bdy \dot Q_j \cap \bdy Q \right) , \quad \op{and} \quad 
W_j := X_j - Y_j .
\eqen
Note that $\dot Q_j$ and $\ol Q_j$ intersect only along their boundaries, so $\bdy \dot Q_j \cap \bdy \ol Q_j = \dot Q_j \cap  \ol Q_j$. 
We also define the \emph{area process} by $A_j := \#\mcl V(\dot Q_j )$. 
In the case of a peeling process on the UIHPQ$_{\op{S}}$, we include an additional superscript $\infty$ in the notation for these objects.
\end{defn}  

In the remainder of this subsection, we specialize to the case of the UIHPQ$_{\op{S}}$. We will prove scaling limit results for the boundary length and area processes for this peeling process analogous to the scaling limit results for general peeling processes of the UIPQ and UIPT proven in~\cite{curien-legall-peeling}. 

Let $(Q^\infty ,\BB e^\infty)$ be a UIHPQ$_{\op{S}}$ and let $\{\dot Q_j^\infty\}_{j\in\BB N_0}$, $\{\ol Q_j^\infty\}_{j\in\BB N_0}$, $\{\dot e_j^\infty\}_{j\in\BB N_0}$, and $\{\mcl F_j^\infty\}_{j\in\BB N_0}$, respectively, be the clusters, unexplored quadrangulations, peeled edges, and filtration of a peeling process of the UIHPQ$_{\op{S}}$ targeted at~$\infty$.  We consider the scaling limit of the joint law of the net boundary length and area processes~$W^\infty$ and~$A^\infty$.  For $n\in\BB N$ and $t\geq 0$, define the scaling constant $b_* := (4/3) \sqrt \pi c_*$, where $c_*$ is the constant in~\eqref{eqn-cover-tail}. Define the rescaled boundary length and area processes by
\eqb \label{eqn-bdy-process-rescale}
Z_t^{\infty,n} :=  b_*^{-1} n^{-1/2}   W_{\lfloor t n^{3/4} \rfloor}^\infty \quad \op{and} \quad U_t^{\infty,n} :=      \frac29  b_*^{-2} n^{-1}  A_{\lfloor t n^{3/4} \rfloor}^\infty .
\eqe 

To describe the limit of the joint law of the processes~\eqref{eqn-bdy-process-rescale}, 
let $Z^\infty$ be a totally asymmetric $3/2$-stable process with no positive jumps, normalized so that its L\'evy measure is $\frac{3}{4\sqrt\pi} |t|^{-5/2} \BB 1_{(t < 0)} \,dt$. Conditionally on $Z^\infty$, let $\{s_j\}_{j\in\BB N}$ be an enumeration of the times when $Z$ has a downward jump and write $\Delta Z_{s_j}^\infty := \lim_{t\rta s_j^-} Z_t - Z_{s_j}$ be the magnitude of the corresponding jump. Also let $\{\chi_j\}_{j\in \BB N}$ be an i.i.d.\ sequence of random variables with the law
\eqb \label{eqn-area-density}
\frac{1}{\sqrt{2\pi} a^{5/2} } e^{-\frac{1}{2a}} \BB 1_{(a\geq 0)} \, da
\eqe
and for $t\geq 0$ define 
\eqbn
U_t^\infty := \sum_{j : s_j \leq t} (\Delta Z_{s_j}^\infty)^2 \chi_j .
\eqen

\begin{prop} \label{prop-length-area-conv}
For any peeling process of the UIHPQ$_{\op{S}}$, we have the joint convergence $(Z^{\infty,n} , U^{\infty,n}) \rta (Z^\infty ,  U^\infty)$ in law with respect to the local Skorokhod topology. 
\end{prop}

For the proof of Proposition~\ref{prop-length-area-conv}, we will use the following result from~\cite{curien-legall-peeling}, which is the quadrangulation version of~\cite[Proposition~9]{curien-legall-peeling} (c.f.~\cite[Section~6.2]{curien-legall-peeling}). 

\begin{lem} \label{lem-fb-area-asymp}
Let $\el\in \BB N$ and let $(Q^\el , \BB e^\el)$ be a free Boltzmann quadrangulation with simple boundary of perimeter $2\el$.
Then as $\el\rta\infty$, 
\eqb \label{eqn-fb-area-mean}
\BB E\left[ \# \mcl V(Q^\el\setminus \bdy Q^\el) \right] = \left(\frac92  + o_\el(1) \right)\el^2 
\eqe 
and
\eqb \label{eqn-fb-area-law}
\el^{-2} \# \mcl V(Q^\el\setminus \bdy Q^\el) \rta  \frac{9}{2}  \chi 
\eqe 
in law, where $\chi$ is a random variable with the law~\eqref{eqn-area-density}. 
\end{lem}

\begin{proof}[Proof of Proposition~\ref{prop-length-area-conv}]
This is proven using essentially the same argument as the proof of~\cite[Theorem~1]{curien-legall-peeling}, but we give the details for the sake of completeness.

By~\eqref{eqn-cover-tail} and the heavy-tailed central limit theorem (see, e.g.~\cite{js-limit-thm}), $Z^{\infty,n} \rta Z^\infty $ in law in the local Skorokhod topology.
Hence it remains to check the joint convergence. 

For $\ep > 0$, $n\in\BB N$, and $t\geq 0$ let
\eqbn
U_t^{\geq \ep , n} := \frac{1}{2^{1/3} n} \sum_{j = 1}^{\lfloor t n^{3/4} \rfloor} (A_j^\infty - A_{j-1}^\infty) \BB 1_{ (W_j^\infty - W_{j-1}^\infty) \leq    \ep b_* n^{1/2} } \quad \op{and} \quad 
U_t^{\geq \ep  } := \sum_{j : s_j \leq t}  (\Delta Z_{s_j}^\infty)^2 \chi_j \BB 1_{( \Delta Z_{s_j}^\infty  \leq \ep )}   .
\eqen

We first argue that for each $\ep > 0$, one has  
\eqb \label{eqn-big-jump-conv}
(Z^{\infty,n} , U^{\geq \ep , n}) \rta (Z^\infty , U^{\geq \ep}) 
\eqe 
in law in the local Skorokhod topology.  To see this, suppose that we have (using the Skorokhod representation theorem) coupled our UIHPQ$_{\op{S}}$ with $Z^\infty$ in such a way that $Z^{\infty,n} \rta Z^\infty$ a.s.\ in the local Skorokhod topology. 

Fix $T>0$ and $\zeta \in (0,1/4)$. We introduce a regularity event which will be used to get around the fact that a single peeled quadrilateral can disconnect two distinct free Boltzmann quadrangulations with simple boundary from $\infty$. For $n\in\BB N$ let $E^n = E^n(\zeta,T)$ be the event that the following is true: there does \emph{not} exist $j\in [1, Tn^{3/4}]_{\BB Z}$ such that the disconnected quadrangulation $\frk F_\infty(\ol Q_{j-1}^\infty , \dot e_j^\infty)$ has more than one connected component with perimeter at least $n^{1/4 + \zeta}$.  (Note that $1/4+\zeta < 1/2$ due to our choice of $\zeta$).  By~\eqref{eqn-uihpq-peel-prob-asymp}, for each $j\in \BB N_0$ and each $k_1,k_2 \in \BB N$ the probability that $\frk F(\ol Q_{j-1}^\infty , \dot e_j^\infty)$ has two connected components with perimeters $k_1$ and $k_2$ is bounded above by a universal constant times $k_1^{-5/2} k_2^{-5/2}$.  Summing this estimate over all $k_1,k_2 \geq n^{1/4+\zeta}$ and all $j\in [1,Tn^{3/4}]_{\BB Z}$ shows that $\BB P[E^n] = 1 - O_n(n^{-3\zeta}) $. 
  
For $n\in\BB N$ let $j_1^n <  \dots < j_{N^n}^n$ for $n\in\BB N$ be the times $j\in [0, T n^{3/4}]_{\BB Z}$ for which $W_j^\infty - W_{j-1}^\infty \leq  \ep b_* n^{1/2}$ and let $t_1 < \dots < t_k$ be the times $t \in [0,T]_{\BB Z}$ for which $\Delta Z_t \geq \ep$. By the local Skorokhod convergence $Z^{\infty,n} \rta Z^\infty$, we infer that a.s.\ $N^n = N$ for large enough $n\in\BB N$ and that $n^{-3/4} j_k^n \rta t_k$ for each $k\in [1,N]_{\BB Z}$. 
 
 For $k \in [1,N^n]_{\BB Z}$ let $Q_k^n$ be the larger simple-boundary component of the disconnected quadrangulation $\frk F_\infty(\ol Q_{j_k^n-1}^\infty , \dot e_{j_k^n}^\infty)$. 
If $E^n$ occurs, then the conditional law of $(Q_1^n ,\dots , Q_{N^n}^n)$ given $W^\infty$ and the perimeters of these quadrangulations is that of a collection of independent free Boltzmann quadrangulations with simple boundary. The increment $A_{j_k^n}^\infty - A_{j_k^n-1}^\infty$ is at least $\#\mcl V(Q_k^n \setminus \bdy Q_k^n)$ and is at most 4 plus the total number of vertices in $\frk F_\infty(\ol Q_{j_k^n-1}^\infty , \dot e_{j_k^n}^\infty)$. The above estimate for $\BB P[E^n]$ together with Lemma~\ref{lem-fb-area-asymp} (applied to the smaller component of the disconnected quadrangulation) shows that
\eqbn
\max_{k\in [1,N]_{\BB Z }} \frac{1}{  n} \left(  A_{j_k^n}^\infty - A_{j_k^n-1}^\infty - \#\mcl V(Q_k^n \setminus \bdy Q_k^n) \right) \rta 0
\eqen
in probability.   
From the convergence in law~\eqref{eqn-fb-area-law} in Lemma~\ref{lem-fb-area-asymp} and since $T>0$ can be made arbitrarily large, we obtain~\eqref{eqn-big-jump-conv}.

We next argue that for $\ep > 0$ and $t\geq 0$, 
\eqb \label{eqn-small-jump-mean}
\BB E\left[U_t^{\infty,n} - U_t^{\geq \ep , n}  \right] \preceq t \ep^{1/2} 
\eqe
with universal implicit constant. 
For each $j\in \BB N_0$, the conditional law of $A_j^\infty - A_{j-1}^\infty$ given $W^\infty$ is stochastically dominated by $2(W_j^\infty - W_{j-1}^\infty) + 4$ plus twice the number of interior vertices of a free Boltzmann quadrangulation with simple boundary of perimeter $W_j^\infty - W_{j-1}^\infty$ (the factor of 2 comes from the fact that $\frk F_\infty(\ol Q_{j-1}^\infty , \dot e_j^\infty)$ can have two connected components). 
By Lemma~\ref{lem-fb-area-asymp}, 
\eqbn
\BB E\left[ A_j^\infty - A_{j-1}^\infty \,|\, W^\infty \right] \preceq (W_j^\infty - W_{j-1}^\infty)^2 .
\eqen
By~\eqref{eqn-cover-tail}, we infer that
\eqbn
\BB E\left[ ( A_j^\infty - A_{j-1}^\infty ) \BB 1_{ (W_j^\infty - W_{j-1}^\infty) < b_* \ep n^{1/2} } \right] 
\preceq \sum_{k=1}^{\lfloor b_* \ep n^{1/2} \rfloor}  k^{-1/2} 
\preceq \ep^{1/2} n^{1/4} .
\eqen
Summing over all $j\in [1,tn^{3/4}]_{\BB Z}$ shows that~\eqref{eqn-small-jump-mean} holds.

It is easy to see that a.s.\ $U_t^\infty - U_t^{\geq \ep} \rta 0$ uniformly on compact intervals as $\ep \rta 0$ (c.f.\ the proof of~\cite[Theorem~1]{curien-legall-peeling}).
Hence the proposition statement follows from~\eqref{eqn-big-jump-conv} and~\eqref{eqn-small-jump-mean} upon sending $\ep \rta 0$. 
\end{proof}

\subsection{Comparing peeling processes on free Boltzmann quadrangulations and the UIHPQ$_{\op{S}}$}
\label{sec-rn-deriv}

Let $(Q^\infty ,\BB e^\infty)$ be a UIHPQ$_{\op{S}}$, let $\beta^\infty$ be its boundary path with $\beta^\infty(0) = \BB e^\infty$, and fix $\el \in \BB N$ and an initial edge set $\BB A \subset \beta^\infty([1,2\el-1]_{\BB Z})$.  Let $\{\dot Q_j^\infty\}_{j\in\BB N_0}$, $\{\ol Q_j^\infty\}_{j\in\BB N_0}$, $\{\dot e_j^\infty\}_{j\in\BB N_0}$, and $\{\mcl F_j^\infty\}_{j\in\BB N_0}$, respectively, be the clusters, unexplored quadrangulations, peeled edges, and filtration of a peeling process of the UIHPQ$_{\op{S}}$ targeted at $\infty$ which satisfies the following property: for each $j\in\BB N_0$, the peeled edge $\dot e_j$ belongs to $\BB A$ or $\bdy \dot Q_j^\infty \cap \bdy \ol Q_j^\infty$, so that we never peel at an edge of $\bdy Q^\infty$ which is not in $\BB A$. 
 
In this subsection we will compare unconditional law of the peeling clusters $\{\dot Q_j^\infty\}_{j\in \BB N_0}$ and the conditional law of these clusters given the event that  the boundary arc $\beta^\infty([1,2\el-1]_{\BB Z})$ is precisely the set of edges of $\bdy Q^\infty$ which are disconnected from $\infty$ by the peeled quadrilateral $\frk f(Q^\infty,\BB e^\infty)$.  Since the bounded complementary connected components of $\frk f(Q^\infty,\BB e^\infty)$ are free Boltzmann quadrangulations with simple boundary, the estimates of this subsection enable us to compare peeling processes on the UIHPQ$_{\op{S}}$ and peeling processes on free Boltzmann quadrangulations with simple boundary of perimeter $2\el$.  We remark that similar ideas to the ones appearing in this subsection (but in the case of triangulations) appear in~\cite[Section~4]{angel-uihpq-perc}.

For the statements of our estimates, we will use the following notation, which is illustrated in Figure~\ref{fig-peeling-rn}.

\begin{defn} \label{def-finite-stuff}
Let $\el\in\BB N$ and consider a peeling process and edge set $\BB A\subset \beta^\infty([1,2\el-1]_{\BB Z})$ as above.   
\begin{itemize}  
\item We write $I^\el$ for the smallest $j \in \BB N$ for which $ \dot Q_j^\infty $ contains an edge of $\mcl E(\bdy Q^\infty) \setminus \beta^\infty([1,2\el-1]_{\BB Z})$. 
\item With $\frk P(Q^\infty,\BB e^\infty)$ the peeling indicator from Section~\ref{sec-general-peeling}, we write $F^\el = \{\frk P( Q^\infty , \BB e^\infty) = (2\el-1 , \infty)\}$. Equivalently, $F^\el$ is the event that the terminal endpoint of $\beta^\infty(2\el-1)$ is a vertex of the peeled quadrilateral $\frk f( Q^\infty , \BB e^\infty)$, and this quadrilateral has one vertex which is not in $\bdy Q^\infty$, so that $\beta^\infty([1,2\el-1]_{\BB Z})$ is precisely the set of edges of $\bdy Q^\infty$ disconnected from $\infty$ by $\frk f(Q^\infty,\BB e^\infty)$.  
\end{itemize}    
\end{defn}

\begin{figure}[ht!]
 \begin{center}
\includegraphics[scale=1]{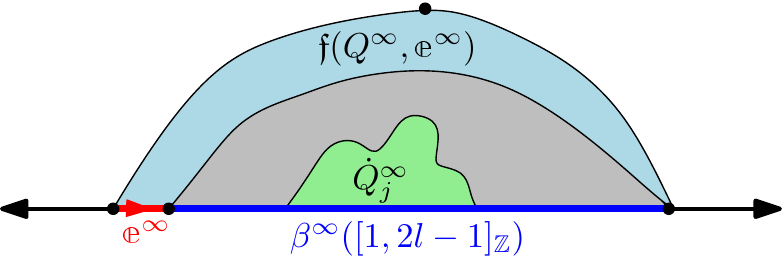} 
\caption{Illustration of the setup considered in Section~\ref{sec-rn-deriv}. If we condition on the event $F^\el$ that the blue arc $\beta^\infty([1,2\el-1]_{\BB Z})$ is precisely the arc disconnected from $\infty$ by the light blue peeled quadrilateral $\frk f(Q^\infty , \BB e^\infty)$ (and this arc lies on the boundary of a single complementary connected component of the light blue quadrilateral), then the bounded complementary connected component of this peeled quadrilateral (grey and green regions) is a free Boltzmann quadrangulation with simple boundary of perimeter $2\el$. We consider a peeling process on $Q^\infty$ which does not peel any edges of $\bdy Q^\infty \setminus \BB A$ for an edge set $\BB A \subset \beta^\infty([1,2\el-1]_{\BB Z})$ (a cluster of this peeling process is shown in light green). If we stop this process at a time before the first time $I^\el$ that it disconnects an edge of  $\bdy Q^\infty \setminus \beta^\infty([1,2\el-1]_{\BB Z})$ from $\infty$, then we can apply Bayes' rule to compute the Radon-Nikodym derivative of conditional law of the process given $F^\el$ with respect to its unconditional law (Lemma~\ref{lem-bubble-cond}). This allows us to compare peeling processes on the UIHPQ$_{\op{S}}$ to peeling processes on free Boltzmann quadrangulations with simple boundary.}\label{fig-peeling-rn}
\end{center}
\end{figure} 

We note that if $F^\el$ occurs, then $I^\el$ is the same as the first time at which the peeled quadrilateral $\frk f(Q^\infty, \BB e^\infty)$ belongs to $\dot Q_j^\infty$. Indeed, since we cannot peel any edges in $ \bdy Q^\el \setminus \BB A$, the cluster $\dot Q_{I^\el}$ must contain a path in the dual of $Q^\infty$ from a quadrilateral which contains an edge of $\BB A$ to a quadrilateral which contains an edge of $\bdy Q^\infty\setminus \beta^\infty([1,2\el-1]_{\BB Z})$, so must contain $\frk f(Q^\infty,\BB e^\infty)$ if $F^\el$ occurs. On the other hand, $\frk f(Q^\infty,\BB e^\infty)$ contains the edge $\BB e^\infty \in  \bdy Q^\infty\setminus \beta^\infty([1,2\el-1]_{\BB Z})$, so cannot belong to $\dot Q_j^\infty$ for $j < I^\el$.

By the Markov property of peeling if we condition on $F^\el$ then the conditional law of the disconnected quadrangulation $Q := \frk F(Q^\infty,\BB e^\infty)$ is that of a free Boltzmann quadrangulation with simple boundary of perimeter $2\el$ and our given peeling process run up to time $I^\el$ is a peeling process of $Q$. 
Hence comparing peeling processes of $Q$ and $Q^\infty$ is equivalent to comparing the conditional law given $I^\el$ of our peeling process run up to time $I^\el$ to its unconditional law. 
The main tool which we will use for this purpose is the following elementary lemma.

\begin{lem} \label{lem-bubble-cond}
Suppose we are in the setting described just above.
Let $\iota$ be a stopping time for $\{\mcl F_{j}^\infty \}_{j\in\BB N}$ which is less than $I^\el$ with positive probability. Then the conditional law of $\{ \dot Q_{j}^\el , \frk P(\ol Q_{j-1}^\infty , \dot e_j^\infty )   \}_{j\in [1,\iota]_{\BB Z}}$ given $F^\el$ restricted to the event $\{\iota < I^\el\}$ is absolutely continuous with respect to the unconditional law of this same peeling process, with Radon-Nikodym derivative given by
\eqb \label{eqn-bubble-cond}
 54^{ -W_\iota^\infty  /2}  \frac{\frk Z(W_\iota^\infty +  2\el)}{ \frk Z( 2\el ) }   \BB 1_{(\iota < I^\el) } 
= (1 + o(1))  \left( \frac{W_\iota^\infty}{2\el} + 1 \right)^{-5/2}  \BB 1_{(\iota < I^\el) } 
\eqe  
where here $W^\infty$ is the net boundary length process from Definition~\ref{def-bdy-process}, $\frk Z$ is the free Boltzmann partition function as in~\eqref{eqn-fb-partition}, and the $o(1)$ tends to zero as $\el \wedge (W_\iota^\infty + 2\el)$ tends to $\infty$, at a deterministic rate. 
\end{lem}
\begin{proof} 
Write $S_\iota^\infty = \{ (\dot Q_{j}^\infty  , \frk P(Q_{j-1}^\infty  , \dot e_j^\infty  ) )    \}_{j\in [1,\iota]_{\BB Z}}$ and let $\frk S^\infty $ be a realization of $S_\iota^\infty $ for which $\iota < I^\el $ (equivalently, the realization of $\dot Q_\iota^\infty $ does not contain any edges outside of $\beta^\infty([1,2\el-1]_{\BB Z})$). By Bayes' rule, 
\eqb \label{eqn-bubble-cond-bayes}
\BB P\left[S_\iota^\infty  = \frk S^\infty  \,|\, F^\el   \right] 
= \frac{  \BB P\left[F^\el \,|\, S_\iota^\infty  = \frk S^\infty    \right] \BB P\left[  S_\iota^\infty  = \frk S^\infty    \right] }{  \BB P\left[ F^\el \right]  } .
\eqe 
By~\eqref{eqn-uihpq-peel-prob}, 
\eqb  \label{eqn-bubble-prob}
\BB P\left[ F^\el  \right] =   \frac{1}{12} 54^{1-\el } \frk Z( 2\el )  .
\eqe 
By the Markov property of peeling, if we condition on $\{S_\iota^\infty  = \frk S^\infty   \}$, then the conditional law of the unexplored quadrangulation $(\ol Q_\iota^\infty  , \BB e^\infty )$, with the original root edge, is that of a UIHPQ$_{\op{S}}$. Note that since $\iota < I^\el$, the edge $\BB e^\infty$ and the terminal endpoint of $\beta^\infty(2\el-1)$ both belong to $\bdy \ol Q_\iota^\infty$. 
The distance along~$\bdy \ol Q_\iota^\infty$ from~$\BB e^\infty$ to the terminal endpoint of~$\beta^\infty(2\el-1) $ is equal to $W_\iota^\infty + 2\el-1$.
By~\eqref{eqn-uihpq-peel-prob}, 
\eqb \label{eqn-bubble-cond-prob}
\BB P\left[F^\el \,|\, S_\iota^\infty = \frk S^\infty  \right] 
= \BB P\left[ \frk P\left( \ol Q_\iota^\infty , \BB e^\infty  \right) = (W_\iota^\infty+ 2\el-1 , \infty) \,|\, S_\iota^\infty = \frk S^\infty    \right]  
=  \frac{1}{12} 54^{ 1-\el    - W_\iota^\infty/ 2} \frk Z(W_\iota^\infty + 2\el )  .
\eqe 
with universal implicit constants. We obtain the first formula in~\eqref{eqn-bubble-cond} by combining~\eqref{eqn-bubble-cond-bayes},~\eqref{eqn-bubble-prob}, and~\eqref{eqn-bubble-cond-prob}. The second formula follows from Stirling's approximation (c.f.~\eqref{eqn-stirling-asymp}).
\end{proof}

Lemma~\ref{lem-bubble-cond} will be our main tool for estimating the probabilities of events associated with peeling processes on a free Boltzmann quadrangulation with simple boundary. 
However, for the sake of completeness we will also record a slightly different estimate with a deterministic Radon-Nikodym derivative which will be needed in~\cite{gwynne-miller-perc}. 
The reader who only wants to see the proof of Theorem~\ref{thm-fb-ghpu} can skip the remainder of this subsection. 
 
Let $A = [a_L , a_R]_{\BB Z}$ be a discrete interval which is contained in $[1,2\el-1]_{\BB Z}$ and such that $\beta^\infty(A)$ contains the initial edge set $\BB A$ and let 
\eqbn
\iota(A) := \min\left\{ j\in\BB N : \mcl E(\dot Q_j^\infty \cap \bdy Q^\infty) \not\subset \beta^\infty(A) \right\} .
\eqen
We note that if $F^\el$ occurs, then necessarily $\iota(A) < I^\el$. 

\begin{lem}  \label{lem-full-bubble-cond} 
For any event $E$ belonging to the $\sigma$-algebra $\mcl F_{\iota(A )-1}^\infty \vee \sigma(\iota(A ))$, 
\eqb \label{eqn-full-bubble-cond}
\BB P\left[ E \,|\, F^\el \right] \preceq \left( \frac{  \el }{ (2\el - a_R ) \wedge a_L } \right)^{5/2}  \BB P[E]
\eqe 
with universal implicit constant.  
\end{lem} 

In the statement of Lemma~\ref{lem-full-bubble-cond}, it is crucial that $E$ does \emph{not} depend on the peeling step at time $\iota(A)$.  The idea of the proof of Lemma~\ref{lem-full-bubble-cond} is to prove deterministic estimates for the conditional law of $W_{\iota(A)}^\infty$ given $\mcl F_{\iota(A )-1}^\infty \vee \sigma(\iota(A ))$, which will in turn lead to estimates for the conditional expectation of the Radon-Nikodym derivative appearing in Lemma~\ref{lem-bubble-cond} at time $\iota=\iota(A)$ given this $\sigma$-algebra.

\begin{lem}
\label{lem-1pt-overshoot} 
Suppose we are in the setting of Lemma~\ref{lem-full-bubble-cond}.  Also let $K_L$ (resp.\ $K_R$) be the largest $k\in\BB N_0$ for which $\beta^\infty(a_L   - k)$ (resp.\ $\beta^\infty(a_R  + k)$) belongs to $\dot Q_{\iota(A) }^\infty$, or~$0$ if no such $k\in\BB N_0$ exists (note that either $K_L$ or $K_R$ must be positive). Then for $k_L , k_R  \in \BB N$, 
\eqbn
\BB P\left[  K_L = k_L   ,\, K_R =k_R \,|\,   \mcl F_{\iota(A)-1}^\infty \vee \sigma(\iota(A) ) \right]  \preceq (k_L + k_R)^{3/2} k_L^{-5/2} k_R^{-5/2}
\eqen
with universal implicit constant. 
\end{lem}
\begin{proof}
Let $\ell_L $ (resp.\ $\ell_R$) be the $\bdy \ol Q_{\iota(A)-1}^\infty$-graph distance from $\dot e_{\iota(A)}^\infty$ to $\beta^\infty(a_L   ) $ (resp.\ $\beta^\infty(a_R)$). Note that these quantities are well-defined since $\beta^\infty(a_L   ) $ and $\beta^\infty(a_R)$ belong to $\bdy \ol Q_{\iota(A)-1}^\infty$. 

Let $\frk i \in \BB N$. By the Markov property of peeling, if we condition on $\mcl F_{\frk i - 1}^\infty$ and the event $\{\iota(A) = \frk i\}$, then the conditional law of the peeled quadrilateral $\frk f(\ol Q_{\frk i -1}^\infty , \dot e_{\frk i}^\infty )$ is the same as its conditional law given that it covers up at least $\ell_L $ edges of $\bdy \ol Q_{\frk i-1}^\infty$ to the left of $\dot e_{\frk i}^\infty$ or at least $\ell_R $ edges of $\bdy \ol Q_{\frk i-1}^\infty$ to the right or $\dot e_{\frk i}^\infty$. By~\eqref{eqn-uihpq-peel-prob-asymp}, the probability that this is the case is proportional to $\left( \ell_L \wedge \ell_R  \right)^{-3/2}$.
By combining this with~\eqref{eqn-uihpq-peel-prob-asymp}, we obtain
\eqb \label{eqn-1pt-overshoot-prob}
\BB P\left[  K_L = k_L,\, K_R =k_R \,|\, \iota(A) = \frk i , \, \mcl F_{\frk i-1}^\infty  \right]  
\preceq \left( \ell_L \wedge \ell_R  \right)^{ 3/2} (k_L + \ell_L)^{-5/2} (k_R + \ell_R )^{-5/2} .
\eqe 
Suppose without loss of generality that $\ell_L \leq \ell_R$. Then the right side of~\eqref{eqn-1pt-overshoot-prob} is at most $\ell_L^{3/2} (k_L + \ell_L)^{-5/2} (k_R + \ell_L)^{-5/2}$. This quantity is maximized over all possible values of $\ell_L$ at $\ell_L = \frac17 \left(  \sqrt{k_L^2 + 23 k_L k_R + k_R^2} - k_L- k_R \right)$ where it equals 
\eqbn
\frac{344\sqrt 7 \left(   \sqrt{k_L^2 + 23 k_L k_R + k_R^2} - k_L- k_R \right)^{3/2} }{ \left(   6 k_R - k_L +  \sqrt{k_L^2 + 23 k_L k_R + k_R^2} \right)^{5/2}  \left(   6 k_L - k_R +  \sqrt{k_L^2 + 23 k_L k_R + k_R^2} \right)^{5/2} }  \preceq (k_L + k_R)^{3/2} k_L^{-5/2} k_R^{-5/2} . \qedhere
\eqen
\end{proof}

From Lemma~\ref{lem-full-bubble-cond}, we deduce estimates for the conditional law given $F^\el$ of the left and right overshoot quantities $K_L$ and $K_R$ appearing in Lemma~\ref{lem-1pt-overshoot}.  For the statement of the estimates, we recall that on $F^\el$, we have $\iota(A)  < I^\el$, so $K_L\leq a_L  $ and $K_R \leq 2\el-1 - a_R$.

\begin{lem} \label{lem-1pt-overshoot-finite}
Suppose we are in the setting of Lemma~\ref{lem-1pt-overshoot-finite} and let $E\in \mcl F_{\iota(A )-1}^\infty \vee \sigma(\iota(A ))$. For $k_L \in [1,a_L  ]_{\BB Z}$ and $k_R \in [1, 2\el-1 -a_R]_{\BB Z}$, 
\eqb \label{eqn-1pt-overshoot-finite}
\BB P\left[  E\cap \left\{ K_L = k_L,\, K_R =k_R \right\}   \,|\,    F^\el  \right]  \preceq    \frac{\el^{5/2} (k_L + k_R)^{3/2} }{k_L^{5/2} k_R^{5/2} (2\el - a_R + a_L    - k_L -k_R  )^{5/2} } \BB P[E] .
\eqe 
with universal implicit constant. In particular,  
\eqb \label{eqn-1pt-overshoot-finite'}
\BB P\left[  E\cap \left\{ K_L \geq \frac12 a_L   \,\op{and} \, K_R  \geq \frac12 (2\el -a_R) \right\}    \,|\,    F^\el \right]  \preceq   \frac{\el^{5/2}  (2\el - a_R + a_L  )^{3/2}  }{  a_L  ^{ 5/2} (2\el -a_R)^{ 5/2}  } \BB P[E] 
\eqe 
with universal implicit constant. 
\end{lem} 
\begin{proof}
To make the symmetry in our formulas more apparent, we define
\eqb \label{eqn-overshoot-abbrev}
m_L := a_L   \quad \op{and} \quad m_R := 2\el - 1 - a_R . 
\eqe 
On the event $\{k_L = k_L ,\, K_R = k_R \}$, the number of covered edges at time $\iota(A)$ (Definition~\ref{def-bdy-process}) satisfies 
\eqbn
 Y_{\iota(A)}^\infty  \leq  a_R - a_L   + k_L + k_R =    2\el-1 - (m_L  +  m_R  - k_L -k_R) .
 \eqen 
The number of exposed edges always satisfies $X_{\iota(A) }^\infty \geq 1$, so 
\eqb \label{eqn-full-bubble-length-lower}
W_{\iota(A)}^\infty \geq   m_L  +  m_R  - k_L -k_R - 2\el  .
\eqe 

By applying~\eqref{eqn-full-bubble-length-lower} to bound the Radon-Nikodym derivative from Lemma~\ref{lem-bubble-cond}, we find that for $E$ as in the statement of the lemma,  
\eqb \label{eqn-apply-rn}
\BB P\left[  E\cap \{ K_L = k_L ,\, K_R = k_R \} \,|\, F^\el \right]
\preceq \left( \frac{2\el }{ m_L  +  m_R  - k_L -k_R + 1}   \right)^{ 5/2}  \BB P\left[ E\cap \{ K_L = k_L ,\, K_R = k_R \} \right]  .
\eqe
By Lemma~\ref{lem-1pt-overshoot} and since $E \in \mcl F_{\iota(A) -1}^\infty \vee \sigma(\iota(A) )$, 
\eqb \label{eqn-apply-overshoot}
 \BB P\left[ E\cap \{ K_L = k_L ,\, K_R = k_R \} \right]  \preceq (k_L + k_R)^{3/2} k_L^{-5/2} k_R^{-5/2} \BB P[E] .
\eqe 
Combining~\eqref{eqn-apply-rn} and~\eqref{eqn-apply-overshoot} yields~\eqref{eqn-1pt-overshoot-finite}. We obtain~\eqref{eqn-1pt-overshoot-finite'} by summing~\eqref{eqn-1pt-overshoot-finite} over all $(k_L,k_R) \in [m_L/2,m_L]_{\BB Z}\times [m_R/2,m_R]_{\BB Z}$.
\end{proof}

\begin{proof}[Proof of Lemma~\ref{lem-full-bubble-cond}]
Let $K_L$ and $K_R$ be as in Lemma~\ref{lem-1pt-overshoot} and let
\eqbn
G := \left\{ K_L \geq \frac12 a_L   \,\op{and} \, K_R  \geq \frac12 (2\el -1 -a_R) \right\} .
\eqen
For an event $E$ as in the statement of the lemma, we have by Lemma~\ref{lem-1pt-overshoot-finite} that
\eqb \label{eqn-full-bubble-cond-error}
\BB P\left[ E\cap G \,|\, F^\el \right] \preceq   \frac{\el^{5/2}  (2\el - a_R + a_L  )^{3/2}  }{  a_L  ^{ 5/2} (2\el-a_R)^{ 5/2}  } \BB P[E] \preceq    \left( \frac{ \el }{ (2\el - a_R) \wedge a_L   } \right)^{5/2} \BB P[E] .
\eqe 
On the other hand, if $G^c$ occurs then either $K_L \leq \frac12 a_L  $ or $K_R \leq \frac12 (2\el-a_R)$ so by~\eqref{eqn-full-bubble-length-lower}, $W_{\iota(A)}^\infty \geq \frac12 (a_L   \wedge ( 2\el - a_R)) - 2\el  $. 
By Lemma~\ref{lem-full-bubble-cond},
\eqb \label{eqn-full-bubble-cond-main}
 \BB P\left[ E\cap G^c\,|\, F^\el \right] \preceq  \left( \frac{ \el }{   (2\el - a_R) \wedge a_L    } \right)^{5/2}  \BB P\left[E \cap G^c \right] .
\eqe 
Summing~\eqref{eqn-full-bubble-cond-error} and~\eqref{eqn-full-bubble-cond-main} yields~\eqref{eqn-full-bubble-cond}.  
\end{proof}

\section{Peeling by layers}
\label{sec-peeling-by-layers}

In this section we introduce the peeling-by-layers process of a free Boltzmann quadrangulation with simple boundary, which is the only peeling process we will consider in the sequel. This peeling process is an analog for quadrangulations with boundary of the peeling-by-layers process for the UIPQ studied in~\cite{curien-legall-peeling}.  (A more complicated two-sided version of this peeling process for a pair of UIHPQ$_{\op{S}}$'s glued together along their boundaries appears in~\cite{gwynne-miller-saw,caraceni-curien-saw}.)

We will define the peeling-by-layers process in Section~\ref{sec-pbl-def}. In Section~\ref{sec-pbl-estimate}, we prove some estimates for the peeling-by-layers process on the UIHPQ$_{\op{S}}$, which can be transferred to estimates for free Boltzmann quadrangulations with simple boundary using the results of Section~\ref{sec-rn-deriv}.  In Section~\ref{sec-coupling}, we explain how the estimates of this paper enable us to couple a UIHPQ$_{\op{S}}$ and a free Boltzmann quadrangulation with finite simple boundary in such a way that they agree in a neighborhood of the root edge with high probability.

\subsection{The peeling-by-layers process}
\label{sec-pbl-def}

Let $\el\in\BB N \cup\{\infty\}$ and let $(Q,\BB e)$ be a free Boltzmann quadrangulation with simple boundary of perimeter $2\el$ (so that $(Q,\BB e)$ is a UIHPQ$_{\op{S}}$ if $\el=\infty$).  Also let $\BB e_* \in \mcl E(\bdy Q)$; we also allow $\BB e_* = \infty$ in the case when $\el = \infty$.  Fix a finite connected initial edge set $\BB A \subset \mcl E(\bdy Q)$ which does not contain $\BB e_*$, in a manner which depends only on $(\bdy Q , \BB e , \BB e_*)$

We will inductively define a peeling process for $Q$ targeted at $\BB e_*$ called the \emph{peeling-by-layers process started from $\BB A$}.  Let $\ol Q_0 = Q $, let $\dot Q_0$ be the quadrangulation with no internal faces whose edge set is $\BB A$.

Inductively, suppose $j\in\BB N$ and $\ol Q_i$ and $\dot Q_i$ have been defined for $i \in [0,j-1]_{\BB Z}$.  If $\ol Q_{j-1} =\emptyset$, we set $\ol Q_j = \emptyset$ and $\dot Q_j = Q$.  Otherwise, let $\dot e_j$ be an edge in $\mcl E( \bdy \ol Q_{j-1} \cap \bdy \dot Q_{j-1} )$ which lies at minimal $\dot Q_{j-1}$-graph distance from $\BB A$, chosen in a manner which depends only on $\bdy \dot Q_{j-1}$ and $\bdy \ol Q_{j-1} \cap \bdy \dot Q_{j-1}$.  Recalling the notation of Section~\ref{sec-general-peeling}, we peel $ \ol Q_{j-1}$ at $\dot e_j$ to obtain the quadrilateral $\frk f( \ol Q_{j-1} , \dot e_j)$ and the planar map $\frk F_\infty( \ol Q_{j-1} ,  \dot e_j)$ which it disconnects from $\BB e_*$ in $\ol Q_{j-1}$. Define 
\alb
 \dot Q_j := \dot Q_{j-1} \cup  \frk f( \ol Q_{j-1}  ,  \dot e_j) \cup \frk F_{\BB e_*}  ( \ol Q_{j-1}  ,  \dot e_j)   \quad \op{and} \quad
  \ol Q_j := \op{Peel}_{\BB e_*} \left( \ol Q_{j-1}  ,  \dot e_j  \right) .
\ale
By induction $\ol Q_j$ is a quadrangulation  with simple boundary, $\ol Q_j$ and $\dot Q_j$ intersect only along their boundaries, and $Q  = \ol Q_j \cup \dot Q_j$.

Define the peeling filtration 
\eqb \label{eqn-peel-filtration}
\mcl F_j := \sigma\left( \dot Q_i,\, \frk P(\ol Q_{i-1} , \dot e_i)  \,:\, i \in [1,j]_{\BB Z} \right) ,\quad \forall j\in\BB N_0 ,
\eqe 
where here $\frk P(\cdot,\cdot)$ is the peeling indicator variable from Section~\ref{sec-general-peeling}.  Note that $\dot Q_j$ and $ \dot e_{j+1}$ are $\mcl F_j$-measurable for $j\in\BB N_0$. 

Also define the boundary length processes $\{X_j\}_{j\in\BB N_0}$, $\{Y_j\}_{j\in\BB N_0}$, and $\{W_j\}_{j\in\BB N_0}$ as in Definition~\ref{def-bdy-process} for the peeling-by-layers process.

We record for reference what the Markov property of peeling tells us in the setting of this subsection. 
   
\begin{lem} \label{lem-peel-law}
Let $\iota$ be an a.s.\ finite stopping time for the filtration $\{\mcl F_j\}_{j\in\BB N_0}$ from~\eqref{eqn-peel-filtration}.  The conditional law of $( \ol Q_\iota ,  \dot e_{\iota+1})$ given $\mcl F_\iota$ is that of a free Boltzmann quadrangulation with simple boundary with perimeter $2\el + W_\iota$, where $W$ is the net boundary length process from Definition~\ref{def-bdy-process}. 
\end{lem}
\begin{proof}
This is immediate from the Markov property of peeling.
\end{proof}

For $r\in\BB N_0$, let 
\eqb \label{eqn-peel-ball-time}
J_r := \min\left\{ j\in \BB N_0 : \op{dist}\left( \dot e_{j+1} , \BB A ; \dot Q_{j } \right) \geq r \right\} = \min\left\{ j \in \BB N_0 : \op{dist}\left( \bdy \ol Q_j \cap \bdy \dot Q_j , \BB A ; \dot Q_{j } \right) \geq r \right\} ,
\eqe
so that $J_r$ for $r\in\BB N_0$ is a stopping time for $\{\mcl F_j\}_{j\in\BB N_0}$.  The following lemma is the main reason for our interest in the peeling-by-layers process. 

\begin{lem} \label{lem-peel-ball}
For $r\in\BB N_0$, let $B_r^\bullet(\BB A ; Q )$ be the filled graph metric ball of radius $r$ centered at $\BB A$, i.e.\ the subgraph of $Q $ which is the union of $B_r(\BB A ; Q )$ and the set of all vertices and edges which it disconnects from~$\BB e_*$ (or $B_r^\bullet(\BB A ; Q) = Q$ if $\BB e_* \in B_r(\BB A ;Q)$).  For each $r\in\BB N_0$,
\eqb \label{eqn-peel-ball}
B_r^\bullet\left( \BB A ; Q  \right) \subset  \dot Q_{J_r}   \subset B_{r+2}^\bullet\left( \BB A ; Q \right) . 
\eqe 
\end{lem}
\begin{proof} 
It suffices to show inclusion of the vertex sets of the graphs in~\eqref{eqn-peel-ball}, since an edge in either of these graphs is the same as an edge of $Q $ whose endpoints are both in the vertex set of the graph.  We proceed by induction on $r$. The base case $r = 0$ (in which case $J_r = 0$) is true by definition. Now suppose $r\in\BB N$ and~\eqref{eqn-peel-ball} holds with $r-1$ in place of $r$. 

If we are given a vertex $v$ of $B_r\left( \BB A ; Q  \right) \setminus \mcl V(\dot Q_{J_{r-1}})$, then there is a $w\in B_{r-1}\left( \BB A ; Q  \right)$ with $\op{dist}\left(w , \BB A ; Q  \right) = r-1$. By the inductive hypothesis, $w$ belongs to $\mcl V(\bdy \ol Q_{J_{r-1}} \cap \bdy \dot Q_{J_{r-1}})$. By the definition~\eqref{eqn-peel-ball-time} of $J_r$, we have $w\notin \mcl V(\bdy \ol Q_{J_r} \cap \bdy \dot Q_{J_r})$ so we must have $v\in \mcl V(\dot Q_{J_r})$. Hence $B_r(\BB A ; Q ) \subset \dot Q_{J_r}$. Since $\dot Q_{J_r}$ contains every vertex or edge which it disconnects from $\infty$, we obtain the first inclusion in~\eqref{eqn-peel-ball}. 

For the second inclusion, we observe that the definition of $J_r$ implies that each of the peeled quadrilaterals $\frk f(\ol Q_{j-1} ,  \dot e_j)$ for $j\in [J_{r-1}+1 , J_r]_{\BB Z}$ has a vertex which lies at $\dot Q_{j-1}$-graph distance at most $r-1$ from $\BB A$. Hence each vertex of this quadrilateral lies at $\dot Q_{j-1}$-graph distance at most $r+2$ from $\BB A$. Every vertex in $\dot Q_{J_r}$ is either contained in $\dot Q_{J_{r-1}}$, incident to one of the quadrilaterals $\frk f(\ol Q_{j-1} ,  \dot e_j)$ for $j\in [J_{r-1}+1 , J_r]_{\BB Z}$, or disconnected from $\infty$ in $Q $ by the union of $\dot Q_{J_{r-1}}$ and these quadrilaterals.  By combining these observations with the inductive hypothesis that $\dot Q_{J_{r-1}} \subset B_{r+1}^\bullet(\BB A ; Q )$, we obtain the second inclusion in~\eqref{eqn-peel-ball}. 
\end{proof}

\subsection{Estimates for the peeling-by-layers process on the UIHPQ$_{\op{S}}$}
\label{sec-pbl-estimate}

For the proof of Theorem~\ref{thm-fb-ghpu}, we will require several estimates for the peeling-by-layers process introduced in the preceding subsection. Throughout this subsection, we consider only the case of the UIHPQ$_{\op{S}}$ (i.e., $\el = \infty$) and we target our process at $\BB e_* = \infty$ (we will eventually transfer these estimates to the case of free Boltzmann quadrangulations with finite boundary using Lemma~\ref{lem-bubble-cond}). 
We use the notation of Section~\ref{sec-pbl-def} but include an additional superscript $\infty$ to denote the UIHPQ$_{\op{S}}$ case.

Our first estimate is a bound for the number of covered edges in the radius-$r$ peeling-by-layers cluster, which is essentially proven in~\cite{gwynne-miller-saw}. 

\begin{lem} \label{lem-uihpq-ball-length}  
Let $\BB A \subset \mcl E(\bdy Q^\infty)$ and define the times $\{J_r^\infty\}_{r\in \BB N_0}$ when the peeling-by-layers clusters reach radius $r$, as in~\eqref{eqn-peel-ball-time}. Also let $Y^\infty$ be the covered boundary length process, as in Definition~\ref{def-bdy-process}. For each $p\in (1,3/2)$ and each $r\in\BB N_0$, it holds that
\eqbn
\BB E\left[\left( Y_{J_r^\infty}^\infty \right)^p \right] \preceq \left( r + (\#\BB A)^{1/2} \right)^{2p} 
\eqen
with implicit constant depending only on $p$. 
\end{lem}
\begin{proof}
For convenience we will deduce the lemma from~\cite[Proposition~5.1]{gwynne-miller-saw}, which gives a moment estimate for the analog of the peeling-by-layers process in the planar map obtained by gluing together two independent UIHPQ$_{\op{S}}$'s along their boundary. The statement of the lemma can also be obtained directly using an argument which is similar to but slightly simpler than the proof of~\cite[Proposition~5.1]{gwynne-miller-saw}. 

By Lemma~\ref{lem-peel-ball}, the peeling-by-layers cluster $\dot Q_{J_r^\infty}^\infty$ is contained in the radius-$(r+2)$ filled metric ball in~$Q^\infty$ centered at $\BB A$ (with respect to $\infty$).  If we glue $Q^\infty$ to another independent UIHPQ$_{\op{S}}$ along their positive boundaries to obtain an infinite quadrangulation with boundary $Q_{\op{zip}}$, then the radius-$(r+2)$ filled metric ball centered at $\BB A$ in $Q^\infty$ is contained in the radius-$(r+2)$ filled metric ball centered at $\BB A$ in $Q_{\op{zip}}$.  By~\cite[Lemma~4.3]{gwynne-miller-saw}, the radius-$(r+2)$ glued peeling cluster for $Q_{\op{zip}}$ started from the initial edge set $\BB A$ (as defined in~\cite[Section~4.1]{gwynne-miller-saw}) contains this latter filled metric ball.  Hence the statement of the lemma follows from~\cite[Proposition~5.1]{gwynne-miller-saw}.
\end{proof}

We next prove an estimate which enables us to compare filled metric balls in the UIHPQ$_{\op{S}}$, which by Lemma~\ref{lem-peel-ball} are essentially the same thing as peeling-by-layers clusters, to ordinary filled metric balls. 

\begin{lem} \label{lem-filled-ball-contain}
For each $\ep \in (0,1)$, there exists $R = R(\ep) >1$ such that the following is true for each $r\in\BB N$. Let $B_r^\bullet(\BB e^\infty ; Q^\infty)$ be the filled metric ball of radius $r$ centered at the root edge, i.e.\ the subgraph of $Q^\infty$ consisting of $B_r (\BB e^\infty ; Q^\infty)$ and all of the vertices and edges which it disconnects from $\infty$. Then 
\eqb \label{eqn-filled-ball-contain}
\BB P\left[ B_r^\bullet(\BB e^\infty ; Q^\infty) \subset B_{R r}(\BB e^\infty ; Q^\infty) \right] \geq 1-\ep.
\eqe 
\end{lem}

The statement of Lemma~\ref{lem-filled-ball-contain} is not immediate from the scaling limit result~\cite[Theorem~1.12]{gwynne-miller-uihpq} for the UIHPQ$_{\op{S}}$ toward the Brownian half-plane in the local GHPU topology since filled metric balls are not a continuous functional with respect to the local GHPU topology. 
We will still use~\cite[Theorem~1.12]{gwynne-miller-uihpq} to prove Lemma~\ref{lem-filled-ball-contain}, but the argument is not as straightforward as one might expect. 
 
\begin{proof}[Proof of Lemma~\ref{lem-filled-ball-contain}]
Let $\beta^\infty$ be the boundary path of $Q^\infty$ satisfying $\beta^\infty(0) = \BB e^\infty$. 
By Lemma~\ref{lem-uihpq-ball-length}, there exists $T = T(\ep ) > 0$ such that for each $r\in \BB N$, it holds with probability at least $1-\ep/2$ that $Y_{J_{r }^\infty}^\infty \leq T r^2$ which by Lemma~\ref{lem-peel-ball} implies that
\eqb \label{eqn-path-disjoint}
\beta^\infty([T r^2 ,\infty)_{\BB Z} )  \cap B_{4r}^\bullet \left( \BB e^\infty ; Q^\infty \right) = \emptyset .
\eqe  

We will now apply the scaling limit result~\cite[Theorem~1.12]{gwynne-miller-uihpq}. 
Let $(H^\infty , d^\infty  ,\mu^\infty,\xi^\infty  )$ be a Brownian half-plane equipped with its natural metric, area measure, and boundary path, with $\xi^\infty(0)$ the root vertex.
For $\rho  > 0$, let $B_\rho^\bullet(\xi^\infty(0) ; d^\infty )$ be the continuum filled metric ball, i.e.\ the union of $B_\rho (\xi^\infty(0) ; d^\infty)$ and the set of points in $H^\infty$ which are disconnected from $\infty $ by $B_\rho^\bullet(\xi^\infty(0) ; d^\infty)$. 
Since the Brownian half-plane has the topology of the ordinary half-plane (see, e.g.,~\cite[Corollary 3.8]{bmr-uihpq}), it is one-ended. Consequently, each $B_\rho^\bullet(\xi^\infty(0) ; d^\infty)$ has finite diameter.

After possibly increasing the parameter $T$ in~\eqref{eqn-path-disjoint}, we can find $R  = R(\ep) > 4$ and $S = S(\ep) > R$ such that with probability at least $1-\ep/4$, the following hold.
\begin{enumerate}
\item $B_2^\bullet\left( \xi^\infty (0) ; d^\infty \right) \subset B_{R/2}\left( \xi^\infty (0) ; d^\infty \right)$. 
\item $\xi^\infty (T) \in B_R \left( \xi^\infty (0) ; d^\infty\right) \setminus B_4^\bullet\left( \xi^\infty (0) ; d^\infty \right)$.
\item The diameter of $ B_{3R}(\xi^\infty (0) ; d^\infty) \setminus B_3(\xi^\infty (0) ; d^\infty)$ with respect to the internal metric of $d^\infty$ on $H^\infty \setminus B_2^\bullet\left(\xi^\infty (0) ; d^\infty\right)$ is at most $S $. 
\end{enumerate}
If this is the case, then for each $x \in B_{3R}(\xi^\infty (0) ; d^\infty) \setminus B_3(\xi^\infty (0) ; d^\infty)$, there is a path of $d^\infty$-length at most $S$ from $x$ to $\xi^\infty (T)$ which does not enter $B_2(\xi^\infty (0) ; d^\infty)$. Hence for each such $x$, there exist $k+1 \leq 100S$ points $x_0,\dots ,x_k \in H^\infty \setminus B_3(\xi^\infty (0) ; d^\infty)$ such that $d^\infty(x_0,x) \leq 1/100$, $d^\infty(x_j  , x_{j-1}) \leq 1/100$ for each $j\in [1,k]_{\BB Z}$, and $d^\infty(x_k , \xi^\infty (T)) \leq 1/100$. 

This latter condition behaves well under Gromov-Hausdorff limits. In particular, GHPU convergence of the UIHPQ$_{\op{S}}$ to the Brownian half-plane~\cite[Theorem~1.12]{gwynne-miller-uihpq} implies that for large enough $r\in\BB N$, it holds with probability at least $1-\ep/2$ that the following is true. For each vertex $v$ in $B_{2R r}(\BB e^\infty ; Q^\infty) \setminus B_{4r}(\BB e^\infty ; Q^\infty)$, there exist $k \leq 2S$ vertices $v_0,\dots ,v_k \in \mcl V\left( Q^\infty \setminus B_{2r}\left( \BB e^\infty ; Q^\infty\right) \right) $ such that the $Q^\infty$-graph distances between each of $v_0$ and $v$, $ v_j $ and $v_{j-1}$ for $j\in [1,k]_{\BB Z}$, and $v_k$ and $\beta^\infty(\lceil T r^2 \rceil)$ are at most $r/50$. Let $E_r$ be the event that this is the case and that~\eqref{eqn-path-disjoint} holds, so that $\BB P[E_r] \geq 1 -\ep$ for large enough $r$.

We claim that if $E_r$ occurs, then the event in~\eqref{eqn-filled-ball-contain} holds. Indeed, suppose to the contrary that there is a vertex $v \in \mcl V\left( B_r^\bullet(\BB e^\infty ; Q^\infty)  ) \setminus B_{Rr}\left( \BB e^\infty; Q^\infty   \right) \right) $. By possibly replacing $v$ by an appropriate vertex along a geodesic from $\BB e^\infty$ to $v$, we can assume that $v$ belongs to $B_{2Rr}\left( \BB e^\infty; Q^\infty   \right)$.  Choose vertices $v_0,\dots ,v_k $ as in the definition of $E_r$ for this choice of $v$.  By~\eqref{eqn-path-disjoint}, $\beta^\infty(\lceil T r^2 \rceil)$ is separated from $v$ by the ``annulus" $ B_{2r} (\BB e^\infty ; Q^\infty) \setminus B_{ r}\left( \BB e^\infty; Q^\infty \right)$, so since the spacing between the $v_j$'s is at most $r/50$, one of these vertices has to belong to this annulus, contrary to the definition of $E_r$. 

Thus~\eqref{eqn-filled-ball-contain} holds with probability at least $1- \ep$ for large enough $r$. By possibly increasing $R$ (in a manner depending only on $\ep$), we can arrange that this is the case for every $r$. 
\end{proof}

We next prove an estimate which says that the times~$J_r^\infty$ for the peeling-by-layers process started from the root edge are typically of order~$r^3$. Such an estimate does not follow immediately from~\cite[Theorem~1.12]{gwynne-miller-uihpq} since the time parameterization of the peeling-by-layers process is not encoded in a simple way by the metric measure space structure of $Q^\infty$.  We expect that one has a scaling limit result analogous to~\cite[Theorem~2]{curien-legall-peeling} for the times $J_r^\infty$, but we do not need such a strong result here.  So, for the sake of brevity we will instead prove only the following bound.

\begin{lem} \label{lem-radius-time-upper}
Let $\{J_r^\infty\}_{r\in\BB N_0}$ be the radius-$r$ times for the peeling-by-layers process of the UIHPQ$_{\op{S}}$ with $\BB A = \{\BB e^\infty\}$. For each $\ep \in (0,1)$, there exists $C = C(\ep ) > 1$ such that for each $r\in \BB N$,  
\eqbn
\BB P\left[ C^{-1} r^3 \leq J_r^\infty \leq C r^3 \right] \geq 1-\ep .
\eqen
\end{lem}
\begin{proof}
We will use the local GHPU scaling limit result for the UIHPQ$_{\op{S}}$ to obtain upper and lower bounds for the area of $B_r(\BB e^\infty ;Q^\infty)$, then use Proposition~\ref{prop-length-area-conv}, Lemma~\ref{lem-peel-ball}, and Lemma~\ref{lem-filled-ball-contain} to argue that these lower bounds are violated if $J_r^\infty$ is either too small or too large.

First choose $R = R(\ep) > 0$ such that the conclusion of Lemma~\ref{lem-filled-ball-contain} is satisfied with $R/2$ in place of $R$, so that by Lemma~\ref{lem-peel-ball}, for $r\geq 2$ it holds that 
\eqb \label{eqn-use-ball-contain}
\BB P\left[ \dot Q_{J_r^\infty}^\infty \subset B_{Rr}(\BB e^\infty ; Q^\infty) \right] \geq 1- \ep/4.
\eqe  

Let $(H^\infty , d^\infty , \mu^\infty ,\xi^\infty)$ be an instance of the Brownian half-plane, equipped with its natural metric, area measure, and boundary path, so that $\xi^\infty(0)$ is the marked boundary point. The measure $\mu^\infty$ a.s.\ assigns positive mass to open subsets of $H^\infty$, so for each $\rho > 0$ there exists $C_0 = C_0(\ep , \rho) > 8$ such that 
\eqb \label{eqn-use-bhp}
\BB P\left[   \mu^\infty\left( B_\rho(\xi^\infty(0) ; d^\infty) \right) \geq 4 C_0^{-1}  \: \op{and} \: 
            \mu^\infty\left( B_{R\rho}(\xi^\infty(0) ; d^\infty) \right) \leq \frac14 C_0
  \right] \geq 1-\ep/4  .
\eqe 

By~\cite[Theorem~1.12]{gwynne-miller-uihpq}, the UIHPQ$_{\op{S}}$ equipped with its graph metric rescaled by $(9/8)^{1/4} r^{-1}$, the measure which assigns mass to each vertex equal to $\frac14 r^{-4}$ times its degree, and its boundary path re-parameterized by $t\mapsto \frac{2^{2/3}}{3} r^3 t$ converges in the local GHPU topology to $(H^\infty , d^\infty , \mu^\infty ,\xi^\infty)$. By Lemma~\ref{lem-peel-ball}, \eqref{eqn-use-ball-contain}, and~\eqref{eqn-use-bhp} there exists $r_0 = r_0(\ep) \geq 2$ such that for $r \geq r_0$, 
\eqb \label{eqn-radius-time-area}
\BB P\left[   C_0^{-1} r  <  \# \mcl V\left( \dot Q_{J_r^\infty}^\infty \right)    <   C_0 r \right] \geq 1- \ep/2 .
\eqe  

By Proposition~\ref{prop-length-area-conv}, the number of vertices in $\dot Q_j^\infty$ is typically of order $j^{4/3}$, i.e.\ we can find $C_1 = C_1(\ep) >1$ and $j_*  = j_*(\ep) \in \BB N_0$ such that for $j\geq j_*$,  
\eqb \label{eqn-typical-time-area}
\BB P\left[ C_1^{-1} j^{4/3} <  \#\mcl V\left( \dot Q_j^\infty \right) <  C_1 j^{4/3} \right] \geq 1-\ep/2 .
\eqe 

Set $C  = (C_0C_1)^{ 3/4}  $ and $r_* = r_0 \vee (C j_*)^{1/3}$. By~\eqref{eqn-radius-time-area} and~\eqref{eqn-typical-time-area} (the latter applied with $j = \lfloor C^{-1} r^3 \rfloor $), for $r\geq r_*$, 
\eqbn
\BB P\left[ J_r^\infty  < C^{-1} r^3  \right] 
\leq \BB P\left[ \# \mcl V\left( \dot Q_{J_r^\infty}^\infty \right) \leq C_0^{-1} r^4 \right] + 
     \BB P\left[ \# \mcl V\left( \dot Q_{\lfloor C^{-1} r^3 \rfloor}^\infty \right) \geq C_0^{-1} r^4 \right] 
\leq   \ep . 
\eqen
We similarly find that for large enough $r $, it holds that $\BB P[ J_r^\infty > C r^3] \leq \ep$. 
By possibly increasing $C$ to deal with finitely many small values of $r$, we obtain the statement of the lemma with $2\ep$ in place of $\ep$. Since $\ep \in (0,1)$ is arbitrary, we conclude.
\end{proof}

\subsection{Coupling a free Boltzmann quadrangulation with the UIHPQ$_{\op{S}}$} 
\label{sec-coupling}

In this subsection we will prove a lemma which gives that one can couple a UIHPQ$_{\op{S}}$ and a free Boltzmann quadrangulation with simple boundary in such a way that they agree in a metric neighborhood of the root edge with high probability.  Actually, we will prove a slightly stronger statement with filled metric balls in place of ordinary metric balls.  The result of this subsection is not needed for the proof of Theorem~\ref{thm-fb-ghpu}, but it is of independent interest and is an easy consequence of our other estimates so we include it for the sake of completeness. 

Let $(Q^\infty,\BB e^\infty)$ be a UIHPQ$_{\op{S}}$ and for $\el\in\BB N$, let $(Q^\el,\BB e^\el)$ be a free Boltzmann quadrangulation with simple boundary of perimeter $2\el$.  Let $\BB e_*^\el$ be the edge of $\bdy Q^\el$ directly opposite from the root edge $\BB e^\el$  (i.e., if $\beta^\el$ is the boundary path started from $\BB e^\el$ then $\BB e_*^\el = \beta^\el(\el)$). For $r\geq 0$, let $B_r^\bullet(\BB e^\infty ;Q^\infty)$ (resp.\ $B_r^\bullet (\BB e^\el ; Q^\el)$) be the filled metric ball relative to $\infty$ (resp.\ $\BB e_*^\el$).

\begin{prop} \label{prop-map-coupling}
For each $\ep \in (0,1)$ there exists $\alpha >0$ and $n_*\in\BB N$ such that for $n\geq n_*$, there is a coupling of $(Q^\infty , \BB e^{\infty})$ with $(Q^\el ,\BB e^\el)$ with the following property. With probability at least $1-\ep$, the filled metric balls $B_{\alpha \el^{1/2}}^\bullet(\BB e^\infty ;Q^\infty)$ and $B^\bullet_{\alpha \el^{1/2}}(\BB e^\el ; Q^\el)$ equipped with the graph structures they inherit from $Q^\infty$ and $Q^\el$, respectively, are isomorphic (as graphs) via an isomorphism which takes $\BB e^\infty$ to $\BB e^\el$ and $ \bdy Q^\infty  \cap  B_{\alpha \el^{1/2}}^\bullet(\BB e^\infty ;Q^\infty)$ to $ \bdy Q^\el  \cap B^\bullet_{\alpha \el^{1/2}}(\BB e^\el ; Q^\el)$.  
\end{prop}

Proposition~\ref{prop-map-coupling} is an analog in the setting of free Boltzmann quadrangulations with simple boundary of~\cite[Proposition~4.5]{gwynne-miller-uihpq} (which treats the case of quadrangulations with general boundary) or~\cite[Proposition~9]{curien-legall-plane} (which treats the case of quadrangulations without boundary).

\begin{proof}[Proof of Proposition~\ref{prop-map-coupling}]
For $r\geq 0$, let $\dot Q_{J_r^\infty}^\infty$ (resp.\ $\dot Q_{J_r^\el}^\el$) be the radius-$r$ peeling-by-layers cluster of $Q^\infty$ (resp.\ $Q^\el$) started from $\BB e^\infty$ (resp.\ $\BB e^\el$) and targeted at $\infty$ (resp.\ $\BB e_*^\el$).  By Lemma~\ref{lem-peel-ball}, it suffices to prove the statement of the lemma with $\dot Q_{J_{\alpha\el^{1/2}}^\infty}^\infty$ and $\dot Q_{J_{\alpha\el^{1/2}}^\el}^\el$ in place of  $B_{\alpha \el^{1/2}}^\bullet(\BB e^\infty ;Q^\infty)$ and $B_{\alpha \el^{1/2}}^\bullet(\BB e^\el ; Q^\el)$. 

By Lemmas~\ref{lem-uihpq-ball-length} and~\ref{lem-radius-time-upper}, there exists $C_0 = C_0(\ep) > 0$ such that for $r \in \BB N$, 
\eqbn
\BB P\left[ Y_{J_r^\infty}^\infty \leq C_0 r^2 \: \op{and} \: J_r^\infty \leq C_0 r^3 \right] \geq 1-\ep/2 .
\eqen
By Proposition~\ref{prop-length-area-conv}, the supremum of the net boundary length process $W^\infty$ up to time $C_0 r^3$ is typically of order $r^2$, so there exists $C = C(\ep) \geq C_0$ such that for $r \in \BB N$, it holds with probability at least $1-\ep$ that $Y_{J_r^\infty}^\infty \leq C  r^2 $ and $ W_{J_r^\infty}^\infty \leq C  r^2$.  If this is the case, then the exposed boundary length satisfies $X_{J_r^\infty}^\infty \leq 2C r^2$. 

Choose $\alpha < (2C)^{-1/2} \ep^{1/2}$. For $\el\in\BB N$, applying the above estimate with $r =\lfloor \alpha \el^{1/2}\rfloor$ shows that it holds with probability at least $1-\ep$ that $Y_{J_{\alpha \el^{1/2}}^\infty}^\infty \vee X_{J_{\alpha \el^{1/2}}^\infty}^\infty \leq \ep \el$. By Lemma~\ref{lem-bubble-cond} (applied with the root edge of $Q^\infty$ translated $\el$ units to the left) we infer that on the event that this is the case, the Radon-Nikodym derivative of the law of the triple $\left(  \dot Q_{J_r^\el}^\el , \bdy Q^\el \cap  \dot Q_{J_r^\el}^\el , \BB e^\el \right)$ with respect to the law of $\left(  \dot Q_{J_r^\infty}^\infty , \bdy Q^\infty \cap  \dot Q_{J_r^\infty}^\infty , \BB e^\infty \right)$ is of order $(1+ o_\el(1)) (1 + 2\ep)^{-5/2}$.  Since $\ep \in (0,1)$ is arbitrary, the statement of the proposition follows. 
\end{proof}

\section{Proofs of Theorems~\ref{thm-fb-ghpu} and~\ref{thm-saw-conv}}
\label{sec-conclusion}

\subsection{Proof of Theorem~\ref{thm-fb-ghpu}}
\label{sec-fb-ghpu-proof}

In this subsection we assume we are in the setting of Theorem~\ref{thm-fb-ghpu} so that for $\el\in\BB N$, $(Q^\el , \BB e^\el)$ is a free Boltzmann quadrangulation with simple boundary of perimeter $2\el$ and $\frk Q^\el = (Q^\el , d^\el , \mu^\el , \xi^\el)$ is the corresponding rescaled curve-decorated metric measure space. We will deduce Theorem~\ref{thm-fb-ghpu} from Proposition~\ref{prop-core-ghpu-fb} in the following manner.  For $\delta>0$, we let $L_\delta^\el$ be a random variable whose law is as in Proposition~\ref{prop-core-ghpu-fb} with $\lfloor (1+\delta) l \rfloor$ in place of $\el$, independent from $(Q^\el ,\BB e^\el)$.  We grow the peeling-by-layers process of $Q^\el$ started from $\BB e^\el$ up to the first time $T_\delta^\el$ that the boundary length of the unexplored quadrangulation is exactly $2L_\delta^\el$. 

On the event $\{T_\delta^\el < \infty\}$, this unexplored quadrangulation $\ol Q_{T_\delta^\el}^\el$ has the law of free Boltzmann quadrangulation with simple boundary of perimeter $2L_\delta^\el$, so by Proposition~\ref{prop-core-ghpu-fb} $\ol Q_{T_\delta^\el}^\el$ (equipped with its rescaled graph metric, area measure, and boundary path) converges in the scaling limit in the GHPU topology to a free Boltzmann Brownian disk of perimeter $1+\delta$.  We will show in Lemma~\ref{lem-ghpu-dist-on-event} that $\ol Q_{T_\delta^\el}^\el$ is a good approximation for $Q^\el$ in the GHPU topology provided a certain regularity event occurs; and in Lemma~\ref{lem-fb-time} that this regularity event occurs with high probability.  We will then deduce Theorem~\ref{thm-fb-ghpu} by combining these statements. 

We now proceed with the details.  For $\el\in\BB N$, we consider the peeling-by-layers process of $(Q^\el,\BB e^\el)$ with initial edge set $\BB A  =\{\BB e^\el\}$ targeted at the edge $\BB e_*^\el$ directly opposite from $\BB e^\el$ in $\bdy Q^\el$.  We define the clusters $\{\dot Q_j^\el\}_{j\in\BB N_0}$, the unexplored quadrangulations $\{\ol Q_j^\el\}_{j\in\BB N_0}$, the peeled edges $\{\dot e_j^\el\}_{j\in\BB N_0}$, the radius-$r$ times $\{J_r^\el\}_{r\in\BB N_0}$, and the $\sigma$-algebras $\{\mcl F_j^\el\}_{j\in\BB N_0}$ for this process as in Section~\ref{sec-pbl-def} and we define the boundary length processes $X^\el, Y^\el$, and $W^\el$ as in Definition~\ref{def-bdy-process}. 

For $\delta > 0$ and $\el\in\BB N$, let $L_\delta^\el$ be a random variable whose law is that of $\frac12 \#\mcl E(\op{Core}(\wh Q_\delta^\el))$, where $\wh Q_\delta^\el$ is a free Boltzmann quadrangulation with general boundary of perimeter $6\lfloor (1+\delta) \el \rfloor$, independent from $(Q^\el  , \BB e^\el)$. 
Define the time 
\eqb \label{eqn-fb-time}
T_\delta^{ \el} := \min\left\{ j \in \BB N_0 : W_j^\el    = 2L_\delta^\el  - 2\el \right\} = \min\left\{j\in \BB N_0 : \# \mcl E(\bdy \ol Q_j^\el )  = 2L_\delta^\el \right\} . 
\eqe 
  
By Lemma~\ref{lem-peel-law}, on the event $\{T_\delta^\el  < \infty\}$, the conditional law of the unexplored quadrangulation~$\ol Q_{T_\delta^\el}^\el$ given~$L_\delta^\el$ and the peeling $\sigma$-algebra $\mcl F_{T_\delta^\el}$ is that of a free Boltzmann quadrangulation with simple boundary of perimeter $2L_\delta^\el$.  Let $\ol d_\delta^\el$ be the internal graph metric on $\ol Q_{T_\delta^\el}^\el$ rescaled by $(2/3)^{-1/2} \el^{-1/2}$, let~$\ol\mu_\delta^\el := \mu^\el|_{\ol Q_{T_\delta^\el}^\el}$, let~$\ol \beta_\delta^\el$ be the boundary path of $\ol Q_{T_\delta^\el}^\el$ started from the rightmost edge of $\bdy \dot Q_{T_\delta^\el}^\el \cap \bdy \ol Q_{T_\delta^\el}^\el$ and extended by linear interpolation, and let $\ol\xi_\delta^\el(t) := \ol\beta^\el(2\el t)$ for $t\in [0,\el^{-1} L_\delta^\el]$.  Also define the curve-decorated metric measure space 
\eqb \label{eqn-unexplored-space}
\ol{\frk Q}_\delta^\el := \left( \ol Q_{T_\delta^\el}^\el , \ol d_\delta^\el , \ol\mu_\delta^\el , \ol\xi_\delta^\el \right)  
\eqe 
so that by Proposition~\ref{prop-core-ghpu-fb}, the conditional law of $\ol{\frk Q}_\delta^\el$ converges weakly to the law of a free Boltzmann Brownian disk of perimeter $1+\delta$ with respect to the GHPU topology . 

The first main input in the proof of Theorem~\ref{thm-fb-ghpu} tells us that $\ol{\frk Q}_\delta^\el$ is a good approximation to $\frk Q_\delta^\el$ in the GHPU sense on a regularity event which (as we will see below) occurs with high probability. We split this event into two parts, one which is $\mcl F_{T_\delta^\el}^\el$-measurable and one which is $\sigma(\ol Q_{T_\delta^\el}^\el)$-measurable. 

For $\ep \in (0,1)$, let $E_\delta^{\el}(\ep)$ be the event that the following hold:
\allb \label{eqn-fb-ball-event} 
& T_\delta^{ \el} < \infty  ,\quad 
Y_{T_\delta^{  \el}}^\el \leq \ep \el  ,\quad 
\#\mcl V\left( \dot Q_{T_\delta^{ \el}}^\el \right) \leq \ep \el^2  , \notag \\
& \max_{v \in \mcl V(\dot Q_{T_\delta^{ \el}}^\el)} \op{dist}\left( v , \BB e^\el ; \dot Q_{T_\delta^{ \el}}^\el \right) \leq \ep \el^{1/2}   ,
\quad \op{and} \quad L_\delta^\el \leq \left( 1 +\ep \right) \el .
\alle 
The event $E_\delta^\el(\ep)$ is in the peeling $\sigma$-algebra $\mcl F_{T_\delta^\el}^\el$; this is why we consider internal graph distances in $\dot Q_{T_\delta^{ \el}}^\el$ instead of graph distances in $Q^\el$ itself. 

For $\alpha  \in (0,1) $, also let 
\eqb \label{eqn-bdy-equicont-end}
F_\delta^\el(\ep,\alpha) := \left\{ \ol d_\delta^\el\left( \ol\xi_\delta^\el(s) , \ol\xi_\delta^\el(t) \right) \leq \alpha ,\: \forall s,t\in [0, \el^{-1} L_\delta^\el] \:\op{with}\:  |s-t|\leq 4 \ep  \right\}   .
\eqe

\begin{lem} \label{lem-ghpu-dist-on-event} 
For each $\ep \in (0,1)$ and each $\alpha \in (0,1)$, there exists $\el_* = \el_*(\alpha,\ep) \in \BB N$ such that for $\el\geq \el_*$ and each $\delta>0$, it holds on $E_\delta^\el(\ep) \cap F_\delta^\el(\ep,\alpha)$ that $\BB d^{\op{GHPU}}(\frk Q^\el , \ol{\frk Q}_\delta^\el) \leq 6\alpha + 18\ep$. 
\end{lem}

For the proof of Lemma~\ref{lem-ghpu-dist-on-event}, we will use the following general lemma about the GHPU metric.

\begin{lem} \label{lem-ghpu-map}
Let $\frk X_1  =( X_1 , d_1 , \mu_1 , \eta_1)$ and $\frk X_2 = (X_2,d_2,\mu_2,\eta_2)$ be elements of $\BB M^{\op{GHPU}}$ (recall Section~\ref{sec-ghpu}).  Let $\ep > 0$ and suppose there is an injective map $f : X_1 \rta X_2$ (not necessarily continuous) such that the following is true. 
\begin{enumerate}
\item For each $x,y \in X_1$, one has $|d_1(x,y) - d_2(f(x) , f(y))| \leq \ep$.  \label{item-ghpu-map-dist}
\item The $d_2$-Hausdorff distance between $f(X_1)$ and $X_2$ is at most $\ep$. \label{item-ghpu-map-haus}
\item The $d_2$-Prokhorov distance between $f_* \mu_1$ and $\mu_2$ is at most $\ep$. \label{item-ghpu-map-measure}
\item The $d_2$-uniform distance between $f\circ \eta_1$ and $\eta_2$ is at most $\ep$.\label{item-ghpu-map-curve}
\end{enumerate}
Then $\BB d^{\op{GHPU}}(\frk X_1,\frk X_2) \leq  6\ep$. 
\end{lem} 
\begin{proof}
Define a metric $d_\sqcup$ on the disjoint union $X_1 \sqcup X_2$ by
\eqbn
d_\sqcup(x,y) = 
\begin{cases}
d_1(x,y) ,\quad &x,y\in X_1 \\
d_2(x,y) ,\quad &x,y \in X_2 \\
\inf_{u \in X_1} \left( d_1 (x ,u) + d_2(y , f(u)) + \ep \right) ,\quad &x \in X_1 ,\: y\in X_2 
\end{cases}
\eqen
and define $d_\sqcup$ in a symmetric manner if $x\in X_2$ and $y\in X_1$. It is easily verified using condition~\ref{item-ghpu-map-dist} that $d_\sqcup$ satisfies the triangle inequality, so is a metric on $X_1 \sqcup X_2$. By condition~\ref{item-ghpu-map-haus} and since $d_\sqcup(x,f(x)) = \ep$ for each $x \in X_1$, we infer that the $d_\sqcup$-Hausdorff distance between $X_1$ and $X_2$ is at most $2\ep$. Condition~\ref{item-ghpu-map-measure} implies that the $d_\sqcup$-Prokhorov distance between~$\mu_1$ and~$\mu_2$ is at most $2\ep$ (here is the only place where we use injectivity of~$f$) and condition~\ref{item-ghpu-map-curve} implies that the $d_\sqcup$-uniform distance between~$\eta_1$ and~$\eta_2$ is at most~$\ep$.
\end{proof}

\begin{proof}[Proof of Lemma~\ref{lem-ghpu-dist-on-event}]
Suppose $E_\delta^\el(\ep) \cap F_\delta^\el(\el ,\alpha) $ occurs. 
We will check the hypotheses of Lemma~\ref{lem-ghpu-map} with $\frk X_1 = \ol{\frk Q}_\delta^\el$, $\frk X_2 =\frk Q^\el$, and $f$ the inclusion map $\ol Q_{T_\delta^\el}^\el \rta Q^\el$. 

Since $Y_{T_\delta^{  \el}}^\el \leq \ep \el$ and $W_{T_\delta^\el}^\el = L_\delta^\el - 2\el \leq 2\ep \el$, the exposed boundary length process satisfies $X_{T_\delta^\el}^\el \leq 3\ep \el$ and hence $\#\mcl E(\bdy \dot Q_\el^\el) \leq 4\ep \el$. 

We first check that $d^\el$ and $\ol d_\delta^\el$ distances are comparable. Suppose $x,y\in \ol Q_{T_\delta^\el}^\el$. It is clear that $d^\el(x,y) \leq \ol d_\delta^\el(x,y)$. To obtain an inequality in the reverse direction, let $\gamma : [0, d^\el(x,y)] \rta Q^\el$ be a $d^\el$-geodesic from $x$ to $y$ (extended by linear interpolation). If $\gamma$ does not enter $\dot Q_{T_\delta^\el}$, then clearly $d^\el(x,y)= \ol d_\delta^\el(x,y)$. Otherwise, let $t_1$ (resp.\ $t_2$) be the first (resp.\ last) time that $\gamma$ enters (resp.\ exits) $\dot Q_{T_\delta^\el}^\el$. Then $\gamma(t_1) , \gamma(t_2) \in \bdy \dot Q_{T_\delta^\el}^\el$ so since $X_{T_\delta}^\el \leq 3\ep \el$ the rescaled $\ol Q_{T_\delta^\el}^\el$-boundary length distance from $\gamma(t_1)$ to $\gamma(t_2)$ is at most $3\ep$. By definition of $ F_\delta^\el(\el ,\alpha)$, there exists a path $\gamma'$ in $\ol Q_{T_\delta^\el}^\el$ from $\gamma(t_1)$ to $\gamma(t_2)$ with $\ol d_\delta^\el$-length at most $\alpha$. Concatenating $\gamma|_{[0,t_1]}$, $\gamma'$, and $\gamma|_{[t_2, d^\el(x,y)]}$ shows that 
\eqbn
 d^\el(x,y) \leq \ol d_\delta^\el(x,y) \leq d^\el(x,y)  + \alpha .
\eqen

Since $\op{diam}(\dot Q_{T_\delta^\el}^\el ; d^\el ) \leq 2\ep$, the $d^\el$-Hausdorff distance from $Q^\el$ to $\ol Q_\delta^\el$ is at most $2\ep$.  Since $\#\mcl V\left( \dot Q_{T_\delta^{ \el}}^\el \right) \leq \ep \el^2$ and $\dot Q_{T_\delta^\el}^\el$ is a quadrangulation with simple boundary, an Euler's formula argument shows that $\mu^\el\left( \dot Q_{T_\delta^{ \el}}^\el  \right) \leq 2\ep +o_\el(1) $ with the rate of the $o_\el(1)$ deterministic and universal (it comes from the fact that $\#\mcl E(\bdy \dot Q_\el^\el) \leq 4\ep\el$). Hence the $d^\el$-Prokhorov distance between $\mu^\el$ and $\ol\mu_\delta^\el$ is at most $2\ep$.  By definition of $ F_\delta^\el(\el ,\alpha)$, since $\#\mcl E(\bdy \dot Q_\el^\el) \leq 4\ep\el$, and since $\op{diam}(\dot Q_{T_\delta^\el}^\el ; d^\el ) \leq 2\ep$, the $d^\el$-uniform distance between $\xi^\el$ and $\ol\xi_\delta^\el$ is at most $\alpha+\ep$.

Thus, for large enough $\el \in \BB N$, depending only on $\ep$ and $\alpha$, the conditions of Lemma~\ref{lem-ghpu-map} are satisfied on $E_\delta^\el(\ep ) \cap  F_\delta^\el(\el ,\alpha)$ with $\frk X_1 = \ol{\frk Q}_\delta^\el$, $\frk X_2 =\frk Q^\el$, $f$ the inclusion map, and $\alpha + 3\ep$ in place of $\ep$. So, the statement of the lemma follows from Lemma~\ref{lem-ghpu-map}
\end{proof}

Before we can deduce Theorem~\ref{thm-fb-ghpu} from Proposition~\ref{prop-core-ghpu-fb} and Lemma~\ref{lem-ghpu-dist-on-event}, we need to argue that the regularity event in Lemma~\ref{lem-ghpu-dist-on-event} occurs with high probability. Actually, we will only explicitly write down an estimate for the probability of the event $E_\delta^\el(\ep)$; the estimate for $\BB P[ F_\delta^\el(\ep,\alpha) ]$ is an easy consequence of Proposition~\ref{prop-core-ghpu-fb} and is explained in the proof of Theorem~\ref{thm-fb-ghpu}. 

\begin{lem} \label{lem-fb-time}
Define the events $E_\delta^\el(\ep)$ as in~\eqref{eqn-fb-ball-event}.
For each $\ep , \alpha \in(0,1/4)$, there exists $\delta_* = \delta_*(\ep,\alpha) > 0$ such that for each $\delta \in (0,\delta_*]$ there exists $\el_* =\el_*( \delta , \ep,\alpha) \in \BB N$ such that for $\el \geq \el_*$,  
\eqbn
\BB P\left[ E_\delta^{ \el}(\ep ) \right] \geq 1-\alpha .
\eqen
\end{lem}

We will deduce Lemma~\ref{lem-fb-time} from an analogous estimate for the peeling-by-layers process on the UIHPQ$_{\op{S}}$ and local absolute continuity, in the form of Lemma~\ref{lem-bubble-cond}. 
Let $(Q^\infty,\BB e^\infty)$ be a UIHPQ$_{\op{S}}$ independent from $L_\delta^\el$ and consider the peeling-by-layers process of $Q^\infty$ started from $\BB e^\infty$ and targeted at $\infty$. We define the objects associated with this process as in Section~\ref{sec-pbl-def} and as per usual we denote these objects by a superscript $\infty$. 

In analogy with~\eqref{eqn-fb-time}, define
\eqbn
T_\delta^{\infty,\el} := \min\left\{ j \in \BB N_0 : W_j^{\infty }    = 2L_\delta^\el  - 2\el \right\} 
\eqen
where here $W^\infty$ is the net boundary length process for our peeling-by-layers process. 
For $\ep \in (0,1)$, let $E_\delta^{\infty,\el}(\ep)$ be the event that 
\allb \label{eqn-uihpq-ball-event} 
&T_\delta^{\infty,\el} < \infty  ,\quad 
  Y_{T_\delta^{\infty,\el}}^\infty \leq \ep \el  ,\quad 
  \#\mcl V\left( \dot Q_{T_\delta^{\infty,\el}}^\infty \right) \leq \ep \el^2  ,\notag \\
&  \max_{v \in \mcl V(\dot Q_{T_\delta^{\infty,\el}}^\infty)} \op{dist}\left( v , \BB e^\infty ; \dot Q_{T_\delta^{\infty,\el}}^\infty \right) \leq \ep \el^{1/2}  ,
  \quad \op{and} \quad L_\delta^\el \leq \left( 1 +\ep  \right) \el .
\alle

\begin{lem} \label{lem-uihpq-time}
For each $\ep , \alpha \in(0,1/4)$, there exists $\delta_* = \delta_*(\ep,\alpha) > 0$ such that for each $\delta \in (0,\delta_*]$ there exists $\el_* =\el_*( \delta , \ep,\alpha) \in \BB N$ such that for $\el \geq \el_*$,  
\eqbn
\BB P\left[ E_\delta^{\infty,\el}(\ep ) \right] \geq 1-\alpha .
\eqen
\end{lem} 
\begin{proof}
Fix $\ep , \alpha \in (0,1/4)$ to be chosen later.  
By Lemma~\ref{lem-uihpq-ball-length}  , there is a $\delta_0 = \delta_0(\ep , \alpha) \in (0,\ep/2]$ such that for $\el \in \BB N$,
\eqb \label{eqn-uihpq-time-length}
\BB P\left[ Y_{\lfloor \delta_0 \el^{3/2} \rfloor}^\infty \leq \ep \el \right] \geq 1-\alpha .
\eqe 
Note here that $j\mapsto Y_j^\infty$ is non-decreasing. 
By Proposition~\ref{prop-length-area-conv}, by possibly shrinking $\delta_0$ we can arrange that also
\eqb \label{eqn-uihpq-time-area}
\BB P\left[ \#\mcl V\left(  \dot Q_{\lfloor \delta_0 \el^{3/2} \rfloor}^\infty \right)   \leq \ep \el^2 \right] \geq 1-\alpha .
\eqe 

By Lemma~\ref{lem-filled-ball-contain}, we can find $\delta_1 = \delta_1(\ep,\alpha) \in (0,\delta_0]$ such that for $\el\in\BB N$, 
\eqb \label{eqn-uihpq-time-diam}
 \BB P\left[  \dot Q_{J_{\lfloor \delta_1\el^{1/2} \rfloor}^\infty}^\infty \subset B_{(\ep/2) \el^{1/2} - 4}\left( \BB e^\infty ;Q^\infty \right)  \right] \geq 1-\alpha . 
\eqe   
By Lemma~\ref{lem-radius-time-upper}, by possibly shrinking $\delta_1$ we can arrange that also 
\eqb \label{eqn-uihpq-time-J}
\BB P\left[ J_{\lfloor \delta_1 \el^{1/2} \rfloor}^\infty \leq \delta_0 \el^{3/2} \right] \geq 1-\alpha .
\eqe 

By Proposition~\ref{prop-length-area-conv}, the process $t \mapsto \el^{-1} W_{\lfloor t \el^{3/2} \rfloor}^\infty$ converges in law in the local Skorokhod topology to a totally asymmetric $3/2$-stable process with no upward jumps. 
In particular, there is a $\delta_2 = \delta_2(\ep,\alpha) > 0$ and an $\el_0 = \el_0(\ep,\alpha) \in \BB N$ such that for $\el\geq \el_0$, 
\eqb \label{eqn-uihpq-time-hit}
\BB P\left[ \max_{j \in [0, \delta_0 \el^{3/2}]_{\BB Z}} W_j^\infty \geq \delta_2 \el \right] \geq 1-\alpha .
\eqe  
By Proposition~\ref{prop-core-ghpu-fb}, the random variable $\el^{-1} L_\delta^\el$ converges in law to the constant $1+\delta$. 
Hence for $\delta \in (0,1)$, there exists $\el_1 = \el_1 ( \delta , \ep , \alpha) \geq \el_0$ such that for $\el \geq \el_1$,  
 \eqb \label{eqn-uihpq-time-approx}
 \BB P\left[   L_\delta^\el \in [ ( 1 +\delta/2 ) \el , ( 1+(3/2) \delta) \el ]  \right] \geq 1-\alpha .
\eqe 

Now set $\delta = \delta_*/3$ and suppose that the events in~\eqref{eqn-uihpq-time-length}--\eqref{eqn-uihpq-time-approx} occur, which happens with probability at least $1-6\alpha$. We claim that $E_\delta^{\infty,\el}(\ep)$ occurs. Indeed, since $W_j^\infty - W_{j-1}^\infty \leq 2$ for each $j\in\BB N$, the event in~\eqref{eqn-uihpq-time-hit} implies that $W^\infty$ hits every even integer in $[0,\delta_2 \el]_{\BB Z}$ before time $\delta_0$. The event in~\eqref{eqn-uihpq-time-approx} implies that $2L_\delta^\el  - 2\el \in [0,\delta_2 \el]_{\BB Z}$, so $T_\delta^{\infty,\el} \leq \delta_0 \el^{3/2}$. The events in~\eqref{eqn-uihpq-time-length},~\eqref{eqn-uihpq-time-area}, and~\eqref{eqn-uihpq-time-J} immediately imply that
\eqbn
Y_{T_\delta^{\infty,\el}}^\infty \leq \ep \el ,\quad 
\#\mcl V\left( \dot Q_{T_\delta^{\infty,\el}}^\infty \right) \leq \ep \el^2, \quad \op{and} \quad
T_\delta^{\infty,\el} \leq J_{\lfloor \delta_1 \el \rfloor}^\infty .
\eqen

It remains only to check the distance condition in the definition of $E_\delta^{\infty,\el}(\ep)$.  Since $T_\delta^{\infty,\el} \leq J_{\lfloor \delta_1 \el \rfloor}^\infty$, the event in~\eqref{eqn-uihpq-time-J} tells us that each vertex of $\dot Q_{T_\delta^{\infty,\el}}^\infty$ lies at $Q^\infty$-graph distance at most $(\ep/2) \el^{1/2} - 4$ from $\BB e^\infty$.  We need to convert this to a bound for $\dot Q_{T_\delta^{\infty,\el}}^\infty$-graph distances.  Let $v \in \mcl V(\dot Q_{T_\delta^{\infty,\el}}^\infty)$ and let $\gamma$ be a $Q^\infty$-geodesic from~$\BB e^\infty$ to~$v$.  If~$\gamma$ stays in $\dot Q_{T_\delta^{\infty,\el}}^\infty$, then $\op{dist}\left( \BB e^\infty ,v  ; Q^\infty \right) = \op{dist}\left( \BB e^\infty ,v  ; \dot Q_{T_\delta^{\infty,\el}}^\infty \right)$ and we are done. Otherwise, let~$r_*$ be the largest time in $[1,|\gamma|]_{\BB Z}$ for which~$\gamma(r)$ does not belong to~$\dot Q_{T_\delta^{\infty,\el}}^\infty$ and let~$v_*$ be the terminal endpoint of~$\gamma(r_*)$. Then $v_* \in \bdy \dot Q_{T_\delta^{\infty,\el}}^\infty$ so by Lemma~\ref{lem-peel-ball}, 
\eqbn
\op{dist}\left( \BB e^\infty ,v_*  ; \dot Q_{T_\delta^{\infty,\el}}^\infty \right) \leq \delta_0 \el^{1/2} + 4 \leq (\ep/2) \el^{1/2} + 4 .
\eqen
 By concatenating a $\dot Q_{T_\delta^{\infty,\el}}^\infty$-geodesic from $\BB e^\infty$ to $v_*$ with the path $\gamma|_{[r_*+1,|\gamma|]_{\BB Z}}$ from $v_*$ to $v$ in $\dot Q_{T_\delta^{\infty,\el}}^\infty$, we obtain the desired distance bound. 
\end{proof}

\begin{proof}[Proof of Lemma~\ref{lem-fb-time}]
For $\el \in \BB N$, let $S^{\el,\infty}(\ep) := \min\left\{ j\in\BB N_0 : Y_j^\infty \geq \ep \el \right\}$.  Since $j\mapsto Y_j^\infty$ is non-decreasing, the event $E_\delta^\infty(\ep)$ is $\mcl F_{S^\el(\ep)}^\infty$-measurable. Since $W_{S^{\el,\infty}(\ep)}^\infty \geq -\el/4$, the statement of the lemma now follows from Lemma~\ref{lem-bubble-cond} (applied at time $S^{\el,\infty}(\ep)$) and Lemma~\ref{lem-uihpq-time}.
\end{proof}

\begin{proof}[Proof of Theorem~\ref{thm-fb-ghpu}] 
Fix $\alpha \in (0,1/100)$ and for $\el\in\BB N$ and $\delta>0$, let $\ol{\frk Q}_\delta^\el$ be the curve-decorated metric measure space as in~\eqref{eqn-unexplored-space}. For $\frk l > 0$, let $\frk H_{\frk l} = (H_{\frk l},  d_{\frk l} , \mu_{\frk l} , \xi_{\frk l})$ denote a free Boltzmann Brownian disk with perimeter $\frk l$. Then with $\frk H$ as in the statement of the lemma, $\frk H_{\frk l} \eqD (H , \frk l^{1/2} d , \frk l^2 \mu , \xi(\frk l \cdot) )$. From this we infer that there exists $\delta_0 = \delta_0(\alpha) \in (0,1)$ such that for $\delta \in (0,\delta_0]$, the Prokhorov distance between the laws of $\frk H$ and $\frk H_{\frk l}$ with respect to the GHPU metric is at most $\alpha$. 
 
Since the conditional law of $\ol Q_{T_\delta}^\el$ given $L_\delta^\el$ and $\mcl F_{T_\delta^\el}$ on the event $\{T_\delta^\el  < \infty\}$ is that of a free Boltzmann quadrangulation with simple boundary of perimeter $2 L_\delta^\el$,  
Proposition~\ref{prop-core-ghpu-fb} implies that for each $\delta \in (0,\delta_0]$ there exists $\el_0 = \el_0(\delta , \alpha) \in \BB N$ such that for $\el \geq \el_0$, the Prokhorov distance between the conditional law of $\ol{\frk Q}_\delta^\el$ given $\{T_\delta^\el < \infty\}$ and the law of $\frk H_{1+\delta}$ with respect to the GHPU metric is at most $\alpha$. 
Hence the Prokhorov distance between the law of $\frk H$ and the conditional law of $\ol{\frk Q}_\delta^\el$ given $\{T_\delta^\el < \infty\}$ with respect to the GHPU metric is at most $2\alpha$. 

By Proposition~\ref{prop-core-ghpu-fb} and the above scaling argument, there exists $\ep_0 = \ep_0(\alpha) \in (0,1/4)$ such that for each $\delta \in (0,\delta_0]$ and $\el \in \BB N$, on the event $\{T_\delta^\el < \infty\}$ it holds with conditional probability at least $1-\alpha$ given $L_\delta^\el$ and $\mcl F_{T_\delta^\el}^\el$ that the event $F_\delta^\el(\ep_0,\alpha)$ of~\eqref{eqn-bdy-equicont-end} occurs.

Set $\ep = \ep_0 \wedge \alpha$. 
By Lemma~\ref{lem-fb-time}, there exists $\delta  \in (0,\delta_0]$ and $\el_* = \el_*(\delta , \ep,\alpha)$ such that for $\el\geq \el_*$, the event $E_\delta^\el(\ep)$ occurs with probability at least $1-\alpha$. 

By combining the preceding two paragraphs with Lemma~\ref{lem-ghpu-dist-on-event} we find that with probability at least $1-2\alpha$, we have $T_\delta^\el < \infty$ and $\BB d^{\op{GHPU}}(\frk Q^\el , \ol{\frk Q}_\delta^\el) \leq 24\alpha$.
Hence the Prokhorov distance between the law of $\frk Q^\el$ and the conditional law of $\ol{\frk Q}_\delta^\el$ given $\{T_\delta^\el < \infty\}$ is at most $100\alpha$. Combining this with the conclusion of the second paragraph and using that $\alpha$ can be made arbitrarily small concludes the proof.
\end{proof}

\subsection{Proof of Theorem~\ref{thm-saw-conv}}
\label{sec-saw-conv-proof}

Throughout this subsection we assume we are in the setting of Theorem~\ref{thm-saw-conv}, so that for $\el\in\BB N$, $(Q_\pm^\el , \BB e_\pm^\el)$ are free Boltzmann quadrangulations with simple boundary of perimeter~$2\el$ identified along their boundary paths~$\beta_\pm^\el$ to obtain the curve-decorated graph $(Q_\glu^\el , \beta_\glu^\el)$. 

Now that Theorem~\ref{thm-fb-ghpu} has been established, the key difficulty in the proof of Theorem~\ref{thm-saw-conv} is showing that paths between two given points of~$Q_\glu^\el$ which cross the gluing interface~$\beta_\glu^\el$ more than a constant order number of times are not substantially shorter than paths which cross only a constant order number of times (recall from Section~\ref{sec-metric-prelim} the definition of the quotient metric); this is analogous to the key difficulty in the proofs of~\cite{gwynne-miller-saw}.  In the present setting, this difficulty will be resolved using the results of~\cite{gwynne-miller-saw} and a local absolute continuity argument.  We now state the key lemma needed for the proof.  

For $\el\in\BB N$, $k_0,k_1 \in [0,2\el]_{\BB Z}$, and $N\in\BB N$, let $F^\el_{N,\zeta}(k_0,k_1)$ be the event that there exists a path~$\wt\gamma$ in~$Q_\glu^\el$ from~$\beta_\glu^\el(k_0) $ to~$\beta_\glu^\el(k_1)$ which crosses~$\beta_\glu^\el$ at most~$N$ times and has length at most
\[ \op{dist}(\beta_\glu^\el(k_0) , \beta_\glu^\el(k_1) ; Q_\glu^\el) + \zeta \el^{1/2}.\] 

\begin{lem} \label{lem-geo-approx}
For each $ \ep \in (0,1)$, there exists $\alpha = \alpha(  \ep) > 0$ and an event $E^\el = E^\el(\ep,\alpha)$ such that for each $\zeta \in (0,\alpha]$, there exists $\el_* = \el_*( \ep ,\zeta) \in\BB N$ and $N = N( \ep , \zeta) \in\BB N$ such that for each $\el\geq \el_*$ and each $k_0,k_1\in [0,2\el]_{\BB Z}$, 
\eqbn
\BB P[E^\el] \geq 1-\ep \quad \op{and} \quad 
\BB P\left[E^\el \cap \left\{  \op{dist}(\beta_\glu^\el(k_0) , \beta_\glu^\el(k_1) ; Q_\glu^\el) \leq \alpha \el^{1/2} \right\} \cap F^\el_{N,\zeta}(k_0,k_1)^c \right] \leq \zeta .
\eqen 
\end{lem}

Lemma~\ref{lem-geo-approx} will be deduced from the infinite-volume scaling limit results of~\cite{gwynne-miller-saw} and a local absolute continuity lemma which follows from Lemma~\ref{lem-bubble-cond}. Let us first record some consequences of~\cite[Theorem~1.2]{gwynne-miller-saw}, which is the infinite-volume analog of Theorem~\ref{thm-saw-conv}. 

Let $(Q_\pm^\infty , \BB e_\pm^\infty)$ be a pair of independent UIHPQ$_{\op{S}}$'s, let $\beta_\pm^\infty$ be their respective boundary paths started from the root edge, and let $Q_\glu^\infty$ be the map obtained by identifying $\beta_-^\infty(j)$ and $\beta_+^\infty(j)$ for each $j\in\BB Z$. Also let $\beta_\glu^\infty : \BB Z \rta \mcl E(Q_\glu^\infty)$ be the path corresponding to $\beta_-^\infty$ and $\beta_+^\infty$. 

By~\cite[Theorem~1.2]{gwynne-miller-saw}, the graph $Q_\glu^\infty$, equipped with its rescaled graph metric, its rescaled natural area measure, and a re-scaling of $\beta_\glu^\infty$ converges in law in the local GHPU topology to a curve-decorated metric measure space $\frk H_\glu^\infty = (H_\glu^\infty,d_\glu^\infty,\mu_\glu^\infty ,  \xi_\glu^\infty)$ consisting of a pair of independent Brownian half-planes with their (full) boundary paths identified.  By definition of the quotient metric (Section~\ref{sec-metric-prelim}), it follows that graph distances in $Q_\pm^\infty$ can be approximated by the lengths of paths which cross the gluing interface $\beta_\glu^\infty$ only a constant order number of times, in the following sense.

\begin{lem} \label{lem-infty-geo-approx}
For each $C>0$ and each $ \zeta \in (0,1)$, there exists $\el_* = \el_*(C, \zeta ) \in\BB N$ and $N = N(C, \zeta) \in\BB N$ such that for $\el\geq \el_*$ and $k_0,k_1\in [-C \el , C \el]_{\BB Z}$, it holds with probability at least $1-\zeta$ that the following is true. There exists a path $\wt\gamma$ in $Q_\glu^\infty$ from $\beta_\glu^\infty(k_0) $ to $\beta_\glu^\infty(k_1)$ which crosses $\beta_\glu^\infty$ at most $N$ times and has length 
\eqbn
|\wt\gamma| \leq \op{dist}(\beta_\glu^\infty(k_0) , \beta_\glu^\infty(k_1) ; Q_\glu^\infty) + \zeta \el^{1/2} .
\eqen
\end{lem}

By~\cite[Corollary~1.5]{gwynne-miller-saw}, $\frk H_\glu^\infty$ has the same law as a certain $\sqrt{8/3}$-LQG surface (namely a weight-$4$ quantum cone) decorated by an independent two-sided SLE$_{8/3}$-type curve (which can be described in terms of a pair of GFF flow lines in the sense of \cite{ig1,ig4}). This curve is a.s.\ simple, so for $a,b\in\BB R$ with $a<b$, the intersection $\bigcap_{\rho > 0} B_\rho(\xi_\glu^\infty([a,b]) ; d_\glu^\infty)$ contains no points of $\xi_\glu^\infty(\BB R\setminus[a,b])$. From this and the above described local GHPU convergence, we infer the following.

\begin{lem} \label{lem-bdy-dist-compare}
For each $a,b\in\BB R$ with $a<b$, each $\alpha_0 > 0$, and each $\ep \in (0,1)$, there exists $\el_* = \el_*(a,b,\alpha_0,\ep ) \in\BB N$ and $\alpha \in (0,\alpha_0]$ such that for $\el\geq \el_*$,
\eqbn
\BB P\left[ B_{\alpha \el^{1/2}} \left( \xi_\glu^\infty([a\el , b\el]_{\BB Z}) ; Q_\glu^\infty \right) 
\subset B_{\alpha_0 \el^{1/2}} \left( \xi_-^\infty([a\el , b\el]_{\BB Z}) ; Q_-^\infty \right) \cup 
B_{\alpha_0 \el^{1/2}} \left( \xi_+^\infty([a\el , b\el]_{\BB Z}) ; Q_+^\infty \right) \right] 
\geq 1-\ep . 
\eqen
\end{lem}

The following local absolute continuity statement will be used to transfer the above lemmas to finite-volume statements.

\begin{lem} \label{lem-bdy-segment-rn}
For $\el\in\BB N$, let $(Q^\el , \BB e^\el)$ be a free Boltzmann quadrangulation with simple boundary of perimeter $2\el$ and let $\beta^\el$ be its boundary path with $\beta^\el(0) = \BB e^\el$.  
For each $   \delta  ,\ep \in (0,1)$, there exists $\alpha = \alpha(  \delta , \ep) > 0$ and $\el_* = \el_*(\delta , \ep) \in \BB N$ such that the following is true for each $\el\geq \el_*$. On an event of probability at least $1-\ep$ (with respect to the law of $(Q^\el , \BB e^\el)$), the law of the curve-decorated graph 
$\left( B_{\alpha \el^{1/2} }(\beta^\el\left([\delta \el , (1-\delta)\el ]_{\BB Z}) ; Q^\el \right) , \beta^\el|_{[\delta \el , (2-\delta)\el ]_{\BB Z} } \right) $
is absolutely continuous with respect to the law of its UIHPQ$_{\op{S}}$ analog $\left( B_{\alpha \el^{1/2} }(\beta^\infty\left([\delta \el , (2-\delta)\el ]_{\BB Z}) ; Q^\infty \right) , \beta^\infty|_{[\delta \el , (1-\delta)\el ]_{\BB Z}} \right)$,
with Radon-Nikodym derivative bounded above by a universal constant times $ \delta^{-5/2}     $. 
\end{lem}
\begin{proof}
By Theorem~\ref{thm-fb-ghpu} and since the boundary path of the Brownian disk has no self-intersections, there exists $\alpha = \alpha(\delta,\ep) > 0$ and $\el_* = \el_*(\delta,\ep) \in \BB N$ such that for $\el\geq \el_*$,  
\eqbn
\BB P\left[ E_0^\el \right] \geq 1-\ep \quad \op{where} \quad 
 E_0^\el := \left\{ \op{dist}\left( \beta^\el\left([\delta \el , (2-\delta)\el ]_{\BB Z} \right) , \beta^\el \left( [0, \tfrac{\delta}{2} \el]_{\BB Z} , [(2-\tfrac{\delta}{2} )\el , 2\el]_{\BB Z} \right)          ;Q^\el     \right) \geq \alpha \el^{1/2} + 2 \right\} .
\eqen
By Lemma~\ref{lem-peel-ball}, on the event $ E_0^\el$ the radius-$\alpha \el^{1/2}$ peeling-by-layers cluster of $Q^\el$ started from $\beta^\el\left([\delta \el , (2-\delta)\el ]_{\BB Z}\right)$ and targeted at $\BB e^\el = \beta^\el(0)$ contains no edge of $\beta^\el \left( [0, \tfrac{\delta}{2} \el]_{\BB Z} , [(2-\tfrac{\delta}{2} )\el , 2\el]_{\BB Z} \right) $ so the peeling-by-layers clusters reach radius $\alpha \el^{1/2}$ before hitting $\BB e^\el$ and its net boundary length process at time $J_{\alpha\el^{1/2}}^\el$ (Definition~\ref{def-bdy-process}) satisfies $W^\el_{J_{\alpha \el^{1/2}}^\el} \geq (2-\delta)\el$. The statement of the lemma follows by combining this with Lemma~\ref{lem-bubble-cond}. 
\end{proof}

\begin{proof}[Proof of Lemma~\ref{lem-geo-approx}]
Define  
\eqbn
I_1^\el := \left[ 0, \tfrac43 \el \right]_{\BB Z} \quad 
I_2^\el := \left[ \tfrac23 \el , 2\el \right]_{\BB Z} \quad \op{and}\quad
I_3^\el := \left[ 0 ,\tfrac13 \el \right]_{\BB Z} \cup \left[\tfrac53 \el , 2\el \right]_{\BB Z} .
\eqen
Then $\beta_\pm^\el(I_i^\el)$ for $i\in \{1,2,3\}$ are connected, overlapping arcs of $\bdy Q_\pm^\el$ and any two edges of $\bdy Q_\pm^\el$ are contained in one of these three arcs. We will prove the lemma by applying Lemma~\ref{lem-bdy-segment-rn} on each of the arcs $I_i^\el$ for $i\in \{1,2,3\}$ to transfer the estimate of Lemma~\ref{lem-infty-geo-approx} from $Q^\infty$ to $Q^\el$. 

By Lemma~\ref{lem-bdy-segment-rn} (applied with $\delta=1/4$ and $\ep/6$ in place of $\ep$) and the invariance of the law of $Q_\pm^\el$ under re-rooting along the boundary, there exists $\alpha_0 = \alpha_0(  \ep) > 0$ and $\el_0 = \el_0( \ep) \in \BB N$ such that the following is true. For each $i\in \{1,2,3\}$, there is an event $E_i^\el = E_i^\el(\ep ,\alpha_0)$ such that for $\el\geq\el_0$, we have $\BB P[E_i^\el]\geq 1-\ep/6$ and on $E_i^\el$, the law of the curve-decorated graph
$\left( B_{\alpha_0 \el^{1/2} }\left(\beta_\pm^\el(I_i^\el) ; Q^\el \right) , \beta_\pm^\el|_{I_i^\el } \right) $
is absolutely continuous with respect to the law of $\left( B_{\alpha_0 \el^{1/2} }\left(\beta_\pm^\infty\left( I_1^\el \right) ; Q^\infty\right) , \beta_\pm^\infty|_{I_1^\el} \right)$,
with Radon-Nikodym derivative bounded above by a universal constant.

By combining the preceding Radon-Nikodym derivative estimate with Lemma~\ref{lem-bdy-dist-compare}, we find that there exists $\alpha = \alpha(\ep) \in (0,\alpha_0/2]$ and $\el_1 = \el_1(\ep) \geq \el_0$ such that for $\el\geq \el_1$, we have $\BB P[E^\el] \geq 1-\ep $ where 
\eqb \label{eqn-quotient-arc}
E^\el :=  \bigcap_{i=1}^3 E_i^\el \cap \left\{  , B_{2\alpha \el^{1/2}} \left( \xi_\glu^\el(I_i^\el) ; Q_\glu^\el \right) 
\subset B_{\alpha_0 \el^{1/2}} \left( \xi_-^\el(I_i^\el) ; Q_-^\el \right) \cup 
B_{\alpha_0 \el^{1/2}} \left( \xi_+^\el(I_i^\el) ; Q_+^\el \right)   \right\} .
\eqe 
Note that on $E^\el$, each path in $Q_\glu^\el$ between points of $\xi_\glu^\el(I_i^\el)$ with length at most $2\alpha\el^{1/2}$ (e.g., the path $\wt\gamma$ in the definition of the event $F_{N,\zeta}^\el(k_0,k_1)$ just above the lemma statement for $\zeta \in (0,\alpha]$ and $k_0,k_1 \in I_i^\el$ if $\op{dist}(\beta_\glu^\el(k_0) , \beta_\glu^\el(k_1) ; Q_\glu^\el) \leq \alpha \el^{1/2}$) must stay in $B_{\alpha_0 \el^{1/2}} \left( \xi_-^\el(I_i^\el) ; Q_-^\el \right) \cup B_{\alpha_0 \el^{1/2}} \left( \xi_+^\el(I_i^\el) ; Q_+^\el \right) $. In particular, for $k_0,k_1 \in I_i^\el$ the occurrence of the event $\left\{ \op{dist}(\beta_\glu^\el(k_0) , \beta_\glu^\el(k_1) ; Q_\glu^\el) \leq \alpha \el^{1/2} \right\} \cap F_{N,\zeta}^\el(k_0,k_1)^c$ is determined by what happens inside of $ B_{\alpha_0 \el^{1/2}} \left( \xi_-^\el(I_i^\el) ; Q_-^\el \right) \cup 
B_{\alpha_0 \el^{1/2}} \left( \xi_+^\el(I_i^\el) ; Q_+^\el \right) $. 
 
By our Radon-Nikodym derivative estimate when we restrict to $E_i^\el$, the preceding paragraph, and our bound for the probability of the infinite-volume analogue of $F_{N,\zeta}^\el(k_0,k_1)^c$ from Lemma~\ref{lem-infty-geo-approx}, for each $\zeta\in(0,\alpha]$, there exists $\el_* = \el_*( \ep ,\zeta ) \in\BB N$ and $N = N( \ep,\zeta) \in\BB N$ such that for $\el\geq \el_*$, each $i\in \{1,2,3\}$, and each $k_0,k_1\in I_i^\el$,  
\eqbn
\BB P\left[ E^\el \cap  \left\{ \op{dist}(\beta_\glu^\el(k_0) , \beta_\glu^\el(k_1) ; Q_\glu^\el) \leq \alpha \el^{1/2} \right\} \cap F_{N,\zeta}^\el(k_0,k_1)^c \right]  \leq \zeta  .
\eqen  
Since any two integers $k_0,k_1 \in [0,2\el]_{\BB Z}$ are both contained in one of $I_1^\el$, $I_2^\el$, or $I_3^\el$, we obtain the statement of the lemma.
\end{proof}

\begin{proof}[Proof of Theorem~\ref{thm-saw-conv}]
For $\el\in\BB N$, let $d_\pm^\el$ be the graph metric on $Q_\pm^\el$ rescaled by $(2\el )^{-1/2}$, let $\mu_\pm^\el$ be the measure on $Q_\pm^\el$ which assigns to each vertex a mass equal to $18^{-1} \el^{-2}$ times its degree, and let $\xi_\pm^\el(s) := \beta_\pm^\el(2\el s)$ for $s\in [0,1]_{\BB Z}$. Define the one-sided curve-decorated metric measure spaces $\frk Q_\pm^\el := (Q_\pm^\el , d_\pm^\el , \mu_\pm^\el , \xi_\pm^\el)$. 

Since the restriction of the graph metric on $Q_\glu^\el$ to each of $Q_\pm^\el$ is bounded above by the graph metric on $Q_\pm^\el$, we easily deduce GHPU tightness of $\frk Q_\glu^\infty$ from GHPU tightness of $\frk Q_\pm^\el$ (Theorem~\ref{thm-fb-ghpu}) and the GHPU compactness criterion~\cite[Lemma~2.6]{gwynne-miller-uihpq}.  Hence for any sequence of positive integers tending to $\infty$, there exists a subsequence $\mcl L$ and a coupling of a random curve-decorated metric measure space $\wt{\frk H} = (\wt H ,\wt d , \wt\mu,\wt\xi)$ with two independent free Boltzmann Brownian disks $\frk H_\pm = (H_\pm , d_\pm , \mu_\pm , \xi_\pm)$ with unit boundary length such that 
\eqbn
(\frk Q_\glu^\el , \frk Q_-^\el , \frk Q_+^\el) \rta (\wt{\frk H} , \frk H_-,\frk H_+) 
\eqen
in law with respect to the GHPU topology on each coordinate as $\mcl L\ni\el\rta\infty$. By the Skorokhod representation theorem, we can couple so that this convergence occurs a.s. 

Let $\frk H_\glu = (H_\glu ,d_\glu,\mu_\glu , \xi_\glu)$ be the curve-decorated metric measure space obtained by metrically gluing $\frk H_-$ and $\frk H_+$ together along their boundary paths as in the theorem statement.  By elementary limiting arguments directly analogous to those in~\cite[Section~7.3]{gwynne-miller-saw} (but somewhat simpler, since we are working with compact spaces so there is no need to ``localize") and the universal property of the quotient metric, we infer that there a.s.\ exists a surjective 1-Lipschitz map $f_\glu : H_\glu  \rta \wt H$ such that $(f_\glu)_* \mu_\glu = \wt\mu$, $f_\glu\circ \xi_\glu = \wt\xi$, and the restrictions $f_\glu|_{H_\pm \setminus \bdy H_\pm}  $ are isometries from $(H_\pm \setminus \bdy H_\pm, d_\pm)$ to $f(H_\pm\setminus \bdy H_\pm)$, equipped with the internal metric induced by $\wt d$. We need to show that $f_\glu$ is itself an isometry. We will accomplish this by taking a limit of the estimate of Lemma~\ref{lem-geo-approx}. 
 
To this end, fix $\ep \in (0,1)$, let $\alpha = \alpha(\ep)$ be as in Lemma~\ref{lem-geo-approx}, and for $\el\in\mcl L$ let $E^\el = E^\el(\ep,\alpha)$ be the event of that lemma. Also let $\{ (t_0^k , t_1^k)\}_{k\in\BB N }$ be an enumeration of the pairs of rational times $(t_0,t_1) \in [0,1]$ and (using Lemma~\ref{lem-geo-approx} with $\zeta = \ep 2^{-k}$) choose for each $k\in\BB N$ a $N^k\in\BB N$ and an $\el_*^k \in \BB N$ such that for $\el\geq \el_*$, 
\eqbn
\BB P\left[E^\el \cap \left\{  \op{dist}(\beta_\glu^\el(\lfloor 2\el t_0^k \rfloor ) , \beta_\glu^\el(\lfloor 2\el t_1^k \rfloor ) ; Q_\glu^\el) \leq \alpha \el^{1/2} \right\} \cap F^\el_{N,\zeta}(\lfloor 2\el t_0^k \rfloor , \lfloor 2\el t_1^k \rfloor)^c \right] \leq \ep 2^{-k} .
\eqen
For $\el \in \BB N$, let $\wt E^\el$ be the event that $E^\el$ occurs and $F^\el_{N,\zeta}(\lfloor 2\el t_0^k \rfloor , \lfloor 2\el t_1^k \rfloor)$ occurs for each $k\in\BB N$ with $\el \geq \el_*^k$ and $\op{dist}(\beta_\glu^\el(\lfloor 2\el t_0^k \rfloor ) , \beta_\glu^\el(\lfloor 2\el t_1^k \rfloor ) ; Q_\glu^\el) \leq \alpha \el^{1/2}$, so that $\BB P[\wt E^\el ] \geq 1-2\ep$.   

Let $\wt E $ be the event that $\wt E^\el$ occurs for infinitely many $\el\in\mcl L$, so that $\BB P[\wt E] \geq 1- 2\ep$. Passing to the limit in the definition of $\wt E^\el$ shows that on $\wt E$, it is a.s.\ the case that for each pair of rational times $(t_0^k,t_1^k)$ for $k\in\BB N$ such that with $\wt d(\wt\xi(t_0^k)  , \wt\xi(t_1^k)) \leq (2/3)^{1/2} \alpha$, there exists points $\wt\xi(t_0^k ) =  z_0 , z_1 , \dots , z_{N^k} = \wt\xi(t_1^k) \in \wt H$ and signs $\chi_0 ,\dots , \chi_{N^k} \in \{-,+\}$ with 
\eqbn
  \sum_{ j=1   }^{N^k} d_{\chi_j}(z_{j-1} ,z_j)     \leq \wt d(\wt\xi(t_0^k)  , \wt\xi(t_1^k))    + \ep 2^{-k} .
\eqen
Hence $d_\glu(\xi_\glu(t_0^k)  ,\xi_\glu(t_1^k)) \leq \wt d(\wt\xi(t_0^k)  , \wt\xi(t_1^k)) + \ep 2^{-k}$. 

Since $\{(t_0^k , t_1^k) : k \geq k_*\}$ is dense in $[0,1]^2$ for each $k_* \in\BB N$, $f_\glu$ is a curve-preserving isometry, and $\xi_\glu$ and $\wt\xi$ are continuous, we infer that on $\wt E$,  
\eqb \label{eqn-iso-short-inc}
d_\glu(\xi_\glu(t_0 )  ,\xi_\glu(t_1)) = \wt d(\wt\xi(t_0 )  , \wt\xi(t_1 )) ,\quad \forall t_0,t_1 \in [0,1]  \:\op{with} \:  \wt d(\wt\xi(t_0 )  , \wt\xi(t_1)) \leq (2/3)^{1/2} \alpha .
\eqe 
By breaking up a $\wt d$-geodesic between two arbitrary points in $\wt H$ into segments of time length (hence diameter) at most $\alpha$ and recalling that $f_\glu$ is an isometry for the internal metrics on the two sides of the curves $\xi_\glu$ and $\wt\xi$, we see that~\eqref{eqn-iso-short-inc} implies that $f_\glu$ is an isometry. Therefore $\frk H_\glu = \wt{\frk H}$ as curve-decorated metric measure spaces, so since our initial choice of subsequence was arbitrary we obtain the theorem statement. 
\end{proof}

\bibliography{cibiblong,cibib,perc-ref}
\bibliographystyle{hmralphaabbrv}

\end{document}